\documentclass[12pt,leqno]{amsart}

\usepackage{pifont}
\usepackage{stmaryrd}
\usepackage{dsfont}
\usepackage{color}
\usepackage{mathrsfs}
\usepackage{amsmath}
\usepackage{amssymb}
\usepackage[hidelinks]{hyperref}
\usepackage{bookmark}
\usepackage[bottom]{footmisc}
\usepackage{verbatim}
\usepackage{extarrows}
\usepackage{mathtools}
\usepackage{orcidlink}

\usepackage{microtype}
\usepackage{graphicx}
\usepackage{enumitem}
\usepackage{xcolor}

\allowdisplaybreaks
\numberwithin{equation}{section}

\newtheorem{thm}{Theorem}[section]
\newtheorem{lem}[thm]{Lemma}

\newtheorem{prop}[thm]{Proposition}
\newtheorem{cor}[thm]{Corollary}
\newtheorem{defi}[thm]{Definition}
\newtheorem{rmk}[thm]{Remark}

\newtheorem*{prelem}{Lemma}

\newcommand{\Vir}{\operatorname{Vir}}
\newcommand{\wt}{\operatorname{wt}}
\newcommand{\Hom}{\operatorname{Hom}}
\newcommand{\Span}{\operatorname{span}}
\newcommand{\Aut}{\operatorname{Aut}}

\newcommand{\id}{\operatorname{id}}

\def\Z{\mathbb{Z}}

\def\C{\mathbb{C}}

\def\Q{\mathbb{Q}}

\DeclareMathOperator{\Sym}{Sym}

\title[The Icosahedral Orbifold $V_{L_2}^{A_5}$]{
$C_2$-Cofiniteness and Rationality of the Icosahedral Orbifold
$V_{L_2}^{A_5}$}

\author{Shun Xu~\orcidlink{0009-0006-8080-8107}}

\address{School of Mathematical Sciences, Anhui University,
Hefei, Anhui 230601, China}

\email{shunxu@ahu.edu.cn}



\begin{document}

\begin{abstract}
Let $L_2=\Z\alpha$ be the rank-one root lattice with
$(\alpha,\alpha)=2$, and let $A_5$ act on the lattice vertex operator
algebra $V_{L_2}$ through an icosahedral subgroup of
$\Aut(V_{L_2})\cong PSL_2(\C)$. We prove that the fixed-point vertex
operator algebra $V_{L_2}^{A_5}$ is  strongly
rational. 
\end{abstract}

\maketitle

\section{Introduction}
\label{sec:introduction}

Orbifold theory asks how finiteness and semisimplicity properties of a vertex operator algebra behave under passage to fixed points. If $V$ is a vertex operator algebra and $G\leq \Aut(V)$ is finite, the basic object is $V^G=\{v\in V\mid gv=v,\ \forall g\in G\}$. For cyclic groups, and more generally for finite solvable groups under standard hypotheses, $C_2$-cofiniteness and regularity are controlled by the work of Miyamoto and Carnahan--Miyamoto \cite{Miyamoto2015,CarnahanMiyamoto}. For non-solvable groups there is no analogous reduction through a normal series. The smallest nonabelian simple group $A_5$ is therefore a natural first test case for orbifold finiteness beyond the solvable regime.

The significance of the $A_5$ case extends beyond the general orbifold problem; it is also an exceptional case in the classification program for rational vertex operator algebras of central charge $1$. The unitary rational conformal field theories at $c=1$ were classified in physics literature by Ginsparg and Kiritsis \cite{Ginsparg,Kiritsis}. In the VOA setting, this leads to rank-one lattice VOAs, their $\mathbb Z_2$-fixed-point subalgebras, and the exceptional polyhedral orbifolds $V_{L_2}^{A_4}$, $V_{L_2}^{S_4}$, and $V_{L_2}^{A_5}$. After characterizing $V_{L_2}^{A_4}$, Dong and Jiang reduced the remaining exceptional part of the $c=1$ classification to $V_{L_2}^{S_4}$ and $V_{L_2}^{A_5}$ \cite{DongJiangCharacterizationA4}. Thus these fixed-point algebras form the exceptional end of the $c=1$ classification picture.

The rank-one lattice VOA attached to the $A_1$ root lattice contains this exceptional part in a rigid form. Let $L_2=\Z\alpha$ with $(\alpha,\alpha)=2$ and $V=V_{L_2}$. Its connected automorphism group is $PSL_2(\C)$, whose finite subgroups yield the classical rank-one orbifolds. Dong, Griess, and Ryba determined the structure of the exceptional $E$-series fixed-point algebras and obtained the decomposition
\begin{equation}\label{eq:intro-DGR-decomposition}
V_{L_2}\cong \bigoplus_{m\ge0}W_{2m+1}\otimes L(1,m^2),
\end{equation}
where $W_{2m+1}$ is the irreducible $SL_2(\C)$-module of dimension $2m+1$ \cite{DGR}. They identified $A_5$-invariant Virasoro highest-weight vectors with the binary icosahedral invariant ring and showed that $V_{L_2}^{A_5}$ is generated by $\omega$ together with primary vectors of conformal weights $36$, $100$, and $225$.

From the viewpoint of the $c=1$ classification, the central difficulty is that this explicit description does not place $V_{L_2}^{A_5}$ inside the strongly rational category. Characterization arguments require finiteness and semisimplicity; indeed, Dong and Jiang used rationality and $C_2$-cofiniteness as structural input for $V_{L_2}^{A_4}$ and noted they are needed for the remaining cases \cite{DongJiangCharacterizationA4}. While Wu and Zhang proved these properties for the $S_4$ orbifold, the corresponding finiteness statement for $V_{L_2}^{A_5}$ remained unavailable \cite{WuZhang}. Establishing these properties removes a basic obstruction in the icosahedral branch of the $c=1$ program.

This difficulty can be formulated intrinsically. Finite strong generation controls normally ordered products, whereas $C_2$-cofiniteness requires the finite-dimensionality of the commutative Poisson algebra $R(V)=V/C_2(V)$, where $C_2(V)=\Span\{u_{-2}v\mid u,v\in V\}$. Since the first non-Virasoro generator lies in weight $36$, direct search for relations is unwieldy, and solvable-orbifold induction fails as $A_5$ is simple. Thus the general non-solvable orbifold problem and the exceptional $c=1$ classification meet in the question: $\dim_{\C}R(V_{L_2}^{A_5})<\infty$.

Representation-theoretically, Wu and Zhang constructed $37$ pairwise inequivalent irreducible $V_{L_2}^{A_5}$-modules, organized into untwisted and order-$2$, order-$3$, and order-$5$ twisted sectors \cite{WuZhang}. Their classification was conditional on the rationality and $C_2$-cofiniteness of $V_{L_2}^{A_5}$. Theorem~\ref{thm:main} and Corollary~\ref{cor:strong-rational} supply these missing finiteness and semisimplicity properties. Moreover, since $V_{L_2}^{A_5}$ is self-contragredient, the orbifold completeness theorem of Dong--Ren--Xu, together with their observation that its proof does not require the twisted-module positivity hypothesis in this self-dual setting, applies \cite{DongRenXu}. Combining this completeness statement with the twisted-sector enumeration of Wu and Zhang gives an unconditional classification. In particular, $V_{L_2}^{A_5}$ has exactly $37$ pairwise inequivalent irreducible modules; see Theorem~\ref{thm:irreducible-classification}.

Our main result settles this problem, placing $V_{L_2}^{A_5}$ unconditionally among strongly rational $c=1$ VOAs. The icosahedral candidate in the $c=1$ program is no longer conditional on orbifold rationality; what remains is an intrinsic characterization problem. The structural information obtained here, especially the reduction of the $C_2$-Poisson algebra, provides a major obstruction to such a characterization theorem.

The proof has two parts. First, representation-theoretically, the decomposition \eqref{eq:intro-DGR-decomposition} separates an $SL_2$ multiplicity space from a $c=1$ Virasoro module. We prove that a nonzero Clebsch--Gordan projection forces the uniquely determined lowest mode landing in a prescribed Virasoro highest-weight line to be nonzero. This channel-separation principle turns invariant-theoretic multiplication into intrinsic $C_2$-relations. Applying the top Clebsch--Gordan channel yields a triangular reduction. The binary icosahedral invariant ring gives $R(U)=\C[a]\{1,x,y,z\}$ for $U=V_{L_2}^{A_5}$, where $a=[\omega]$ and $x,y,z$ are classes of basic invariant primaries. Two further projections eliminate $z$ and $y$, leaving the two-generator reduction
\begin{equation}\label{eq:intro-rank-two-reduction}
R(U)=\C[a]\cdot1+\C[a]\cdot x,\qquad \wt x=36.
\end{equation}
Thus the problem reduces to the nilpotency of $a$.

Second, Virasoro-theoretically, since the next invariant primary after weight $36$ occurs at weight $100$, the relations $X_{-3}X, X_{-5}X\in C_2(U)$ receive contributions modulo $C_2(U)$ only from vacuum and self-$L(1,36)$ channels. This gives a $2\times2$ linear system for $a^{38}$ and $a^{20}x$. To prove the determinant is nonzero without large PBW computations, we introduce a Virasoro $C_2$-symbol defined by \begin{equation}\label{eq:intro-symbol}
\sigma(L(-2))=T,
\qquad
\sigma(L(-n))=0\quad(n\ne2),
\end{equation} extracting the part visible in the $C_2$-quotient. As $L(1,36)$ is degenerate with a singular vector at level $13$, the Benoit--Saint-Aubin formula turns the null-vector equation into a scalar recurrence of width six via \eqref{eq:intro-symbol}.

Since $c=1$ is non-generic, we first construct global first-row forms over $\Lambda=\C[t,t^{-1}]$ and then localize at $t=1$ to $\mathcal R=\Lambda_{(t-1)}$. Highest-vector coefficients are reconstructed over $\mathcal R$ from positive-mode Ward identities. Injectivity on the special fibers $L(1,0)$ and $L(1,36)$ gives local uniqueness, while generic first-row intertwining operators \cite{KoshidaKytola} verify the remaining identities by evaluation of finitely many rational PBW coordinates on a Zariski-dense set. A computer-assisted exact calculation over $\Q$ yields a nonzero determinant, implying $a^{38}=0$, $a^{20}x=0$, and consequently $\dim_{\C}R(U)\le58$. This proves $C_2$-cofiniteness, and McRae's theorem then implies strong rationality.

\begin{thm}\label{thm:main}
Let $L_2=\Z\alpha$ with $(\alpha,\alpha)=2$, and let $A_5\leq\Aut(V_{L_2})$ be an icosahedral subgroup. Then $V_{L_2}^{A_5}$ is $C_2$-cofinite. Specifically, if $a=[\omega]\in R(V_{L_2}^{A_5})$, then $a^{38}=0$ and $\dim_{\C}R(V_{L_2}^{A_5})\le58$.
\end{thm}

\begin{cor}\label{cor:strong-rational}
The vertex operator algebra $V_{L_2}^{A_5}$ is strongly rational.
\end{cor}

Theorem~\ref{thm:main} closes the finiteness side of the exceptional $A_5$ case, removing the rationality hypothesis and reducing its role in the $c=1$ classification to an intrinsic characterization question.

The paper is organized as follows. Section~\ref{sec:preliminaries} records structural input from rank-one lattice VOAs, binary icosahedral invariant theory, and degenerate $c=1$ Virasoro theory. Section~\ref{sec:channel-separation} establishes channel separation and lowest-mode nonvanishing. Sections~\ref{sec:triangular-reduction}--\ref{sec:eliminate-yz} reduce the $C_2$-Poisson algebra to \eqref{eq:intro-rank-two-reduction}. Section~\ref{sec:symbol-determinant} introduces the Virasoro $C_2$-symbol and determinant criterion. Section~\ref{sec:local-regularization} constructs global first-row forms, localizes them at $t=1$, and proves positive-mode reconstruction there. Section~\ref{sec:BSA-recurrence} derives the scalar Benoit--Saint-Aubin recurrence. Section~\ref{sec:exact-completion} evaluates the determinant, completes the proof, and records consequences. The appendix gives the permanent download link for the exact-verification program.

\section{Preliminaries}
\label{sec:preliminaries}

\subsection{Zhu's Poisson algebra}

Throughout the paper, all vector spaces are over $\C$. Let $V=(V,Y,\mathbf 1,\omega)$ be a vertex operator algebra. We use the mode convention $Y(u,z)=\sum_{n\in\Z}u_n z^{-n-1}$ for $u\in V$, where $\mathbf 1$ is the vacuum and $\omega$ is the conformal vector. The Virasoro operators are defined by $Y(\omega,z)=\sum_{n\in\Z}L(n)z^{-n-2}$, so that $L(n)=\omega_{n+1}$. For homogeneous $u\in V$, we denote its conformal weight by $\wt u$, and write $\Span_{\C}S$ for the complex linear span of $S$.

We call a vertex operator algebra \emph{strongly rational} if it is simple, of CFT type, self-contragredient, rational, and $C_2$-cofinite. 

Following Zhu \cite{Zhu}, define $C_2(V)=\Span_{\C}\{u_{-2}v\mid u,v\in V\}$. The VOA $V$ is called \emph{$C_2$-cofinite} if $\dim_{\C}V/C_2(V)<\infty$. We set $R(V):=V/C_2(V)$ and denote the image of $u\in V$ in the quotient by $[u]=u+C_2(V)$.

The quotient $R(V)$, known as Zhu's $C_2$-algebra or Poisson algebra, carries a natural commutative  Poisson structure \cite{Zhu,LiAbelianizing}. Explicitly, the product and bracket are given by $[u]\cdot[v]=[u_{-1}v]$ and $\{[u],[v]\}=[u_0v]$ for $u,v\in V$. Henceforth, $[u][v]$ always denotes the product $[u]\cdot[v]$ in $R(V)$.

Since $C_2(V)$ respects the conformal-weight grading, $R(V)$ inherits a $\Z$-grading
\[
R(V)=\bigoplus_{n\in\Z}R(V)_n,
\qquad
R(V)_n=(V_n+C_2(V))/C_2(V).
\]
If $V$ is nonnegatively graded, in particular if it is of CFT type, then $R(V)_n=0$ for $n<0$. In particular, if $u$ is homogeneous, then $[u]$ is homogeneous of the same weight whenever nonzero.

We repeatedly use the following consequence of the $L(-1)$-derivative property.

\begin{lem}\label{lem:deepmodes}
For any $u,v\in V$ and integer $m\ge2$, $u_{-m}v\in C_2(V)$.
\end{lem}

\begin{proof}
The $L(-1)$-derivative property $Y(L(-1)u,z)=\frac{d}{dz}Y(u,z)$ implies $(L(-1)u)_n=-n u_{n-1}$. Iterating gives
\[
(L(-1)^{m-2}u)_{-2}=(m-1)!u_{-m}.
\]
Applying both sides to $v$ yields
\[
(L(-1)^{m-2}u)_{-2}v=(m-1)!u_{-m}v.
\]
The left-hand side lies in $C_2(V)$ by definition. Since $(m-1)!\neq0$ over $\C$, it follows that $u_{-m}v\in C_2(V)$.
\end{proof}

Applying Lemma~\ref{lem:deepmodes} to $\omega$ yields $L(-n)u=\omega_{-n+1}u\in C_2(V)$ for $n\ge3$. On the other hand, $L(-2)u=\omega_{-1}u$, so by the product definition in $R(V)$ we have $[L(-2)u]=[\omega]\cdot[u]$. Since the class of the conformal vector appears frequently, we set $a:=[\omega]\in R(V)$, whence $[L(-2)u]=a[u]$.

\subsection{The rank-one lattice VOA and the binary icosahedral group}
\label{subsec:rank-one}

Let $L_2=\Z\alpha$ with $(\alpha,\alpha)=2$ be the positive-definite even root lattice of type $A_1$, and let $V_{L_2}=M(1)\otimes\C\{L_2\}$ be the associated lattice VOA of central charge $1$; see \cite{FLM,DongLattice}. The weight-one subspace $(V_{L_2})_1\cong\mathfrak{sl}_2(\C)$, and the connected automorphism group is $PSL_2(\C)$ \cite{DongGriess,DGR}. We use the $SL_2(\C)$-module structure obtained from the double covering $\pi:SL_2(\C)\to PSL_2(\C)$. Only odd-dimensional irreducible $SL_2(\C)$-modules occur, so $-1\in SL_2(\C)$ acts trivially on all multiplicity spaces.

For $c,h\in\C$, let $L(c,h)$ denote the irreducible Virasoro module with central charge $c$ and highest weight $h$, and let $W_n$ denote the $n$-dimensional irreducible $SL_2(\C)$-module. Dong--Griess--Ryba recall the decomposition
\begin{equation}\label{eq:DGRdecomp}
V_{L_2}\cong\bigoplus_{m\ge0}W_{2m+1}\otimes L(1,m^2)
\end{equation}
as a module for the commuting actions of $SL_2(\C)$ and the Virasoro VOA $L(1,0)$ \cite{DongGriess,DGR}. Here $W_{2m+1}$ is identified with the space of Virasoro highest-weight vectors of weight $m^2$ in the corresponding summand \cite[Section~3]{DGR}. Set $E_m:=W_{2m+1}$ for $m\ge0$. The Clebsch--Gordan decomposition gives
\begin{equation}\label{eq:CG}
E_s\otimes E_t\cong\bigoplus_{r=|s-t|}^{s+t}E_r,
\end{equation}
each summand with multiplicity one. For $|s-t|\le r\le s+t$, fix a nonzero $SL_2(\C)$-homomorphism $p_{s,t}^{\,r}:E_s\otimes E_t\to E_r$, unique up to scalar. For the top summand $r=s+t$, we make a coherent normalization as follows. Under the standard realization
\[
E_m\cong \Sym^{2m}(\C^2),
\]
we define
\[
\mu_{s,t}:=p_{s,t}^{\,s+t}:E_s\otimes E_t\longrightarrow E_{s+t}
\]
to be ordinary multiplication of homogeneous binary forms. With this normalization, the maps $\mu_{s,t}$ are associative and commutative and give the usual Cartan product on $\bigoplus_{m\ge0}E_m$. For $r<s+t$, the maps $p_{s,t}^{\,r}$ remain fixed only up to an arbitrary nonzero scalar normalization.

Let $G\cong A_5$ be an icosahedral subgroup of $\Aut(V_{L_2})\cong PSL_2(\C)$, and let $I:=\pi^{-1}(G)\subset SL_2(\C)$ be its full inverse image under the covering $\pi$. Then $I\cong 2.A_5\cong SL_2(5)$ is the binary icosahedral group \cite[Section~2]{DGR}. Since $E_m$ has odd dimension, $-1\in I$ acts trivially, so $E_m^I=E_m^G$. Define
\begin{equation}\label{eq:invariant-primary-space}
P_m:=E_m^I.
\end{equation}
Taking $G$-fixed points in \eqref{eq:DGRdecomp} gives
\begin{equation}\label{eq:Udecomp}
U:=V_{L_2}^{A_5}=V_{L_2}^{G}\cong\bigoplus_{m\ge0}P_m\otimes L(1,m^2),
\end{equation}
where $P_m$ is identified with Virasoro primary vectors of weight $m^2$ in $U$.

We recall the invariant-theoretic description of these spaces. Let $\mathcal{H}\subset V_{L_2}$ denote the space of Virasoro highest-weight vectors. DGR identify $\mathcal{H}$ with the even-degree part $\C[x_1,x_2]^+$ of the polynomial algebra in two variables, with $W_{2m+1}$ corresponding to homogeneous binary forms of degree $2m$ \cite[before Proposition~3.2]{DGR}. Thus $\bigoplus_{m\ge0}P_m$ is identified with the invariant ring $\C[x_1,x_2]^I$. To distinguish polynomial degree from the index $m$ in \eqref{eq:DGRdecomp}, we use the half-degree grading $\deg_{1/2}f=\tfrac{1}{2}\deg_{\mathrm{poly}}f$ and set $\mathcal{P}:=\bigoplus_{m\ge0}P_m$. As an algebra, $\mathcal{P}$ is equipped with the Cartan product $P_s\otimes P_t\xrightarrow{\mu_{s,t}}P_{s+t}$, corresponding to ordinary polynomial multiplication; this should not be confused with the $(-1)$-product of the VOA.

Classical invariant theory gives
\begin{equation}\label{eq:icosaring}
\mathcal{P}\cong\C[A,B,C]/(A^5+B^3+C^2),\qquad \deg A=6,\ \deg B=10,\ \deg C=15.
\end{equation}
More precisely, DGR show that \(\C[x_1,x_2]^I\) is generated by homogeneous invariants \(h_{12},h_{20},h_{30}\), of ordinary polynomial degrees \(12,20,30\), subject to the relation
$$
h_{12}^5+h_{20}^3+h_{30}^2=0
$$
\cite[Proposition~3.2(iii)]{DGR}. Thus, after setting \(A=h_{12}\), \(B=h_{20}\), and \(C=h_{30}\) and passing to the half-degree grading, we obtain \eqref{eq:icosaring}.
 Choose nonzero invariant primaries $X\in P_6$, $Y\in P_{10}$, $Z\in P_{15}$ corresponding to $h_{12},h_{20},h_{30}$, so that $\wt X=36$, $\wt Y=100$, $\wt Z=225$. DGR prove that $U$ is generated by these together with $\omega$:
\begin{equation}\label{eq:DGRgenerators}
U=\langle\omega,X,Y,Z\rangle_{\mathrm{VOA}},
\end{equation}
the $A_5$ case of \cite[Proposition~3.4]{DGR}.

We also need one finite-dimensional projection computed by DGR. In their notation, \(p_k\) denotes the projection of a finite-dimensional \(SL_2(\C)\)-module onto its \(W_k\)-isotypic component, and they prove
$$
p_{31}\bigl(W_{13}^I\otimes W_{21}^I\bigr)\neq0
$$
\cite[Theorem~2.1(ii)]{DGR}. Since \(E_6=W_{13}\), \(E_{10}=W_{21}\), and \(E_{15}=W_{31}\), and since \(W_{31}\) occurs with multiplicity one in \(W_{13}\otimes W_{21}\), their projection \(p_{31}\), after identifying its image with \(E_{15}\), agrees with our \(p_{6,10}^{\,15}\) up to a nonzero scalar. Consequently,
\begin{equation}\label{eq:DGRp31}
p_{6,10}^{,15}(P_6\otimes P_{10})\neq0.
\end{equation}

Finally, we record the mode-generation statement underlying the channel-separation argument in Section~\ref{sec:channel-separation}. For $s\ge t\ge0$, DGR use a compact-group density lemma \cite[Lemma~3.3]{DGR} and show in the proof of \cite[Proposition~3.4]{DGR} that
\begin{equation}\label{eq:DGR-mode-span}
\Span_{\C}\{u_n v\mid u\in E_s\otimes L(1,s^2),\ v\in E_t\otimes L(1,t^2),\ n\in\Z\}
=\bigoplus_{r=s-t}^{s+t}E_r\otimes L(1,r^2),
\end{equation}
and the corresponding invariant-input version:
\begin{equation}\label{eq:DGR-invariant-mode-span}
\Span_{\C}\{u_n v\mid u\in P_s\otimes L(1,s^2),\ v\in P_t\otimes L(1,t^2),\ n\in\Z\}
=\bigoplus_{r=s-t}^{s+t}p_{s,t}^{\,r}(P_s\otimes P_t)\otimes L(1,r^2).
\end{equation}
The second equality is not obtained by merely taking fixed points in the first; it is the additional invariant-input statement proved in \cite{DGR}. In particular, whenever $p_{s,t}^{\,r}(P_s\otimes P_t)\neq0$, the $r$-th Virasoro channel is genuinely present. Section~\ref{sec:channel-separation} will strengthen this by identifying the unique mode whose component lies in the highest-weight line of $L(1,r^2)$ and proving its coefficient nonzero.

\subsection{Degenerate Virasoro modules at \texorpdfstring{$c=1$}{c=1}}
\label{subsec:c1-Virasoro}

Let $\Vir=\bigoplus_{n\in\Z}\C L(n)\oplus\C C$ denote the Virasoro Lie algebra with relations $[L(m),L(n)]=(m-n)L(m+n)+\frac{m^3-m}{12}\delta_{m+n,0}C$ and $[C,L(n)]=0$. For $c,h\in\C$, let $M(c,h)$ denote the Verma module of central charge $c$ and highest weight $h$, with fixed highest-weight vector $v_{c,h}$ satisfying $Cv_{c,h}=cv_{c,h}$, $L(0)v_{c,h}=hv_{c,h}$, and $L(n)v_{c,h}=0$ for $n>0$. The level decomposition is $M(c,h)=\bigoplus_{N\ge0}M(c,h)[N]$, where $M(c,h)[N]=\{w\in M(c,h)\mid L(0)w=(h+N)w\}$; we call $N$ the \emph{level}. A nonzero homogeneous vector $w\in M(c,h)[N]$ with $N>0$ is a \emph{singular vector} if $L(n)w=0$ for all $n>0$. Let $L(c,h)$ denote the irreducible quotient of $M(c,h)$.

 At $c=1$, $M(1,m^2/4)$ is reducible for $m\ge0$. Following Milas, we repeatedly use the irreducible modules $L(1,m^2/4)$, whose Verma covers have particularly simple embedding structure.

\begin{prop}[Milas]\label{prop:Milas-exact}
For every $m\in\Z_{\ge0}$, $M(1,m^2/4)$ contains, up to scalar, a unique singular vector at level $m+1$, which generates the maximal proper submodule. Equivalently, there is a short exact sequence
\[
0\longrightarrow M\!\left(1,\tfrac{(m+2)^2}{4}\right)\longrightarrow M\!\left(1,\tfrac{m^2}{4}\right)\longrightarrow L\!\left(1,\tfrac{m^2}{4}\right)\longrightarrow 0.
\]
\end{prop}

This is \cite[Proposition~2.1]{Milas}; note that $(m+2)^2/4-m^2/4=m+1$ explains the singular vector level. For the module in the $A_5$-orbifold decomposition, take $m=12$. Since $12^2/4=36$ and $14^2/4=49$, Proposition~\ref{prop:Milas-exact} gives
\begin{equation}\label{eq:M36}
0\longrightarrow M(1,49)\longrightarrow M(1,36)\longrightarrow L(1,36)\longrightarrow 0,
\end{equation}
so the maximal proper submodule of $M(1,36)$ is generated by a singular vector at level $49-36=13$.

We also record the graded-dimension formula for the local deformation argument. Let $p(N)$ denote the partition function with $p(N)=0$ for $N<0$. Since $\dim M(c,h)[N]=p(N)$, the exact sequence above implies
\begin{equation}\label{eq:c1-level-dim}
\dim L\!\left(1,\tfrac{m^2}{4}\right)[N]=p(N)-p(N-m-1).
\end{equation}
In particular,
\begin{equation}\label{eq:L136-level-dim}
\dim L(1,36)[N]=p(N)-p(N-13).
\end{equation}

Next we recall intertwining-operator multiplicities. For $L(1,0)$-modules $W_1,W_2,W_3$, let $I\binom{W_3}{W_1\;W_2}$ denote the space of intertwining operators of type $\binom{W_3}{W_1\;W_2}$ \cite{FHL}, with fusion coefficient $N_{W_1,W_2}^{W_3}:=\dim_{\C}I\binom{W_3}{W_1\;W_2}$.

\begin{prop}\label{prop:Milas-fusion}
For $m,n,r\in\Z_{\ge0}$,
\begin{equation}\label{eq:Milas-fusion-general}
\dim_{\C}I\binom{L(1,r^2/4)}{L(1,m^2/4)\;L(1,n^2/4)}
=
\begin{cases}
1,& r\in\{|m-n|,|m-n|+2,\ldots,m+n\},\\
0,& \text{otherwise}.
\end{cases}
\end{equation}
\end{prop}

When $mn>0$, this is exactly \cite[Theorem~3.3]{Milas}.
The cases $mn=0$ are the standard vacuum cases; cf. 
\cite[Remark~3.1]{Milas}. For irreducible $L(1,0)$-modules, the vacuum property and skew symmetry give the canonical identifications of the relevant intertwining spaces with module-homomorphism spaces; hence Schur's lemma gives
\[
\dim I\binom{L(1,r^2/4)}{L(1,0)\;L(1,n^2/4)}=\delta_{r,n},
\qquad
\dim I\binom{L(1,r^2/4)}{L(1,m^2/4)\;L(1,0)}=\delta_{r,m};
\]
see \cite[\S5.4, Proposition 5.4.7]{FHL}. These identities include $m=n=0$ and agree exactly with \eqref{eq:Milas-fusion-general}. 

For the two channels needed later, set $m=n=12$. The admissible indices are $r=24,22,\ldots,0$; in particular $r=0$ and $r=12$ are both allowed. Hence
\begin{equation}\label{eq:fusion-vacuum}
\dim_{\C}I\binom{L(1,0)}{L(1,36)\;L(1,36)}=1,
\qquad
\dim_{\C}I\binom{L(1,36)}{L(1,36)\;L(1,36)}=1.
\end{equation}
We refer to these as the \emph{vacuum channel} and \emph{self channel}, respectively.

\subsection{A generic first-row deformation}
\label{subsec:generic-first-row}

We shall later deform the two $c=1$ Virasoro channels in
$L(1,36)\times L(1,36)$ to first-row modules of a generic Virasoro VOA.
Using the Koshida--Kyt\"ol\"a parametrization
\cite[Section~4.2]{KoshidaKytola}, extended algebraically to
$\kappa\in\C^\times$, set
\begin{equation}\label{eq:ckappa}
c(\kappa)
=
1-\frac{3(\kappa-4)^2}{2\kappa}
=
13-6\left(\frac{\kappa}{4}+\frac{4}{\kappa}\right).
\end{equation}
The first-row Kac weights are
\[
h(\lambda;\kappa)
=
h_{1,1+\lambda}(\kappa)
=
\frac{\lambda(2(\lambda+2)-\kappa)}{2\kappa},
\qquad \lambda\in\Z_{\ge0},
\]
as in \cite[(4.9)--(4.10)]{KoshidaKytola}.
In the generic regime considered in \cite[(4.13)]{KoshidaKytola},
namely $\kappa\in(0,\infty)\setminus\Q$,
the Verma module
$M(c(\kappa),h(\lambda;\kappa))$
has maximal proper submodule isomorphic to
$M(c(\kappa),h(\lambda;\kappa)+\lambda+1)$
\cite[Lemma~4.1]{KoshidaKytola}.
For such $\kappa$, define the irreducible first-row module
\begin{equation}\label{eq:Qlambda-generic}
Q_\lambda(\kappa)
:=
M\bigl(c(\kappa),h(\lambda;\kappa)\bigr)
\big/
M\bigl(c(\kappa),h(\lambda;\kappa)+\lambda+1\bigr).
\end{equation}
The singular vector generating the denominator occurs at level
$\lambda+1$; see \cite[Lemma~4.1 and (4.12)]{KoshidaKytola}.

For compatibility with $c=1$, set $\kappa=4t$ and define
\(
c(t):=c(4t)=13-6t-6t^{-1}.
\)
For $\lambda=12$, we similarly set
\begin{equation}\label{eq:Ht}
H(t):=h(12;4t)=42t^{-1}-6.
\end{equation}
Thus $c(1)=1$ and $H(1)=36$. We shall call
$t_0\in\C^\times$ \emph{admissible generic} if
\[
4t_0\in(0,\infty)\setminus\Q.
\]
Equivalently, $\kappa_0:=4t_0$ lies in the generic regime of
\cite[(4.13)]{KoshidaKytola}. Notice that $t=1$ ($\kappa=4$) is not
admissible generic, so the generic results of \cite{KoshidaKytola}
cannot be specialized directly to this point.

Define the selection rule for $\lambda,\mu\in\Z_{\ge0}$ by
\begin{equation}\label{eq:Sel}
\operatorname{Sel}(\lambda,\mu)
:=
\{\nu\in\Z_{\ge0}\mid
|\lambda-\mu|\le\nu\le\lambda+\mu,\ 
\lambda+\mu+\nu\equiv0\pmod2\},
\end{equation}
the same index set as the $\mathfrak{sl}_2$ Clebsch--Gordan rule
\cite[Section~2.1]{KoshidaKytola}.
For $\kappa\in(0,\infty)\setminus\Q$, Koshida--Kyt\"ol\"a prove
\cite[Corollary~4.16]{KoshidaKytola}
\begin{equation}\label{eq:generic-first-row-fusion}
\dim_{\C}I\binom{Q_\nu(\kappa)}
{Q_\lambda(\kappa)\;Q_\mu(\kappa)}
=
\begin{cases}
1,& \nu\in\operatorname{Sel}(\lambda,\mu),\\
0,& \text{otherwise}.
\end{cases}
\end{equation}

In our case $\lambda=\mu=12$, so
$\operatorname{Sel}(12,12)=\{0,2,\ldots,24\}$; in particular
$0,12\in\operatorname{Sel}(12,12)$. Hence for every
$\kappa\in(0,\infty)\setminus\Q$, both intertwining spaces
$I\binom{Q_0(\kappa)}{Q_{12}(\kappa)\;Q_{12}(\kappa)}$
and
$I\binom{Q_{12}(\kappa)}{Q_{12}(\kappa)\;Q_{12}(\kappa)}$
are one-dimensional. After fixing highest-weight vectors, a nonzero
intertwining operator is determined by its nonzero initial coefficient
\cite[Proposition~4.12]{KoshidaKytola}. We use the normalization of
\cite[Definition~5.1]{KoshidaKytola}; we call these the generic
\emph{vacuum channel} and \emph{self channel}.

\subsection{The level-\texorpdfstring{$13$}{13} Benoit--Saint-Aubin operator}
\label{subsec:BSA}

We record the Benoit--Saint-Aubin singular-vector formula \cite{BenoitSaintAubin}, in the first-row parametrization and normalization of \cite[(4.12)]{KoshidaKytola}, specialized below to $\lambda=12$. Let $\Vir_{-}:=\bigoplus_{n\ge1}\C L(-n)$. An \emph{ordered composition} of $N\in\Z_{>0}$ is a sequence $\mathbf p=(p_1,\ldots,p_k)$ with $p_i\in\Z_{>0}$ and $\sum p_i=N$; write $\mathbf p\models N$, and denote its
length by $\ell(\mathbf p):=k$. Define prefix and suffix sums $P_j:=\sum_{i=1}^j p_i$ and $Q_j:=\sum_{i=j+1}^k p_i$ for $1\le j\le k-1$, so $P_j+Q_j=N$. Empty products equal $1$.

For general $\lambda\in\Z_{\ge0}$, the Benoit--Saint-Aubin singular operator is
\begin{equation}\label{eq:BSA-KK-general}
S_\lambda^{\mathrm{KK}}(\kappa)
=
\sum_{\mathbf p\models\lambda+1}
\left(-\frac4\kappa\right)^{\lambda+1-k}
\frac{(\lambda!)^2}{\prod_{j=1}^{k-1}P_jQ_j}
L(-p_1)\cdots L(-p_k).
\end{equation}
For $\kappa\in(0,\infty)\setminus\Q$, $S_\lambda^{\mathrm{KK}}(\kappa)v_{c(\kappa),h(\lambda;\kappa)}$ is a singular vector of level $\lambda+1$ generating the maximal proper submodule \cite[Lemma~4.1 and (4.12)]{KoshidaKytola}.

Now take $\lambda=12$, $\kappa=4t$, so the singular vector has level $13$. For later calculations we use a modified normalization:
\begin{equation}\label{eq:BSA-beta}
S_{13}(t):=\sum_{\mathbf p\models13}\beta_{\mathbf p}(t)L(-p_1)\cdots L(-p_k),
\qquad
\beta_{\mathbf p}(t):=\prod_{j=1}^{k-1}\frac{t}{Q_j(Q_j-13)}.
\end{equation}
Since $P_j+Q_j=13$, we have $Q_j-13=-P_j$, whence
\begin{equation}\label{eq:BSA-beta-rewritten}
\beta_{\mathbf p}(t)=(-1)^{k-1}t^{k-1}\frac{1}{\prod_{j=1}^{k-1}P_jQ_j}.
\end{equation}
Substituting $\lambda=12$, $\kappa=4t$ into \eqref{eq:BSA-KK-general} gives coefficient $(-1)^{13-k}t^{-(13-k)}(12!)^2/\prod P_jQ_j$. Since $(-1)^{13-k}=(-1)^{k-1}$, comparison yields
\begin{equation}\label{eq:BSA-normalization-comparison}
S_{12}^{\mathrm{KK}}(4t)=(12!)^2t^{-12}S_{13}(t).
\end{equation}
Put \(\Lambda:=\C[t,t^{-1}]\). Since \((12!)^2t^{-12}\in\Lambda^\times\), equation~\eqref{eq:BSA-normalization-comparison} shows that \(S_{12}^{\mathrm{KK}}(4t)\) and \(S_{13}(t)\) differ by a unit of the coefficient ring \(\Lambda\). In particular, for every admissible generic specialization \(t=t_0\), the two specialized operators are nonzero scalar multiples of one another and hence determine the same Benoit--Saint-Aubin singular-vector relation.

For later use, consider the monomial \(L(-1)^{13}\). For \(\mathbf p=(1^{13})\), we have \(Q_j=13-j\) (\(1\le j\le12\)), so
\begin{equation}\label{eq:BSA-Lminusone-coefficient}
\beta_{(1^{13})}(t)
=\prod_{j=1}^{12}\frac{t}{(13-j)(-j)}
\frac{t^{12}}{(12!)^2}
\in\Lambda^\times.
\end{equation}
Here the signs cancel since \(12\) is even. Thus the coefficient of \(L(-1)^{13}\) in \(S_{13}(t)\) is a unit of \(\Lambda\). This observation will be used in Section~\ref{sec:local-regularization}. Finally, the number of ordered compositions of \(13\) is \(2^{12}=4096\); we shall evaluate the associated scalar recurrence via dynamic programming rather than expanding all terms individually.

\section{Channel separation and lowest-mode nonvanishing}
\label{sec:channel-separation}

We now establish the mechanism connecting the $SL_2(\C)$-representation theory of multiplicity spaces in \eqref{eq:DGRdecomp} with individual vertex-operator modes. For each $m\ge0$, fix a nonzero highest-weight vector $\mathbf v_m\in L(1,m^2)$ with $L(0)\mathbf v_m=m^2\mathbf v_m$ and $L(n)\mathbf v_m=0$ ($n>0$). Under $V_{L_2}\cong\bigoplus_{m\ge0}E_m\otimes L(1,m^2)$, identify the Virasoro highest-weight subspace $E_m\otimes\C\mathbf v_m$ with $E_m$ via $\xi\leftrightarrow\xi\otimes\mathbf v_m$. Let $\Pi_r:V_{L_2}\to E_r\otimes L(1,r^2)$ denote the projection associated with \eqref{eq:DGRdecomp}.  Since \eqref{eq:DGRdecomp} is a decomposition into submodules for the commuting actions of \(SL_2(\C)\) and \(L(1,0)\), the projection \(\Pi_r\) is equivariant for both actions.
 Write $Y_r(a,z)$ for the action of $L(1,0)$ on $L(1,r^2)$, so that $Y_{V_{L_2}}(a,z)|_{E_r\otimes L(1,r^2)}=\id_{E_r}\otimes Y_r(a,z)$.

\begin{prop}[Channel separation]\label{prop:channel-separation}
Let $s,t,r\in\Z_{\ge0}$ with $|s-t|\le r\le s+t$, and fix the nonzero Clebsch--Gordan projection $p_{s,t}^{\,r}:E_s\otimes E_t\to E_r$ from Section~\ref{subsec:rank-one}. There exists a unique nonzero intertwining operator $\mathcal Y_{s,t}^{\,r}\in I\binom{L(1,r^2)}{L(1,s^2)\;L(1,t^2)}$ such that for all $\xi\in E_s$, $\eta\in E_t$, $u\in L(1,s^2)$, $v\in L(1,t^2)$,
\begin{equation}\label{eq:channel-factorization}
\Pi_r Y(\xi\otimes u,z)(\eta\otimes v)=p_{s,t}^{\,r}(\xi\otimes\eta)\otimes\mathcal Y_{s,t}^{\,r}(u,z)v.
\end{equation}
\end{prop}

\begin{proof}
\noindent\textbf{Step 1: mode-wise factorization.} Fix $u\in L(1,s^2)$, $v\in L(1,t^2)$, $n\in\Z$, and define $B_n^{u,v}:E_s\otimes E_t\to E_r\otimes L(1,r^2)$ by
\[
B_n^{u,v}(\xi\otimes\eta):=\Pi_r((\xi\otimes u)_n(\eta\otimes v)).
\]
Since $SL_2(\C)$ acts by VOA automorphisms and trivially on the Virasoro factors, $B_n^{u,v}$ is $SL_2(\C)$-equivariant. Because $E_s\otimes E_t$ is finite-dimensional and $L(1,r^2)$ is a trivial $SL_2(\C)$-module, there is a canonical isomorphism
\[
\Hom_{SL_2(\C)}\!\bigl(E_s\otimes E_t,E_r\otimes L(1,r^2)\bigr)
\cong
\Hom_{SL_2(\C)}(E_s\otimes E_t,E_r)\otimes L(1,r^2).
\]
No completed tensor product is involved: since \(E_s\otimes E_t\) is finite-dimensional, the image of any linear map \(E_s\otimes E_t\to E_r\otimes L(1,r^2)\) is contained in \(E_r\otimes W\) for some finite-dimensional subspace \(W\subset L(1,r^2)\). Hence the usual algebraic tensor-product identification applies.
 Since
$\Hom_{SL_2(\C)}(E_s\otimes E_t,E_r)=\C p_{s,t}^{\,r}$,
there is therefore a unique $b_n(u,v)\in L(1,r^2)$ such that
\[
B_n^{u,v}(\xi\otimes\eta)
=p_{s,t}^{\,r}(\xi\otimes\eta)\otimes b_n(u,v)
\]
for all $\xi,\eta$. Uniqueness shows that $b_n(u,v)$ is bilinear in $u,v$. Choose $\xi_0,\eta_0$ with
$p_{s,t}^{\,r}(\xi_0\otimes\eta_0)\neq0$. Ambient VOA lower truncation gives
$(\xi_0\otimes u)_n(\eta_0\otimes v)=0$ for $n\gg0$, and hence
$b_n(u,v)=0$ for $n\gg0$. Thus
\[
\mathcal Y_{s,t}^{\,r}(u,z)v
:=\sum_{n\in\Z}b_n(u,v)z^{-n-1}
\]
is lower truncated, and summing the mode identities gives
\eqref{eq:channel-factorization}.

\noindent\textbf{Step 2: $L(-1)$-derivative property.} 
Under the decomposition \eqref{eq:DGRdecomp}, the Virasoro algebra acts only on the second tensor factor, so
$
L(-1)(\xi\otimes u)=\xi\otimes L(-1)u.
$
Using the ambient VOA derivative property
$
Y(L(-1)a,z)=\frac{d}{dz}Y(a,z),
$
the fact that \(\Pi_r\) is independent of the formal variable \(z\), and the factorization from Step~1, we obtain
\begin{align*}
p_{s,t}^{\,r}(\xi\otimes\eta)\otimes
\mathcal Y_{s,t}^{\,r}(L(-1)u,z)v
&=\Pi_rY(\xi\otimes L(-1)u,z)(\eta\otimes v)\\
&=\Pi_r\frac{d}{dz}Y(\xi\otimes u,z)(\eta\otimes v)\\
&=\frac{d}{dz}\Pi_rY(\xi\otimes u,z)(\eta\otimes v)\\
&=p_{s,t}^{\,r}(\xi\otimes\eta)\otimes
\frac{d}{dz}\mathcal Y_{s,t}^{\,r}(u,z)v.
\end{align*}
Choosing \(\xi,\eta\) with
\(p_{s,t}^{\,r}(\xi\otimes\eta)\neq0\), we obtain
\begin{equation}\label{eq:channel-Lminus1}
\mathcal Y_{s,t}^{\,r}(L(-1)u,z)
=\frac{d}{dz}\mathcal Y_{s,t}^{,r}(u,z).
\end{equation}

\noindent\textbf{Step 3: Jacobi identity.} For $a\in L(1,0)$, apply the VOA Jacobi identity to
$a$, $\xi\otimes u$, and $\eta\otimes v$, and then apply $\Pi_r$.
Using
\[
\Pi_rY(a,x_1)=(\id_{E_r}\otimes Y_r(a,x_1))\Pi_r,
\]
together with
\[
Y(a,x_1)(\eta\otimes v)=\eta\otimes Y_t(a,x_1)v,
\qquad
Y(a,x_0)(\xi\otimes u)=\xi\otimes Y_s(a,x_0)u,
\]
and applying the factorization from Step~1 coefficientwise to the formal series \(Y_t(a,x_1)v\) and \(Y_s(a,x_0)u\), we obtain
\[
\begin{aligned}
&x_0^{-1}\delta\!\left(\frac{x_1-x_2}{x_0}\right)
 p_{s,t}^{\,r}(\xi\otimes\eta)\otimes
 Y_r(a,x_1)\mathcal Y_{s,t}^{\,r}(u,x_2)v\\
&\quad-
 x_0^{-1}\delta\!\left(\frac{x_2-x_1}{-x_0}\right)
 p_{s,t}^{\,r}(\xi\otimes\eta)\otimes
 \mathcal Y_{s,t}^{\,r}(u,x_2)Y_t(a,x_1)v\\
&=
 x_2^{-1}\delta\!\left(\frac{x_1-x_0}{x_2}\right)
 p_{s,t}^{\,r}(\xi\otimes\eta)\otimes
 \mathcal Y_{s,t}^{\,r}(Y_s(a,x_0)u,x_2)v.
\end{aligned}
\]
Choose $\xi,\eta$ with
$p_{s,t}^{\,r}(\xi\otimes\eta)\neq0$ and cancel this common tensor
factor. The resulting identity is exactly the FHL Jacobi identity for
an intertwining operator of type
$\binom{L(1,r^2)}{L(1,s^2)\;L(1,t^2)}$.
Together with lower truncation and the $L(-1)$-derivative property
\eqref{eq:channel-Lminus1}, this proves
$\mathcal Y_{s,t}^{\,r}\in
I\binom{L(1,r^2)}{L(1,s^2)\;L(1,t^2)}$ \cite{FHL}.

\noindent\textbf{Step 4: nonvanishing.} The DGR mode-span statement recalled in \eqref{eq:DGR-mode-span} is formulated for the ordered range $s\ge t$. If $s<t$, VOA skew-symmetry gives
\[
Y(a,z)b=e^{zL(-1)}Y(b,-z)a.
\]
The projection $\Pi_r$ commutes with $L(-1)$, and $e^{zL(-1)}$ is invertible as a formal power series. Consequently the $r$-channel vanishes for every pair of inputs of types $(s,t)$ if and only if it vanishes for every pair of the exchanged types $(t,s)$. Thus, for the nonvanishing argument, skew-symmetry reduces the case $s<t$ to the ordered case $t>s$, and we may apply the DGR result in all cases. If $\mathcal Y_{s,t}^{\,r}=0$, then $\Pi_rY(\xi\otimes u,z)(\eta\otimes v)=0$ for all inputs, contradicting \cite[Lemma~3.3 and proof of Proposition~3.4]{DGR}, which guarantees that $E_r\otimes L(1,r^2)$ appears nontrivially in the span of modes when $|s-t|\le r\le s+t$. Uniqueness follows from Step~1 once $p_{s,t}^{\,r}$ is fixed.
\end{proof}

The next lemma shows that a nonzero Virasoro channel cannot first appear at positive level.

\begin{lem}[Lowest-weight coefficient]\label{lem:lowest-coeff}
Let $r,s,t\in\Z_{\ge0}$, and let
\(
\mathcal Y\in
I\binom{L(1,r^2)}{L(1,s^2)\;L(1,t^2)}
\)
be nonzero. Then
\begin{equation}\label{eq:lowest-expansion}
\mathcal Y(\mathbf v_s,z)\mathbf v_t
=
z^{r^2-s^2-t^2}
\left(
c\,\mathbf v_r+\sum_{N\ge1}w_Nz^N
\right),
\end{equation}
where $c\in\C^\times$ and
$w_N\in L(1,r^2)[N]$.
\end{lem}

\begin{proof}
\noindent\textbf{Step 1: highest-vector matrix coefficient is nonzero.} 
Suppose, to the contrary, that
$
\mathcal Y(\mathbf v_s,z)\mathbf v_t=0.
$
Since $\mathbf v_s$ is primary of conformal weight $s^2$, the Virasoro commutator formula gives
$$
[L(m),\mathcal Y(\mathbf v_s,z)]
=
z^m\left(z\frac{d}{dz}+(m+1)s^2\right)
\mathcal Y(\mathbf v_s,z)
\qquad (m\in\mathbb Z).
$$
First set
$
K_t:=
\left\{
v\in L(1,t^2)\ \middle|\ 
\mathcal Y(\mathbf v_s,z)v=0
\right\}.
$
We claim that $K_t$ is an $L(1,0)$-submodule. Indeed, if
$v\in K_t$, then for every $m\in\mathbb Z$,
\begin{align*}
\mathcal Y(\mathbf v_s,z)L(m)v
&=
L(m)\mathcal Y(\mathbf v_s,z)v-
[L(m),\mathcal Y(\mathbf v_s,z)]v\\
&=
-
z^m\left(z\frac{d}{dz}+(m+1)s^2\right)
\mathcal Y(\mathbf v_s,z)v=0.
\end{align*}
Hence $L(m)v\in K_t$ for every $m\in\mathbb Z$, and therefore
$K_t$ is an $L(1,0)$-submodule of $L(1,t^2)$. By our assumption,
$\mathbf v_t\in K_t$. Since $L(1,t^2)$ is irreducible and
$\mathbf v_t\neq0$, it follows that
$
K_t=L(1,t^2).
$
Consequently,
$
\mathcal Y(\mathbf v_s,z)v=0$
for every 
$v\in L(1,t^2).$
Now set
$$
K_s:=
\left\{
u\in L(1,s^2)\ \middle|\ 
\mathcal Y(u,z)v=0
\text{ for all }v\in L(1,t^2)
\right\}.
$$
We claim that $K_s$ is an $L(1,0)$-submodule. Let
$u\in K_s$ and $v\in L(1,t^2)$. Apply the intertwining-operator
Jacobi identity with the Virasoro vector $\omega$:
\begin{align*}
&x_0^{-1}\delta\left(\frac{x_1-x_2}{x_0}\right)
Y_r(\omega,x_1)\mathcal Y(u,x_2)v\\
&\quad-
x_0^{-1}\delta\left(\frac{x_2-x_1}{-x_0}\right)
\mathcal Y(u,x_2)Y_t(\omega,x_1)v\\
&=
x_2^{-1}\delta\left(\frac{x_1-x_0}{x_2}\right)
\mathcal Y\left(Y_s(\omega,x_0)u,x_2\right)v .
\end{align*}
Since $u\in K_s$, the first term on the left is zero, and the second
is also zero coefficientwise because every coefficient of
$Y_t(\omega,x_1)v$ lies in $L(1,t^2)$. Hence the right-hand side
vanishes. Taking $\operatorname{Res}_{x_1}$ gives
$
\mathcal Y\!\left(Y_s(\omega,x_0)u,x_2\right)v=0.
$
Since
$
Y_s(\omega,x_0)u
=
\sum_{n\in\mathbb Z}L(n)u\,x_0^{-n-2},
$
comparison of coefficients of $x_0$ yields
$
\mathcal Y(L(n)u,x_2)v=0(n\in\mathbb Z).$
As $v$ was arbitrary, $L(n)u\in K_s$ for every $n\in\mathbb Z$.
Thus $K_s$ is an $L(1,0)$-submodule.
Since $K_t=L(1,t^2)$, we have
$
\mathcal Y(\mathbf v_s,z)v=0$
for every $v\in L(1,t^2),
$
and hence $\mathbf v_s\in K_s$. Since $L(1,s^2)$ is irreducible,
$
K_s=L(1,s^2).
$
It follows that $\mathcal Y=0$, contradicting the assumption that
$\mathcal Y$ is nonzero. Therefore
$
\mathcal Y(\mathbf v_s,z)\mathbf v_t\neq0.
$

\noindent\textbf{Step 2: form of expansion.} Set $h_s=s^2$, $h_t=t^2$, $h_r=r^2$. The $L(0)$-commutator gives $L(0)\mathcal Y(\mathbf v_s,z)\mathbf v_t=(h_s+h_t+z\frac{d}{dz})\mathcal Y(\mathbf v_s,z)\mathbf v_t$. Since $L(0)$ acts as $h_r+N$ on level $N$, the expansion must be $\mathcal Y(\mathbf v_s,z)\mathbf v_t=z^{h_r-h_s-h_t}\sum_{N\ge0}w_Nz^N$ with $w_N\in L(1,r^2)[N]$. By Step~1, not all $w_N$ vanish.

\noindent\textbf{Step 3: first nonzero coefficient at level zero.} Let $N_0\ge0$ be minimal with $w_{N_0}\neq0$. For $m>0$, using $L(m)\mathbf v_t=0$ and the primary-vector commutator formula, we obtain
$$L(m)\mathcal Y(\mathbf v_s,z)\mathbf v_t = z^m\left(z\frac d{dz}+(m+1)s^2\right) \mathcal Y(\mathbf v_s,z)\mathbf v_t.$$
Comparing coefficients of $z^{r^2-s^2-t^2+N}$, with $w_j:=0$ for $j<0$, gives $L(m)w_N=(r^2-t^2+N+m(s^2-1))w_{N-m}$. Minimality implies $L(m)w_{N_0}=0$ for $m>0$, so $w_{N_0}$ is singular. If $N_0>0$, the submodule generated by $w_{N_0}$ has weights $\ge r^2+N_0>r^2$, contradicting irreducibility of $L(1,r^2)$. Hence $N_0=0$, and $L(1,r^2)[0]=\C\mathbf v_r$ gives $w_0=c\mathbf v_r$ with $c\neq0$.
\end{proof}

\begin{cor}[Fixed-mode nonvanishing]\label{cor:fixed-mode}
Let $s,t,r\in\Z_{\ge0}$ with $|s-t|\le r\le s+t$.
Let $\xi\in E_s$ and $\eta\in E_t$, and regard
$\xi\otimes\mathbf v_s$ and $\eta\otimes\mathbf v_t$
as Virasoro highest-weight vectors in $V_{L_2}$.
Suppose $p_{s,t}^{\,r}(\xi\otimes\eta)\neq0$. Set
\begin{equation}\label{eq:fixed-mode-index}
n(s,t;r):=s^2+t^2-r^2-1.
\end{equation}
Let
\(
\operatorname{pr}_{r,0}:L(1,r^2)\longrightarrow
L(1,r^2)[0]=\C\mathbf v_r
\)
denote the projection onto the level-zero subspace, and define
\(
\operatorname{pr}_{E_r\otimes\C\mathbf v_r}
:=
(\id_{E_r}\otimes\operatorname{pr}_{r,0})\circ\Pi_r
:
V_{L_2}\longrightarrow E_r\otimes\C\mathbf v_r.
\)
Then $(\xi\otimes\mathbf v_s)_{n(s,t;r)}(\eta\otimes\mathbf v_t)$ has nonzero component in $E_r\otimes\C\mathbf v_r$. Precisely, there exists $c_{s,t}^{\,r}\in\C^\times$, depending on the fixed normalizations of $p_{s,t}^{\,r}$ and of the highest-weight vectors $\mathbf v_s,\mathbf v_t,\mathbf v_r$, but independent of $\xi,\eta$, such that
\begin{equation}\label{eq:fixed-mode-leading}
\operatorname{pr}_{E_r\otimes\C\mathbf v_r}\bigl((\xi\otimes\mathbf v_s)_{n(s,t;r)}(\eta\otimes\mathbf v_t)\bigr)=c_{s,t}^{\,r}\,p_{s,t}^{\,r}(\xi\otimes\eta)\otimes\mathbf v_r.
\end{equation}
\end{cor}

\begin{proof}
By Proposition~\ref{prop:channel-separation}, $\Pi_r Y(\xi\otimes\mathbf v_s,z)(\eta\otimes\mathbf v_t)=p_{s,t}^{\,r}(\xi\otimes\eta)\otimes\mathcal Y_{s,t}^{\,r}(\mathbf v_s,z)\mathbf v_t$. Lemma~\ref{lem:lowest-coeff} gives $\mathcal Y_{s,t}^{\,r}(\mathbf v_s,z)\mathbf v_t=z^{r^2-s^2-t^2}(c_{s,t}^{\,r}\mathbf v_r+\sum_{N\ge1}w_Nz^N)$ with $c_{s,t}^{\,r}\neq0$. The mode convention $Y(a,z)=\sum a_nz^{-n-1}$ implies the power $z^{r^2-s^2-t^2}$ corresponds to $n=s^2+t^2-r^2-1$. Extracting this coefficient yields \eqref{eq:fixed-mode-leading}; nonvanishing follows since both $c_{s,t}^{\,r}$ and $p_{s,t}^{\,r}(\xi\otimes\eta)$ are nonzero.
\end{proof}

\section{Virasoro descendants and triangular \texorpdfstring{$C_2$}{C2}-reduction}
\label{sec:triangular-reduction}

We now pass from the channel decomposition of Section~\ref{sec:channel-separation} to relations in the  $C_2$-Poisson algebra $R(U)=U/C_2(U)$ where $U=V_{L_2}^{A_5}$ and $a=[\omega]\in R(U)$. If $p\in P_r\subset U$ is a Virasoro highest-weight vector ($\wt p=r^2$), a homogeneous Virasoro descendant $w$ of $p$ has level $N$ if $L(0)w=(r^2+N)w$. By PBW, the level-$N$ descendant space is spanned by vectors $L(-n_1)\cdots L(-n_\ell)p$ with $n_1\ge\cdots\ge n_\ell\ge1$ and $N=\sum n_i$.

\begin{lem}[Descendant reduction]\label{lem:descendant}
Let $p\in U$ be a Virasoro highest-weight vector and $w$ a homogeneous descendant at level $N$. Then $[w]\in\C a^{N/2}[p]$ if $N$ is even, and $[w]=0$ if $N$ is odd. Explicitly, for $N=2k$, $[w]=c_w a^k[p]$ for some $c_w\in\C$.
\end{lem}

\begin{proof}
By PBW, it suffices to treat monomials $L(-n_1)\cdots L(-n_\ell)p$ with $n_1\ge\cdots\ge n_\ell\ge1$, $\sum n_i=N$. If $n_1\ge3$, then $L(-n_1)=\omega_{-n_1+1}$ with $-n_1+1\le-2$, so Lemma~\ref{lem:deepmodes} gives $[\omega_{-n_1+1}u]=0$; hence such monomials vanish in $R(U)$. Remaining monomials have $n_i\in\{1,2\}$, taking the form $L(-2)^kL(-1)^jp$ with $N=2k+j$. Since $L(-2)u=\omega_{-1}u$ and $[\omega][u]=[\omega_{-1}u]$, repeated application gives $[L(-2)^kL(-1)^jp]=a^k[L(-1)^jp]$. For $j>0$, set $u=L(-1)^{j-1}p$; then $L(-1)^jp=u_{-2}\mathbf1\in C_2(U)$, so $[L(-1)^jp]=0$. Thus only $L(-2)^kp$ survives, with $N=2k$ and $[L(-2)^kp]=a^k[p]$.
\end{proof}

For each $r\ge0$, let
\(
\Pi_r:V_{L_2}\to E_r\otimes L(1,r^2)
\)
be the projection defined above. Since $\Pi_r$ is $I$-equivariant and
$U=V_{L_2}^{I}$, we have
\(
\Pi_r(U)
=(E_r\otimes L(1,r^2))^I
=P_r\otimes L(1,r^2).
\)

\begin{prop}[Triangular $C_2$-reduction]\label{prop:triangular}
Let $s,t>0$ and set $M:=s+t$. There exist a constant
$\kappa_{s,t}\in\C^\times$, depending only on the fixed normalization
choices made in Section~\ref{sec:channel-separation}, and, for each
$r$ satisfying
\(
|s-t|\le r<M,
M^2-r^2\in2\Z_{\ge0},
\)
a linear map
\(
\beta_r:P_s\otimes P_t\longrightarrow P_r
\)
such that, for all $p\in P_s$ and $q\in P_t$, one has in $R(U)$
\begin{equation}\label{eq:triangular-relation}
\kappa_{s,t}[\mu_{s,t}(p,q)]
+
\sum_{\substack{|s-t|\le r<M\\ M^2-r^2\in2\Z_{\ge0}}}
a^{(M^2-r^2)/2}
[\beta_r(p\otimes q)]
=0.
\end{equation}
The maps $\beta_r$ depend on auxiliary choices of linear sections and
are not asserted to be canonical; only their existence and linearity
are needed.
\end{prop}

\begin{proof}

If \(p=0\) or \(q=0\), then the desired identity is immediate from
bilinearity, since both \(\beta_r(p,q)\) and
\(\mu_{s,t}(p,q)\) vanish.  Hence, for the remainder of the proof,
we may and shall assume that
\(
p\neq0, q\neq0.
\)

\noindent\textbf{Step 1: mode reaching top channel at level zero.} Since $\wt p=s^2$, $\wt q=t^2$, and $\wt(u_nv)=\wt u+\wt v-n-1$, set $n_0:=-2st-1$. Then $\wt(p_{n_0}q)=s^2+t^2+2st=M^2$. Equivalently, $n_0=s^2+t^2-M^2-1=n(s,t;M)$ from Corollary~\ref{cor:fixed-mode}.

\noindent\textbf{Step 2: this mode lies in $C_2(U)$.} Since $s,t>0$, $2st+1\ge3$, so $n_0\le-3$. Lemma~\ref{lem:deepmodes} with $m=2st+1\ge2$ gives $p_{-2st-1}q\in C_2(U)$, hence $[p_{-2st-1}q]=0$ in $R(U)$.

\noindent\textbf{Step 3: channel decomposition.}
Since \(p\in P_s\subset E_s\) and \(q\in P_t\subset E_t\), the
Clebsch--Gordan decomposition
\[
E_s\otimes E_t
\cong
\bigoplus_{r=|s-t|}^{s+t}E_r
\]
shows that only the channels
\(
|s-t|\le r\le s+t=M
\)
can occur in \(p_{-2st-1}q\).  Equivalently, by the mode-wise
channel decomposition established above,
\(\Pi_r\bigl(p_{-2st-1}q\bigr)=0\)
for  \(r<|s-t|\)  or  
\(r>M.\)
Since \(p,q\in U=V_{L_2}^{I}\) and \(I\) acts by VOA
automorphisms, \(p_{-2st-1}q\in U\).  Moreover, \(\Pi_r\) is
\(I\)-equivariant; hence
\[
\Pi_r\bigl(p_{-2st-1}q\bigr)
\in
\bigl(E_r\otimes L(1,r^2)\bigr)^I
=
P_r\otimes L(1,r^2).
\]
Thus
\[
p_{-2st-1}q
=
\sum_{r=|s-t|}^{M} w_r(p,q),
\qquad
w_r(p,q):=
\Pi_r\bigl(p_{-2st-1}q\bigr)
\in
P_r\otimes L(1,r^2).
\]
By Step~1,
\(
\wt\bigl(p_{-2st-1}q\bigr)=M^2.
\)
Since the \(r\)-th summand
\(P_r\otimes L(1,r^2)\) has lowest conformal weight \(r^2\),
the component \(w_r(p,q)\) lies at Virasoro level
\(
N_r:=M^2-r^2.
\)
Therefore
\(
w_r(p,q)
\in
P_r\otimes L(1,r^2)[N_r],\)
\(
N_r:=M^2-r^2.
\)

\noindent\textbf{Step 4: top channel coefficient nonzero.} For $r=M$, $N_M=0$, so $w_M(p,q)\in P_M\otimes\C\mathbf v_M$. Since $\mu_{s,t}$ is ordinary multiplication of binary forms and $\C[x_1,x_2]$ is a domain, $p,q\neq0$ implies $\mu_{s,t}(p,q)\neq0$. Corollary~\ref{cor:fixed-mode} gives $w_M(p,q)=\kappa_{s,t}\mu_{s,t}(p,q)\otimes\mathbf v_M$ with $\kappa_{s,t}\neq0$ independent of $p,q$. Identifying $P_M\otimes\C\mathbf v_M\cong P_M$, the top channel contributes $[w_M(p,q)]=\kappa_{s,t}[\mu_{s,t}(p,q)]$.

\noindent\textbf{Step 5: reduce lower channels modulo $C_2(U)$.} For $|s-t|\le r<M$, for arbitrary $p'\in P_s$, $q'\in P_t$, let $w_r(p',q'):=\Pi_r(p'_{n_0}q')\in P_r\otimes L(1,r^2)[N_r]$. Choosing a basis $e_1,\ldots,e_d$ of $P_r$, write $w_r(p',q')=\sum_i e_i\otimes u_i(p',q')$ with $u_i(p',q')\in L(1,r^2)[N_r]$. If $N_r$ is odd, Lemma~\ref{lem:descendant} gives $[w_r(p',q')]=0$ for all $p',q'$, and we set $\beta_r:=0$. If $N_r=2k$ is even, Lemma~\ref{lem:descendant} gives
$$
[w_r(p',q')]\in a^k\Span\{[e]\mid e\in P_r\}=\operatorname{Im}F_{r,k},
\qquad
F_{r,k}:P_r\longrightarrow R(U),\quad e\longmapsto a^k[e].
$$
Since $(p',q')\mapsto[w_r(p',q')]$ is bilinear, it induces a linear map
$$
G_r:P_s\otimes P_t\longrightarrow\operatorname{Im}F_{r,k},
\qquad
G_r(p'\otimes q')=[w_r(p',q')].
$$
Choose a linear section $\sigma_{r,k}:\operatorname{Im}F_{r,k}\to P_r$ of the surjection $F_{r,k}:P_r\twoheadrightarrow\operatorname{Im}F_{r,k}$, and set $\beta_r:=\sigma_{r,k}\circ G_r$. Then
$$
[w_r(p',q')]=a^{(M^2-r^2)/2}[\beta_r(p'\otimes q')]
$$
for all $p'\in P_s$, $q'\in P_t$.

\noindent\textbf{Step 6: sum the channel relations.}
By Step~2,
\(
[p_{-2st-1}q]=0\)
in \(R(U).
\)
On the other hand, Step~3 gives the channel decomposition
\(
p_{-2st-1}q
=
\sum_{r=|s-t|}^{M}w_r(p,q).
\)
Therefore, after passing to \(R(U)\),
\(
0
=
[p_{-2st-1}q]
=
\sum_{r=|s-t|}^{M}[w_r(p,q)].
\)
We now separate the top channel \(r=M\) from the lower channels:
\(
0
=
[w_M(p,q)]
+
\sum_{r=|s-t|}^{M-1}[w_r(p,q)].
\)
By Step~4,
\(
[w_M(p,q)]
=
\kappa_{s,t}[\mu_{s,t}(p,q)].
\)
For a lower channel \(r<M\), if \(N_r\) is odd, Step~5 gives
\(
[w_r(p,q)]=0.
\)
If \(N_r\) is even, say \(N_r=2k\), then Step~5 gives
\[
[w_r(p,q)]
=
a^k[\beta_r(p\otimes q)]
=
a^{(M^2-r^2)/2}[\beta_r(p\otimes q)].
\]
Hence all lower channels with \(M^2-r^2\) odd disappear, while the
even-level lower channels contribute
\[
\sum_{\substack{|s-t|\le r<M\\ M^2-r^2\in2\Z_{\ge0}}}
a^{(M^2-r^2)/2}
[\beta_r(p\otimes q)].
\]
Substituting these expressions into the preceding channel sum yields
\[
\kappa_{s,t}[\mu_{s,t}(p,q)]
+
\sum_{\substack{|s-t|\le r<M\\ M^2-r^2\in2\Z_{\ge0}}}
a^{(M^2-r^2)/2}
[\beta_r(p\otimes q)]
=
0,
\]
which is precisely \eqref{eq:triangular-relation}.
\end{proof}

\section{Finite generation over \texorpdfstring{$\C[a]$}{C[a]}}
\label{sec:finite-generation}

We now use the binary icosahedral invariant ring to convert the triangular reduction of Proposition~\ref{prop:triangular} into a finite-generation statement for the  $C_2$-Poisson algebra. Recall from Section~\ref{subsec:rank-one} that $\mathcal P=\bigoplus_{m\ge0}P_m$ is a connected graded commutative algebra under the Cartan product with $P_0=\C\mathbf1$, and with respect to the half-degree grading,
\begin{equation}\label{eq:P-ring-again}
\mathcal P\cong\C[A,B,C]/(A^5+B^3+C^2),\qquad \deg A=6,\ \deg B=10,\ \deg C=15.
\end{equation}
This is the binary icosahedral invariant-ring description of \cite[Proposition~3.2(iii)]{DGR} after halving ordinary polynomial degrees. Let $\mathcal P_+:=\bigoplus_{m>0}P_m$ be the augmentation ideal; its square satisfies $(\mathcal P_+^2)_m=\sum_{s+t=m,\,s,t>0}\mu_{s,t}(P_s\otimes P_t)$.

Recall the previously chosen nonzero invariant primaries \( X\in P_6, Y\in P_{10}, Z\in P_{15}, \) corresponding respectively to the basic binary icosahedral invariants \(A,B,C\).    

\begin{lem}\label{lem:small-invariant-degrees}
The invariant primary spaces in low degrees satisfy $P_r=0$ for $1\le r\le15$, $r\notin\{6,10,12,15\}$, while $P_6=\C X$, $P_{10}=\C Y$, $P_{12}=\C\,\mu_{6,6}(X,X)$, $P_{15}=\C Z$, and $P_0=\C\mathbf1$.
\end{lem}

\begin{proof}
By \eqref{eq:icosaring}, $P_r$ is the degree-$r$ component of $\C[A,B,C]/(A^5+B^3+C^2)$. The relation has degree $30$, imposing no constraint below $30$. For $0<r\le15$, monomials $A^iB^jC^k$ have degree $6i+10j+15k=r$; the only positive solutions are $r=6,10,12,15$ corresponding to $A,B,A^2,C$. Thus these spaces are one-dimensional and others vanish. Under the Cartan-product realization, $A^2$ corresponds (up to scalar) to $\mu_{6,6}(X,X)$, giving $P_{12}$.
\end{proof}

\begin{lem}\label{lem:indecomposable-primary}
The graded space of indecomposables is $\mathcal P_+/\mathcal P_+^2\cong\C\overline A\oplus\C\overline B\oplus\C\overline C$ with $\deg\overline A=6$, $\deg\overline B=10$, $\deg\overline C=15$. Consequently, for every $m>0$ with $m\notin\{6,10,15\}$,
\begin{equation}\label{eq:Pmproducts}
P_m=\sum_{\substack{s+t=m\\s,t>0}}\mu_{s,t}(P_s\otimes P_t).
\end{equation}
\end{lem}

\begin{proof}
Let $S=\C[A,B,C]$, $\mathfrak m=(A,B,C)$, $F=A^5+B^3+C^2$. Then $\mathcal P\cong S/(F)$ and $\mathcal P_+$ is the image of $\mathfrak m$. Since each monomial in $F$ has algebraic degree $\ge2$, $F\in\mathfrak m^2$, so $(F)\subset\mathfrak m^2$. Hence $\mathcal P_+/\mathcal P_+^2\cong\mathfrak m/(\mathfrak m^2+(F))=\mathfrak m/\mathfrak m^2=\C\overline A\oplus\C\overline B\oplus\C\overline C$ with degrees $6,10,15$. For $m\notin\{6,10,15\}$, $(\mathcal P_+/\mathcal P_+^2)_m=0$, so $P_m=(\mathcal P_+)_m=(\mathcal P_+^2)_m$, yielding \eqref{eq:Pmproducts}.
\end{proof}

In \(R(U)\), set \( x:=[X], y:=[Y], z:=[Z], \) and recall that \(a=[\omega]\). Define the \(\C[a]\)-submodule 
\begin{equation}\label{eq:M4-def} \mathcal M := \C[a]\cdot1+\C[a]\cdot x+\C[a]\cdot y+\C[a]\cdot z \subset R(U). \end{equation}

\begin{thm}[Four-generator reduction]\label{thm:rank4}
We have $R(U)=\mathcal M$.
\end{thm}

\begin{proof}
We prove $R(U)\subset\mathcal M$ in two stages.

\noindent\textbf{Step 1: reduction of all Virasoro primary classes.} We prove by strong induction on \(m\) that \( [P_m]\subset\mathcal M  \) for every  $m\ge0$. Assume that \([P_j]\subset\mathcal M\) for all \(0\le j<m\). If \(m=0\), then \(P_0=\C\mathbf1\), so \([P_0]\subset\mathcal M\). If \(m\in\{6,10,15\}\), then \( [P_6]\subset\C x, [P_{10}]\subset\C y, [P_{15}]\subset\C z, \) respectively, and hence again \([P_m]\subset\mathcal M\). It remains to consider \(m>0\) with \(m\notin\{6,10,15\}\). By Lemma~\ref{lem:indecomposable-primary}, \[ P_m = \sum_{\substack{s+t=m\\s,t>0}} \mu_{s,t}(P_s\otimes P_t). \] Thus every \(p\in P_m\) can be written as a finite sum \( p = \sum_\nu \mu_{s_\nu,t_\nu}(p_\nu,q_\nu), \) where \( s_\nu,t_\nu>0, s_\nu+t_\nu=m, p_\nu\in P_{s_\nu}, q_\nu\in P_{t_\nu}. \) After omitting zero summands, we may and do assume \(p_\nu\neq0\) and \(q_\nu\neq0\) for every \(\nu\). Since \(s_\nu,t_\nu>0\) and \(s_\nu+t_\nu=m\), we have \(s_\nu,t_\nu<m\). For each \(\nu\), Proposition~\ref{prop:triangular}, applied with \(s=s_\nu\), \(t=t_\nu\), and \(M=m\), gives \[ \kappa_{s_\nu,t_\nu} [\mu_{s_\nu,t_\nu}(p_\nu,q_\nu)] = - \sum_{\substack{|s_\nu-t_\nu|\le r<m\\ m^2-r^2\in2\Z_{\ge0}}} a^{(m^2-r^2)/2} [\beta_r(p_\nu\otimes q_\nu)], \] where \( \kappa_{s_\nu,t_\nu}\in\C^\times,  \beta_r(p_\nu\otimes q_\nu)\in P_r,  r<m. \) By the strong induction hypothesis, \( [\beta_r(p_\nu\otimes q_\nu)]\in\mathcal M \) for every \(r<m\). Since \(\mathcal M\) is a \(\C[a]\)-submodule of \(R(U)\), the entire right-hand side belongs to \(\mathcal M\). As \(\kappa_{s_\nu,t_\nu}\neq0\), it follows that \( [\mu_{s_\nu,t_\nu}(p_\nu,q_\nu)] \in\mathcal M. \) Summing over \(\nu\) gives \([p]\in\mathcal M\). Hence \([P_m]\subset\mathcal M\), completing the induction.

\noindent\textbf{Step 2: reduction of arbitrary vectors.} By \eqref{eq:Udecomp}, \( U\cong\bigoplus_{m\ge0}P_m\otimes L(1,m^2). \) Since every vector of \(U\) is a finite sum of homogeneous vectors from these summands, it suffices to consider a homogeneous vector \( u\in P_m\otimes L(1,m^2) \) of conformal weight \(m^2+N\). Choose a basis \(p_1,\ldots,p_d\) of \(P_m\). Then \(u\) can be written as \( u=\sum_{i=1}^d w_i, \) where each \(w_i\) lies in the level-\(N\) descendant space generated by the Virasoro highest-weight vector \(p_i\). By PBW, each \(w_i\) is a linear combination of vectors of the form \[ L(-n_1)\cdots L(-n_\ell)p_i, \qquad n_j\ge1,\qquad \sum_{j=1}^{\ell}n_j=N. \] Lemma~\ref{lem:descendant} therefore gives \[ [w_i]\in \begin{cases} \C\,a^{N/2}[p_i],& N\ \text{even},\\ 0,& N\ \text{odd}. \end{cases} \] By Step~1, \([p_i]\in\mathcal M\) for every \(i\). Since \(\mathcal M\) is a \(\C[a]\)-submodule of \(R(U)\), it follows that \([w_i]\in\mathcal M\) for every \(i\), and hence \( [u]=\sum_{i=1}^d[w_i]\in\mathcal M. \) Thus every class in \(R(U)=U/C_2(U)\) belongs to \(\mathcal M\), so \( R(U)\subset\mathcal M. \) The reverse inclusion \(\mathcal M\subset R(U)\) is immediate from the definition of \(\mathcal M\). Therefore  $ R(U)=\mathcal M.$
\end{proof}

\section{Eliminating the weight-\texorpdfstring{$225$}{225} and weight-\texorpdfstring{$100$}{100} classes}
\label{sec:eliminate-yz}

We now sharpen the four-generator reduction of Theorem~\ref{thm:rank4}. The two exceptional generators of conformal weights $225$ and $100$ can be eliminated by suitable Clebsch–Gordan projections.

\subsection{The weight-\texorpdfstring{$225$}{225} class}

\begin{prop}\label{prop:z0}
In the  $C_2$-Poisson algebra $R(U)$, $z=0$.
\end{prop}

\begin{proof}
Dong--Griess--Ryba prove $p_{31}(W_{13}^I\otimes W_{21}^I)\neq0$ \cite[Theorem~2.1(ii)]{DGR}. In our indexing $E_6=W_{13}$, $E_{10}=W_{21}$, $E_{15}=W_{31}$, so $p_{6,10}^{\,15}(P_6\otimes P_{10})\neq0$. Since $P_6=\C X$, $P_{10}=\C Y$ with $X,Y\neq0$, we have $p_{6,10}^{\,15}(X\otimes Y)\neq0$. By Corollary~\ref{cor:fixed-mode}, the mode reaching the $r=15$ highest-weight line has index $n(6,10;15)=6^2+10^2-15^2-1=-90$. Thus the projection of $X_{-90}Y$ onto $P_{15}\otimes\C\mathbf v_{15}$ is nonzero; since $P_{15}=\C Z$, there exists $c_{15}\in\C^\times$ such that the level-zero $r=15$ component is $c_{15}Z$.

The Clebsch--Gordan decomposition gives $E_6\otimes E_{10}\cong\bigoplus_{r=4}^{16}E_r$, so only $4\le r\le16$ occur. The total weight is $\wt(X_{-90}Y)=36+100+90-1=225$. Since $P_r\otimes L(1,r^2)$ has lowest weight $r^2$, we need $r^2\le225$, i.e., $r\le15$; thus $r=16$ cannot contribute. By Lemma~\ref{lem:small-invariant-degrees}, the only nonzero $P_r$ with $4\le r\le15$ are $r=6,10,12,15$. Hence $X_{-90}Y=w_6+w_{10}+w_{12}+w_{15}$ with $w_r\in P_r\otimes L(1,r^2)$ of total weight $225$ and $w_{15}=c_{15}Z$. The levels are $N_6=189$, $N_{10}=125$, $N_{12}=81$, all odd; Lemma~\ref{lem:descendant} gives $[w_6]=[w_{10}]=[w_{12}]=0$. Since $-90\le-2$, Lemma~\ref{lem:deepmodes} implies $X_{-90}Y\in C_2(U)$, so $[X_{-90}Y]=0$. Taking classes yields $0=c_{15}z$, hence $z=0$.
\end{proof}

\subsection{A transvectant computation}

To eliminate $y$, we show the self-product of the degree-$6$ invariant has nonzero $r=10$ component. 
Under $E_6\cong\Sym^{12}(\C^2)$, choose coordinates on $\C^2$ so that the invariant line $P_6=E_6^I$ is spanned by the classical degree-$12$ binary icosahedral invariant
\begin{equation}\label{eq:Klein-f}
f(u,v)=uv(u^{10}+11u^5v^5-v^{10})
\end{equation}
\cite[Example~1.16]{Mukai}. 
Define the unnormalized $q$-th transvectant $$(F,G)_q=\sum_{j=0}^q(-1)^j\binom qj\frac{\partial^qF}{\partial u^{q-j}\partial v^j}\frac{\partial^qG}{\partial u^j\partial v^{q-j}},$$ which realizes the Clebsch--Gordan projection $\Sym^m\otimes\Sym^n\to\Sym^{m+n-2q}$ up to scalar. For $m=n=12$, $q=2$, the target is $\Sym^{20}\cong E_{10}$, representing $p_{6,6}^{\,10}$.

\begin{lem}\label{lem:transvectant-2}
For $f$ in \eqref{eq:Klein-f}, $(f,f)_2=242(-u^{20}+228u^{15}v^5-494u^{10}v^{10}-228u^5v^{15}-v^{20})\neq0$. Consequently, $p_{6,6}^{\,10}(X\otimes X)\neq0$.
\end{lem}

\begin{proof}
Expanding $f=u^{11}v+11u^6v^6-uv^{11}$, compute $f_{uu}=110u^9v+330u^4v^6$, $f_{uv}=11u^{10}+396u^5v^5-11v^{10}$, $f_{vv}=330u^6v^4-110uv^9$. Then $(f,f)_2=2(f_{uu}f_{vv}-f_{uv}^2)=242(-u^{20}+228u^{15}v^5-494u^{10}v^{10}-228u^5v^{15}-v^{20})\neq0$. Since $\dim\Hom_{SL_2(\C)}(E_6\otimes E_6,E_{10})=1$, this transvectant is a nonzero multiple of $p_{6,6}^{\,10}$, proving the claim.
\end{proof}

\subsection{The weight-\texorpdfstring{$100$}{100} class}

\begin{prop}\label{prop:y-elim}
There exist $\lambda,\mu\in\C$ such that $y=\lambda a^{50}+\mu a^{32}x$.
\end{prop}

\begin{proof}
By Lemma~\ref{lem:transvectant-2}, $p_{6,6}^{\,10}(X\otimes X)\neq0$. Corollary~\ref{cor:fixed-mode} applies to the $r=10$ channel. The mode index is $n(6,6;10)=36+36-100-1=-29$, so the $P_{10}$ highest-weight component of $X_{-29}X$ is nonzero; since $P_{10}=\C Y$, there exists $c_{10}\in\C^\times$ with level-zero component $c_{10}Y$. Total weight is $\wt(X_{-29}X)=100$. Clebsch--Gordan gives $E_6\otimes E_6\cong\bigoplus_{r=0}^{12}E_r$, but $r^2\le100$ forces $r\le10$. By Lemma~\ref{lem:small-invariant-degrees}, only $P_0,P_6,P_{10}$ are nonzero in this range. Thus $X_{-29}X=w_0+w_6+w_{10}$ with $w_{10}=c_{10}Y$. Levels: $N_0=100$ (even), $N_6=64$ (even). Lemma~\ref{lem:descendant} gives $[w_0]=c_0a^{50}$, $[w_6]=c_6a^{32}x$. Since $-29\le-2$, $X_{-29}X\in C_2(U)$, so $0=c_0a^{50}+c_6a^{32}x+c_{10}y$. Solving yields $y=-(c_0/c_{10})a^{50}-(c_6/c_{10})a^{32}x$.
\end{proof}

\begin{thm}[Two-generator reduction]\label{thm:rank2}
The intrinsic $C_2$-Poisson algebra is generated by two elements as a $\C[a]$-module:
\begin{equation}\label{eq:rank2}
R(U)=\C[a]\cdot1+\C[a]\cdot x.
\end{equation}
Moreover, $R(U)_{2N+1}=0$ for all $N\ge0$, and
\begin{equation}\label{eq:even-R-bound}
R(U)_{2N}\subseteq
\begin{cases}
\C a^N,& 0\le N<18,\\
\C a^N+\C a^{N-18}x,& N\ge18.
\end{cases}
\end{equation}
In particular, $\dim_\C R(U)_n\le2$ for all $n\ge0$.
\end{thm}

\begin{proof}
Theorem~\ref{thm:rank4} gives $R(U)=\C[a]\cdot1+\C[a]\cdot x+\C[a]\cdot y+\C[a]\cdot z$. Propositions~\ref{prop:z0} and \ref{prop:y-elim} give $z=0$ and $y=\lambda a^{50}+\mu a^{32}x\in\C[a]\cdot1+\C[a]\cdot x$, proving \eqref{eq:rank2}. Since $\wt a=2$ and $\wt x=36$, monomials in $\C[a]\cdot1$ have weight $2k$ and those in $\C[a]\cdot x$ have weight $36+2k=2(k+18)$, both even; hence $R(U)_{2N+1}=0$. For weight $2N$, the first summand contributes $\C a^N$; the second requires $36+2k=2N$, i.e., $k=N-18\ge0$, contributing $\C a^{N-18}x$ only when $N\ge18$. This proves \eqref{eq:even-R-bound} and the dimension bound.
\end{proof}

\begin{rmk}\label{rem:rank2-significance}
Theorem~\ref{thm:rank2} is the decisive structural reduction for $C_2$-cofiniteness: every invariant Virasoro primary class reduces in the $C_2$-quotient to an element of $\C[a]\cdot1+\C[a]\cdot x$. The remaining problem is controlled entirely by powers of $a=[\omega]$; it remains to prove that sufficiently high powers of $a$ and $a$ acting on $x$ vanish.
\end{rmk}

\section{The Virasoro \texorpdfstring{$C_2$}{C2}-symbol and the determinant criterion}
\label{sec:symbol-determinant}

After Theorem~\ref{thm:rank2}, the $C_2$-Poisson algebra is $R(U)=\C[a]\cdot1+\C[a]\cdot x$ with $a=[\omega]$, $x=[X]$. It remains to produce relations forcing high powers of $a$ and $a$ acting on $x$ to vanish. This section reduces that problem to the nonvanishing of a single $2\times2$ determinant, whose numerical evaluation follows later.

\subsection{The Virasoro \texorpdfstring{$C_2$}{C2}-symbol}

Let $\mathfrak n_-:=\bigoplus_{n\ge1}\C L(-n)$ be graded by $\deg L(-n)=n$.

\begin{lem}\label{lem:sigma}
There exists a unique algebra homomorphism
$\sigma:U(\mathfrak n_-)\to\C[T]$ such that
$\sigma(L(-2))=T$ and $\sigma(L(-n))=0$ for $n\neq2$, where
$\deg T=2$. Moreover, if $p\in U$ is a Virasoro highest-weight vector and $A\in U(\mathfrak n_-)$, then in $R(U)$, 
\begin{equation}\label{eq:sigma-C2}
[Ap]=\sigma(A)|_{T=a}\,[p].
\end{equation}
\end{lem}

\begin{proof}
Define
\(
\ell:\mathfrak n_-\longrightarrow \C[T]
\)
by
\(
\ell(L(-2))=T,
\ell(L(-n))=0
\quad (n\neq2),
\)
where the target is regarded as an abelian Lie algebra. Since
\[
[L(-m),L(-n)]=(n-m)L(-(m+n)),
\]
we have \(\ell([L(-m),L(-n)])=0\): indeed, the only possible case in which
\(\ell(L(-(m+n)))\neq0\) is \(m+n=2\), forcing \(m=n=1\), and then
\(n-m=0\). Hence \(\ell\) is a Lie algebra homomorphism. By the universal
property of the enveloping algebra, it extends uniquely to an algebra
homomorphism
\(
\sigma:U(\mathfrak n_-)\longrightarrow \C[T].
\)
It remains to prove \eqref{eq:sigma-C2}. By Lemma~\ref{lem:descendant},
modulo \(C_2(U)\) every PBW monomial acting on a Virasoro highest-weight
vector vanishes unless it is a power of \(L(-2)\), while
\(
[L(-2)^r p]=a^r[p].
\)
On the other hand, by definition,
\(
\sigma(L(-2)^r)=T^r,
\)
whereas every PBW monomial containing some \(L(-n)\) with \(n\neq2\)
is sent to zero by \(\sigma\). Therefore, by linearity, for every
\(A\in U(\mathfrak n_-)\),
\(
[Ap]=\sigma(A)\big|_{T=a}\,[p].
\)
\end{proof}

\begin{rmk}\label{rem:sigma-level}
If $A\in U(\mathfrak n_-)$ is homogeneous of level $N$, then $\sigma(A)=0$ for odd $N$, while for $N=2k$, $\sigma(A)\in\C T^k$ (only $L(-2)^k$ survives).
\end{rmk}

\subsection{The symbol on \texorpdfstring{$L(1,0)$ and $L(1,36)$}{L(1,0) and L(1,36)}}

For $h\in\{0,36\}$, let
\[
q_h:M(1,h)\longrightarrow L(1,h)
\]
be the canonical quotient map, and write
\(
v_h:=q_h(v_{1,h})
\)
for the image of the fixed highest-weight vector of the Verma module.
Thus $v_0$ is the vacuum vector of $L(1,0)$.

\begin{lem}\label{lem:module-symbol}
For each $h\in\{0,36\}$, there is a well-defined linear map
\[
\Sigma_h:L(1,h)\longrightarrow\C[T]
\]
characterized by
\begin{equation}\label{eq:Sigma-def}
\Sigma_h(Av_h)=\sigma(A),
\qquad A\in U(\mathfrak n_-).
\end{equation}
In particular,
\(
\Sigma_0(v_0)=1,
\Sigma_{36}(v_{36})=1.
\)
Moreover, if $h\in\{0,36\}$ and $w\in L(1,h)$ is homogeneous of
level $N$,
then
\(
\Sigma_h(w)=0\), if $N$ is odd,
while, if $N=2k$,
\(
\Sigma_h(w)\in\C T^k.
\)
\end{lem}

\begin{proof}
By the PBW theorem, for each $h\in\{0,36\}$ the map
\[
U(\mathfrak n_-)\longrightarrow M(1,h),
\qquad
A\longmapsto Av_{1,h},
\]
is an isomorphism of vector spaces. Hence there is a well-defined
linear map
\[
\widetilde\Sigma_h:M(1,h)\longrightarrow\C[T],
\qquad
\widetilde\Sigma_h(Av_{1,h})=\sigma(A).
\]
It remains to show that the maximal proper submodule of $M(1,h)$ is
contained in $\ker\widetilde\Sigma_h$.

For $h=0$, Proposition~\ref{prop:Milas-exact} with $m=0$ shows that
the maximal proper submodule of $M(1,0)$ is generated by its unique,
up to scalar, singular vector at level $1$. Since
\(
M(1,0)[1]=\C L(-1)v_{1,0},
\)
this singular vector is a nonzero scalar multiple of
$L(-1)v_{1,0}$. Hence
\(
\ker q_0=U(\mathfrak n_-)L(-1)v_{1,0}.
\)
Since $\sigma(L(-1))=0$ and $\sigma$ is multiplicative,
\(
\sigma(AL(-1))=0\)\(
(A\in U(\mathfrak n_-)),
\)
so $\widetilde\Sigma_0$ vanishes on $\ker q_0$.

For $h=36$, Proposition~\ref{prop:Milas-exact} with $m=12$ shows
that the maximal proper submodule of $M(1,36)$ is generated by a
nonzero singular vector of level $13$. By the PBW theorem, this
vector may be written as
\(
Sv_{1,36}
\)
for some homogeneous $S\in U(\mathfrak n_-)$ of level $13$.
Therefore
\(
\ker q_{36}=U(\mathfrak n_-)Sv_{1,36}.
\)
Since $13$ is odd, Remark~\ref{rem:sigma-level} gives
$\sigma(S)=0$. Thus, for every $A\in U(\mathfrak n_-)$,
\(
\sigma(AS)=\sigma(A)\sigma(S)=0,
\)
so $\widetilde\Sigma_{36}$ vanishes on $\ker q_{36}$.

Therefore $\widetilde\Sigma_h$ factors uniquely through the quotient
$q_h:M(1,h)\to L(1,h)$, giving a well-defined linear map
\[
\Sigma_h:L(1,h)\longrightarrow\C[T]
\]
satisfying \eqref{eq:Sigma-def}.
Equivalently, if $A,B\in U(\mathfrak n_-)$ satisfy
$Av_h=Bv_h$ in $L(1,h)$, then
\(
(A-B)v_{1,h}\in\ker q_h,
\)
and the preceding argument gives
\(
\sigma(A-B)=0.
\)
Thus $\sigma(A)=\sigma(B)$, which proves explicitly that
\eqref{eq:Sigma-def} is independent of the choice of representative.

Finally, if $w\in L(1,h)$ has level $N$, choose a homogeneous lift
$\widetilde w\in M(1,h)[N]$. By PBW,
\(
\widetilde w=Av_{1,h}
\)
for some homogeneous $A\in U(\mathfrak n_-)$ of level $N$.
Remark~\ref{rem:sigma-level} gives
\(
\sigma(A)=0
\)
when $N$ is odd, and
\(
\sigma(A)\in\C T^k
\)
when $N=2k$. Since $\Sigma_h(w)=\sigma(A)$, the asserted level
statements follow.
\end{proof}

\subsection{The vacuum and self symbols}

Retain the highest-weight vectors $v_0\in L(1,0)$ and
$v_{36}\in L(1,36)$ fixed in Section~7.2, and set
\(
\mathbf v_0:=v_0, \mathbf v_6:=v_{36}.
\)
Thus the normalizations of $\Sigma_0$ and $\Sigma_{36}$ are
\(
\Sigma_0(\mathbf v_0)=1,
\Sigma_{36}(\mathbf v_6)=1.
\)
By Lemma~\ref{lem:lowest-coeff}, choose nonzero intertwining operators
\[
\mathcal Y_0\in
I\binom{L(1,0)}{L(1,36)\;L(1,36)},
\qquad
\mathcal Y_6\in
I\binom{L(1,36)}{L(1,36)\;L(1,36)}
\]
and rescale them so that
\begin{align}
\mathcal Y_0(\mathbf v_6,z)\mathbf v_6
&=z^{-72}
  \bigl(\mathbf v_0+\text{positive-level terms}\bigr),
\label{eq:Y0-normalization}\\
\mathcal Y_6(\mathbf v_6,z)\mathbf v_6
&=z^{-36}
  \bigl(\mathbf v_6+\text{positive-level terms}\bigr).
\label{eq:Y6-normalization}
\end{align}

\begin{defi}\label{def:uv}
Define scalars $u_n,v_n\in\C$ ($n\ge0$) by
\begin{align}
\Sigma_0(\mathcal Y_0(\mathbf v_6,z)\mathbf v_6)&=z^{-72}\sum_{n\ge0}u_n(Tz^2)^n,\quad u_0=1,\label{eq:u-series}\\
\Sigma_{36}(\mathcal Y_6(\mathbf v_6,z)\mathbf v_6)&=z^{-36}\sum_{n\ge0}v_n(Tz^2)^n,\quad v_0=1.\label{eq:v-series}
\end{align}
\end{defi}

In the vacuum channel, level-$N$ coefficients occur at $z^{-72+N}$; $\Sigma_0$ kills odd $N$, and for $N=2n$ gives
$z^{-72}u_n(Tz^2)^n$. Similarly, in the self channel the leading power is $z^{-36}$. The chosen normalizations imply that the scalar initial values are
$
u_0=v_0=1.
$

Hereafter, the bold symbols $\mathbf v_0,\mathbf v_6$ denote the fixed highest-weight vectors inherited from Section~7.2, whereas $u_n,v_n$ denote the scalar sequences in Definition~\ref{def:uv}.

\subsection{The two VOA coupling constants}

By Lemma~\ref{lem:small-invariant-degrees},
$
P_0=\C\mathbf1, P_6=\C X.
$
Applying Proposition~\ref{prop:channel-separation} with $s=t=6$ and
$r=0,6$, and using the one-dimensionality of the corresponding
intertwining-operator spaces in \eqref{eq:fusion-vacuum}, we obtain
scalars $\gamma_0,\gamma_6\in\C$ such that
\begin{align}
\Pi_0Y(X,z)X
&=\gamma_0\mathbf{1}\otimes\mathcal Y_0(\mathbf v_6,z)\mathbf v_6,
\label{eq:actual-vacuum-channel}\\
\Pi_6Y(X,z)X
&=\gamma_6X\otimes
\mathcal Y_6(\mathbf v_6,z)\mathbf v_6.
\label{eq:actual-self-channel}
\end{align}

\begin{lem}\label{lem:couplings}
$\gamma_0\neq0$ and $\gamma_6\neq0$.
\end{lem}

\begin{proof}
\noindent\textbf{Vacuum channel.} The projection $p_{6,6}^{\,0}:E_6\otimes E_6\to E_0$ corresponds to $p_1:W_{13}\otimes W_{13}\to W_1$. By \cite[Lemma~3.5]{DGR}, $W_{13}$ carries a nonzero invariant bilinear form nondegenerate on the one-dimensional invariant line $W_{13}^I=P_6=\C X$; thus $p_{6,6}^{\,0}(X\otimes X)\neq0$. Corollary~\ref{cor:fixed-mode} implies the vacuum channel in $Y(X,z)X$ is nonzero, so $\gamma_0\neq0$.

\noindent\textbf{Self channel.} Under $E_6\cong\Sym^{12}(\C^2)$, $X$ corresponds to $f(u,v)=uv(u^{10}+11u^5v^5-v^{10})$. A direct sixth-transvectant calculation gives
\[
(f,f)_6=-26345088000\,f\neq0.
\]
Since $(f,f)_6$ is an invariant binary form of degree $12$ and $P_6=\C f$, this determines the displayed identity. The sixth transvectant realizes $p_{6,6}^{\,6}:E_6\otimes E_6\to E_6$ up to scalar, so $p_{6,6}^{\,6}(X\otimes X)\neq0$. Corollary~\ref{cor:fixed-mode} implies the self channel is nonzero, so $\gamma_6\neq0$.
\end{proof}

\subsection{The weight-\texorpdfstring{$74$}{74} and weight-\texorpdfstring{$76$}{76} relations}

\begin{lem}\label{lem:74-76-channels}
At total conformal weights $74$ and $76$, the only nonzero invariant
Virasoro-channel components of $Y(X,z)X$ that can occur are those with
$r=0$ and $r=6$.
\end{lem}

\begin{proof}
Clebsch--Gordan gives $E_6\otimes E_6=\bigoplus_{r=0}^{12}E_r$. For total weight $W\in\{74,76\}$, we need $r^2\le W<81=9^2$, so $0\le r\le8$. By Lemma~\ref{lem:small-invariant-degrees}, only $P_0$ and $P_6$ are nonzero in this range.
\end{proof}

\begin{prop}\label{prop:two-symbol-relations}
In $R(U)$,
\begin{align}
0=[X_{-3}X]&=\gamma_0u_{37}a^{37}+\gamma_6v_{19}a^{19}x,\label{eq:X-3}\\
0=[X_{-5}X]&=\gamma_0u_{38}a^{38}+\gamma_6v_{20}a^{20}x.\label{eq:X-5}
\end{align}
\end{prop}

\begin{proof}
By Lemma~\ref{lem:74-76-channels}, only the $r=0$ and $r=6$
channels can contribute at total conformal weights $74$ and $76$.
First consider $X_{-3}X$. Since
$
\wt(X_{-3}X)=36+36+3-1=74,
$
and the mode $-3$ is the coefficient of $z^2$ in $Y(X,z)X$,
the vacuum-channel contribution is obtained from
\eqref{eq:u-series} by
$
-72+2j=2,
$
hence $j=37$. By Lemma~\ref{lem:sigma} and the normalization
$\Sigma_0(\mathbf v_0)=1$, its class in $R(U)$ is
$
\gamma_0u_{37}a^{37}.
$
Similarly, in the self channel,
$
-36+2j=2,
$
so $j=19$, and Lemma~\ref{lem:sigma}, together with
$\Sigma_{36}(\mathbf v_6)=1$, gives the class
$
\gamma_6v_{19}a^{19}x.
$
Therefore
$$
[X_{-3}X]
=\gamma_0u_{37}a^{37}
+\gamma_6v_{19}a^{19}x.
$$
Since $X_{-3}X\in C_2(U)$ by Lemma~\ref{lem:deepmodes},
this class is zero, proving \eqref{eq:X-3}.

The same argument applies to $X_{-5}X$, whose coefficient in
$Y(X,z)X$ is that of $z^4$. The equations
$
-72+2j=4, -36+2j=4
$
give $j=38$ and $j=20$, respectively. Hence
$$
[X_{-5}X]
=\gamma_0u_{38}a^{38}
+\gamma_6v_{20}a^{20}x.
$$
Again $X_{-5}X\in C_2(U)$ by Lemma~\ref{lem:deepmodes}, so this class
vanishes, yielding \eqref{eq:X-5}.
\end{proof}

\subsection{The determinant criterion}

\begin{prop}[Determinant criterion]\label{prop:det-criterion}
Define $D:=u_{37}v_{20}-u_{38}v_{19}$. If $D\neq0$, then $a^{38}=0$ and $a^{20}x=0$. Consequently, $R(U)=\Span_\C\{1,a,\ldots,a^{37},x,ax,\ldots,a^{19}x\}$ and $\dim_\C R(U)\le58$. Hence $U$ is $C_2$-cofinite.
\end{prop}

\begin{proof}
Multiply \eqref{eq:X-3} by $a$: $\gamma_0u_{37}a^{38}+\gamma_6v_{19}a^{20}x=0$. Together with \eqref{eq:X-5}, this forms the system
\[
\begin{pmatrix}\gamma_0u_{37}&\gamma_6v_{19}\\\gamma_0u_{38}&\gamma_6v_{20}\end{pmatrix}
\begin{pmatrix}a^{38}\\a^{20}x\end{pmatrix}=0.
\]
The determinant is $\gamma_0\gamma_6D\neq0$ (Lemma~\ref{lem:couplings} and hypothesis), so $a^{38}=0$, $a^{20}x=0$. Theorem~\ref{thm:rank2} gives $R(U)=\C[a]\cdot1+\C[a]\cdot x$; nilpotency implies the stated spanning set of size $38+20=58$, hence $\dim_\C R(U)\le58$, i.e., $C_2$-cofiniteness.
\end{proof}

\begin{rmk}\label{rem:determinant-reduction}
Proposition~\ref{prop:det-criterion} isolates the sole remaining analytic-algebraic input: $D=u_{37}v_{20}-u_{38}v_{19}\neq0$. All prior results follow from $SL_2(\C)\times L(1,0)$ decomposition, binary icosahedral invariant theory, and intrinsic $C_2$-reduction. The remainder is purely within degenerate $c=1$ Virasoro theory: compute $u_{37},u_{38},v_{19},v_{20}$ and verify $D\neq0$.
\end{rmk}

\section{Global first-row forms and local positive-mode reconstruction}
\label{sec:local-regularization}

The generic first-row theory of Section~\ref{subsec:generic-first-row} does not apply directly at $t=1$, $\kappa=4$. We therefore separate the global parameter family, on which generic fibers exist, from its localization at $t=1$. Recall $\Lambda=\C[t,t^{-1}]$ from Section~\ref{subsec:BSA}, and set
\begin{equation}\label{eq:local-ring}
\mathcal R:=\Lambda_{(t-1)},\qquad
\mathcal K:=\operatorname{Frac}(\Lambda)=\C(t).
\end{equation}
For every $t_0\in\C^\times$, evaluation $t\mapsto t_0$ defines a ring homomorphism $\Lambda\to\C$. By contrast, $\mathcal R$ is used only near $t=1$: it consists of rational functions regular there, has maximal ideal $(t-1)\mathcal R$, and does not admit evaluation at $t_0\ne1$ in general. Generic fibers below are therefore fibers of $\Lambda$-modules; elements of localized modules are compared with them only by evaluating finitely many rational PBW coordinates away from their poles. Recall that $c(t)=13-6t-6t^{-1}$, $H(t)=42t^{-1}-6$, and all $\beta_{\mathbf p}(t)$ lie in $\Lambda$.

\subsection{Global Verma modules and first-row quotients}

Let
\(
\Vir_\Lambda:=\Lambda\otimes_\C \Vir
\)
be the scalar extension of the Virasoro Lie algebra from $\C$ to
$\Lambda$, regarded as a Lie algebra over $\Lambda$.  Thus its bracket is
the $\Lambda$-bilinear extension of the usual Virasoro bracket.  Set
\(
\mathfrak n_{-,\Lambda}
:=
\bigoplus_{n\ge1}\Lambda L(-n),
\)
and denote by
\(
U_\Lambda(\mathfrak n_-)
:=
U_\Lambda(\mathfrak n_{-,\Lambda})
\)
the universal enveloping algebra of $\mathfrak n_{-,\Lambda}$ over
$\Lambda$.

For $h(t)\in\Lambda$, let $M_\Lambda(c(t),h(t))$ be the Verma
$\Vir_\Lambda$-module generated by a highest-weight vector $v_{h(t)}$
subject to
\[
L(n)v_{h(t)}=0\qquad(n>0),
\]
\[
L(0)v_{h(t)}=h(t)v_{h(t)},\qquad
Cv_{h(t)}=c(t)v_{h(t)}.
\]
By the PBW theorem,
\(
M_\Lambda(c(t),h(t))
=
U_\Lambda(\mathfrak n_-)v_{h(t)}
\)
is free as a $\Lambda$-module, with basis
\(
\bigl\{
L(-\lambda)v_{h(t)}
\bigm|
\lambda \text{ a partition}
\bigr\},
\)
where, for
$\lambda=(\lambda_1,\ldots,\lambda_\ell)$ with
$\lambda_1\ge\cdots\ge\lambda_\ell\ge1$, we write
\(
L(-\lambda)
:=
L(-\lambda_1)\cdots L(-\lambda_\ell).
\)
Finally, write
\(
v_{12}:=v_{H(t)},
\mathbf 1:=v_0.
\)

\begin{lem}[Global BSA singularity]\label{lem:BSA-local-singular}
$S_{13}(t)v_{12}\in M_\Lambda(c(t),H(t))$ is a nonzero singular vector of level $13$.
\end{lem}

\begin{proof}
All coefficients of $S_{13}(t)$ lie in
$\Lambda=\C[t,t^{-1}]$.  For every admissible generic
$t_0\in\C^\times$, base change along the evaluation homomorphism
\(
\Lambda\longrightarrow\C, t\longmapsto t_0,
\)
gives a natural identification
\[
M_\Lambda(c(t),H(t))
\otimes_{\Lambda,t\mapsto t_0}\C
\cong
M(c(t_0),H(t_0)).
\]
Let
\(
v_{12}(t_0):=v_{12}\otimes 1
\)
denote the image of the global highest-weight vector in this specialized
Verma module.  Under the same base change, $S_{13}(t)$ specializes to
$S_{13}(t_0)$.  By the normalization comparison
\eqref{eq:BSA-normalization-comparison}, $S_{13}(t_0)$ is a nonzero
scalar multiple of the first-row Benoit--Saint-Aubin singular operator
for $\lambda=12$.  Hence
\(
L(m)S_{13}(t_0)v_{12}(t_0)=0 (m>0)
\)
by \cite[Lemma~4.1 and (4.12)]{KoshidaKytola}.

We now prove that the corresponding vector over $\Lambda$ is singular.
It is enough to consider $1\le m\le13$, since for $m>13$ the operator
$L(m)$ sends a level-$13$ vector to a negative level and therefore acts
by zero.  By PBW,
\[
L(m)S_{13}(t)v_{12}
=
\sum_{\lambda\vdash 13-m}
c_{\lambda,m}(t)L(-\lambda)v_{12},
\qquad
c_{\lambda,m}(t)\in\Lambda.
\]
Specializing this identity at an admissible generic $t_0$ gives
\[
L(m)S_{13}(t_0)v_{12}(t_0)
=
\sum_{\lambda\vdash 13-m}
c_{\lambda,m}(t_0)
L(-\lambda)v_{12}(t_0).
\]
The left-hand side vanishes by the preceding paragraph.  On the other
hand, the vectors
\(
\bigl\{
L(-\lambda)v_{12}(t_0)
\bigm|
\lambda\vdash 13-m
\bigr\}
\)
are linearly independent by the PBW theorem for the specialized Verma
module $M(c(t_0),H(t_0))$.  Therefore
\(
c_{\lambda,m}(t_0)=0
\)
for every partition $\lambda\vdash 13-m$ and every admissible generic
$t_0$.

There are infinitely many admissible generic values of $t_0$, while
each $c_{\lambda,m}(t)$ belongs to the Laurent polynomial ring
$\Lambda=\C[t,t^{-1}]$.  A nonzero Laurent polynomial has only finitely
many zeros in $\C^\times$.  Hence
\(
c_{\lambda,m}(t)=0
\)
identically for every $\lambda$ and $m>0$.  Thus
\(
L(m)S_{13}(t)v_{12}=0 (m>0),
\)
so $S_{13}(t)v_{12}$ is singular.

Finally, the coefficient of $L(-1)^{13}$ in $S_{13}(t)$ is
\(
\frac{t^{12}}{(12!)^2}\in\Lambda^\times.
\)
Since the PBW monomials $L(-\lambda)v_{12}$ form a free
$\Lambda$-basis of $M_\Lambda(c(t),H(t))$, this unit coefficient implies
\(
S_{13}(t)v_{12}\neq0.
\)
Therefore $S_{13}(t)v_{12}$ is a nonzero singular vector of level $13$.
\end{proof}

Define the global quotients
\begin{align}
Q_{12,\Lambda}&:=M_\Lambda(c(t),H(t))/U_\Lambda(\mathfrak n_-)S_{13}(t)v_{12},\label{eq:Q12A}\\
Q_{0,\Lambda}&:=M_\Lambda(c(t),0)/U_\Lambda(\mathfrak n_-)L(-1)\mathbf1.\label{eq:Q0A}
\end{align}
Here \(L(-1)\mathbf1\in M_\Lambda(c(t),0)\) is the nonzero
level-\(1\) singular vector, since
\[
L(1)L(-1)\mathbf1=2L(0)\mathbf1=0,
\qquad
L(n)L(-1)\mathbf1=(n+1)L(n-1)\mathbf1=0
\quad (n\ge2).
\]
Thus the denominators in \eqref{eq:Q12A} and \eqref{eq:Q0A}
are Virasoro submodules, by Lemma~\ref{lem:BSA-local-singular}
and the preceding observation, respectively. For \(t_0\in\C^\times\), let \(\C_{t_0}\) denote \(\C\) regarded as a \(\Lambda\)-algebra via the specialization \(t\mapsto t_0\), and set
$
Q_{\tau,\Lambda}(t_0)
:=Q_{\tau,\Lambda}\otimes_\Lambda \C_{t_0}.
$

By right exactness of base change, \(Q_{\tau,\Lambda}(t_0)\) is naturally identified with the quotient of the specialized Verma module by the image of the corresponding specialized singular-vector submodule. Hence, if \(t_0\) is admissible generic, \eqref{eq:BSA-normalization-comparison} yields
$
Q_{12,\Lambda}(t_0)\cong Q_{12}(4t_0),
$
and similarly
$
Q_{0,\Lambda}(t_0)\cong Q_0(4t_0).
$

\subsection{Global PBW bases and localization}

For a partition $\lambda$, let $m_1(\lambda)$ be the multiplicity of part $1$, and let $p(N)$ be the number of partitions of $N$ ($p(N)=0$ for $N<0$).

\begin{lem}[Freeness]\label{lem:freeness}
For every $N$, the set
\[
\mathcal B_N=\{L(-\lambda)v_{12}\mid\lambda\vdash N,\ m_1(\lambda)\le12\}
\]
is a $\Lambda$-basis of $Q_{12,\Lambda}[N]$, of rank $p(N)-p(N-13)$. Likewise
$\mathcal B_N^{(0)}=\{L(-\lambda)\mathbf1\mid\lambda\vdash N,\ m_1(\lambda)=0\}$
is a $\Lambda$-basis of $Q_{0,\Lambda}[N]$, of rank $p(N)-p(N-1)$.
After localization,
\begin{equation}\label{eq:localized-quotients}
Q_{\tau,\mathcal R}:=\mathcal R\otimes_\Lambda Q_{\tau,\Lambda}
\end{equation}
has the same PBW bases over $\mathcal R$ in every degree.
\end{lem}

\begin{proof}
\noindent\textbf{Step 1: spanning over $\Lambda$.} In $S_{13}(t)$, the term indexed by $(1^{13})$ is $c(t)L(-1)^{13}$ with $c(t)=t^{12}/(12!)^2\in\Lambda^\times$; every other PBW monomial has at most $12$ factors $L(-1)$. Thus the quotient relation solves $L(-1)^{13}v_{12}$ as a $\Lambda$-linear combination with fewer $L(-1)$ factors. More precisely, let an ordered PBW monomial contain $r\ge13$ factors $L(-1)$ and write it as $A L(-1)^{13}v_{12}$, where $A$ contains the remaining $r-13$ factors $L(-1)$. After substituting the singular-vector relation, every resulting term contains at most $(r-13)+12=r-1$ factors $L(-1)$. When it is reordered into PBW form, each commutator is of the form $[L(-m),L(-n)]=(n-m)L(-(m+n))$ with $m+n\ge2$, and therefore creates no new $L(-1)$. The number of $L(-1)$ factors thus decreases strictly at every elimination, so the process terminates. It follows that $\mathcal B_N$ spans $Q_{12,\Lambda}[N]$. Its cardinality is $p(N)-p(N-13)$.

\noindent\textbf{Step 2: independence over \(\Lambda\).}
For \(t_0\in\C^\times\), let
$$
\mathcal B_N(t_0)
:=
\{\,L(-\lambda)v_{12}\otimes 1
\mid \lambda\vdash N,\ m_1(\lambda)\le12\,\}
\subset Q_{12,\Lambda}(t_0)[N]
$$
denote the specialization of \(\mathcal B_N\). Since \(\mathcal B_N\) spans \(Q_{12,\Lambda}[N]\) by Step~1, tensoring the corresponding surjection
$
\bigoplus_{b\in\mathcal B_N}\Lambda b
\twoheadrightarrow Q_{12,\Lambda}[N]
$
with \(\C_{t_0}\) shows that
$
\mathcal B_N(t_0)
=\{\,b\otimes1\mid b\in\mathcal B_N\,\}
$
spans \(Q_{12,\Lambda}(t_0)[N]\) for every \(t_0\in\C^\times\).
Suppose, toward a contradiction, that \(\mathcal B_N\) is not \(\Lambda\)-linearly independent. Then there exist Laurent polynomials
$
a_\lambda(t)\in\Lambda,
\lambda\vdash N, m_1(\lambda)\le12,
$
not all zero, such that
$$
\sum_{\substack{\lambda\vdash N\\m_1(\lambda)\le12}}
a_\lambda(t)L(-\lambda)v_{12}=0
\qquad\text{in }Q_{12,\Lambda}[N].
$$
Choose an admissible generic \(t_0\in\C^\times\) such that not all
\(a_\lambda(t_0)\) vanish. After base change along
\(\Lambda\to\C\), \(t\mapsto t_0\), this relation specializes to the nontrivial \(\C\)-linear relation
$$
\sum_{\substack{\lambda\vdash N\\m_1(\lambda)\le12}}
a_\lambda(t_0)\bigl(L(-\lambda)v_{12}\otimes1\bigr)=0
\qquad\text{in }Q_{12,\Lambda}(t_0)[N].
$$
On the other hand, for admissible generic \(t_0\) we have
$
Q_{12,\Lambda}(t_0)\cong Q_{12}(4t_0),
$
and \cite[Lemma~4.1]{KoshidaKytola} gives
$
\dim_\C Q_{12,\Lambda}(t_0)[N]
=p(N)-p(N-13).
$
The spanning family \(\mathcal B_N(t_0)\) is indexed by exactly
\(p(N)-p(N-13)\) partitions. Hence it is a basis of
\(Q_{12,\Lambda}(t_0)[N]\), contradicting the nontrivial specialized relation above. Therefore \(\mathcal B_N\) is \(\Lambda\)-linearly independent. Together with Step~1, this proves that \(\mathcal B_N\) is a \(\Lambda\)-basis of \(Q_{12,\Lambda}[N]\).

\noindent\textbf{Step 3: the vacuum quotient.}
Let
$
\mathcal C_N
:=
\{L(-\lambda)\mathbf1\mid \lambda\vdash N,\ m_1(\lambda)\ge1\}.
$
We first observe that
$
\bigl(U_\Lambda(\mathfrak n_-)L(-1)\mathbf1\bigr)[N]
=
\operatorname{Span}_\Lambda\mathcal C_N.
$
Indeed, if \(m_1(\lambda)\ge1\), then
$
L(-\lambda)\mathbf1
=
L(-\mu)L(-1)\mathbf1
$
for a partition \(\mu\vdash N-1\), so every element of \(\mathcal C_N\) lies in
\(U_\Lambda(\mathfrak n_-)L(-1)\mathbf1\). Conversely, after expressing any
\(D\in U_\Lambda(\mathfrak n_-)\) in the PBW basis, every term of
\(DL(-1)\mathbf1\) is an ordered PBW monomial containing at least one
factor \(L(-1)\). Hence it lies in \(\operatorname{Span}_\Lambda\mathcal C_N\).
Since the full PBW monomials
$
\{L(-\lambda)\mathbf1\mid\lambda\vdash N\}
$
form a \(\Lambda\)-basis of \(M_\Lambda(c(t),0)[N]\), we therefore have the direct-sum decomposition
$$
M_\Lambda(c(t),0)[N]
=
\operatorname{Span}_\Lambda\mathcal B_N^{(0)}
\oplus
\bigl(U_\Lambda(\mathfrak n_-)L(-1)\mathbf1\bigr)[N].
$$
Passing to the quotient shows that the images of
$
\mathcal B_N^{(0)}
=
\{L(-\lambda)\mathbf1\mid\lambda\vdash N,\ m_1(\lambda)=0\}
$
form a \(\Lambda\)-basis of \(Q_{0,\Lambda}[N]\). Their number is
$
p(N)-p(N-1).
$

\noindent\textbf{Step 4: localization.}
Since \(Q_{\tau,\Lambda}[N]\) is free over \(\Lambda\) with the basis obtained above, base change from \(\Lambda\) to \(\mathcal R\) gives
$
Q_{\tau,\mathcal R}[N]
=
\mathcal R\otimes_\Lambda Q_{\tau,\Lambda}[N],
$
and the corresponding vectors \(1\otimes b\), with \(b\) ranging over the appropriate PBW basis above, form an \(\mathcal R\)-basis. Thus \(Q_{\tau,\mathcal R}\) has the asserted PBW bases in every degree.
\end{proof}

For later use, put
\[
\Vir_{\mathcal R}:=\mathcal R\otimes_\Lambda\Vir_\Lambda,\qquad
U_{\mathcal R}(\mathfrak n_-)
:=\mathcal R\otimes_\Lambda U_\Lambda(\mathfrak n_-),\qquad
M_{\mathcal R}(c(t),h(t))
:=\mathcal R\otimes_\Lambda M_\Lambda(c(t),h(t)).
\]
Thus \(Q_{\tau,\mathcal R}\) is the localization of the corresponding global quotient, rather than a new quotient defined independently over \(\mathcal R\). Since localization is flat, applying \(\mathcal R\otimes_\Lambda-\) to the defining quotient sequences gives
$$
Q_{12,\mathcal R}
\cong
M_{\mathcal R}(c(t),H(t))
\big/
U_{\mathcal R}(\mathfrak n_-)S_{13}(t)v_{12},
$$
and
$$
Q_{0,\mathcal R}
\cong
M_{\mathcal R}(c(t),0)
\big/
U_{\mathcal R}(\mathfrak n_-)L(-1)\mathbf1.
$$
In particular, localization preserves the global quotient presentations.

\subsection{The special fiber}

\begin{prop}\label{prop:special-fibers}
There are natural Virasoro-module isomorphisms \begin{equation}\label{eq:special-Q12}
Q_{12,\mathcal R}/(t-1)Q_{12,\mathcal R}
\cong
L(1,36),
\end{equation}
and 
\begin{equation}\label{eq:special-Q0}
Q_{0,\mathcal R}/(t-1)Q_{0,\mathcal R}
\cong
L(1,0).
\end{equation}
\end{prop}

\begin{proof}
Let
\(
\C_1:=\mathcal R/(t-1)\mathcal R .
\)
Since
\(
c(1)=1, H(1)=36,
\)
the highest-weight relations defining
\(M_{\mathcal R}(c(t),H(t))\) specialize, under
\(
\mathcal R\longrightarrow\C_1\cong\C,
 t\longmapsto1,
\)
to those defining \(M(1,36)\).
Moreover, by the PBW theorem,
\(M_{\mathcal R}(c(t),H(t))\) is free over \(\mathcal R\) with basis
\(
\bigl\{
L(-\lambda)v_{12}
\bigm|
\lambda\text{ a partition}
\bigr\}.
\)
Therefore specialization preserves this PBW basis and yields a
natural isomorphism
\[
M_{\mathcal R}(c(t),H(t))
\otimes_{\mathcal R}\C_1
\cong
M(1,36),
\qquad
v_{12}\otimes1\longmapsto v_{1,36}.
\]
Let
\(
v_{1,36}:=v_{12}\otimes1
\)
denote the image of the global highest-weight vector under this
identification.
We first show that the specialization
\(
S_{13}(1)v_{1,36}\in M(1,36)
\)
is a nonzero singular vector of level \(13\).  Notice that this does
not follow by applying the generic Benoit--Saint-Aubin formula directly
at \(t=1\).  Instead, we use the global identity established in
Lemma~\ref{lem:BSA-local-singular}.  Namely, that lemma proves in the
\(\Lambda\)-Verma module \(M_\Lambda(c(t),H(t))\) that
\(
L(m)S_{13}(t)v_{12}=0
 (m>0).
\)
After extension of scalars from \(\Lambda\) to \(\mathcal R\), the same
identity holds in \(M_{\mathcal R}(c(t),H(t))\).  Specializing at
\(t=1\) therefore gives
\[
L(m)S_{13}(1)v_{1,36}
=
\bigl(L(m)S_{13}(t)v_{12}\bigr)\otimes1
=0
\qquad (m>0).
\]
Hence \(S_{13}(1)v_{1,36}\) is singular.

It remains to check that this vector does not vanish under
specialization.  By \eqref{eq:BSA-Lminusone-coefficient}, the coefficient of
\(L(-1)^{13}\) in \(S_{13}(t)\) is
\(
\frac{t^{12}}{(12!)^2}.
\)
Hence its specialization at \(t=1\) is
\(
\frac{1}{(12!)^2}\neq0.
\)
Since the PBW monomials
\(
\{L(-\lambda)v_{1,36}\mid\lambda\vdash13\}
\)
are linearly independent in \(M(1,36)\), we obtain
\(
S_{13}(1)v_{1,36}\neq0.
\)
Thus \(S_{13}(1)v_{1,36}\) is a nonzero level-\(13\) singular vector
of \(M(1,36)\).

Now Proposition~\ref{prop:Milas-exact}, with \(m=12\), gives the exact
sequence
\[
0\longrightarrow M(1,49)
\longrightarrow M(1,36)
\longrightarrow L(1,36)
\longrightarrow0,
\]
and states that the maximal proper submodule of \(M(1,36)\) is
generated by its unique, up to a nonzero scalar, singular vector at
level \(13\).  Since \(S_{13}(1)v_{1,36}\) is such a vector, we obtain
\[
U(\mathfrak n_-)S_{13}(1)v_{1,36}
=
\operatorname{Ker}\bigl(M(1,36)\twoheadrightarrow L(1,36)\bigr).
\]
Hence
\[
M(1,36)\big/
U(\mathfrak n_-)S_{13}(1)v_{1,36}
\cong L(1,36).
\]
Since
\(
\C_1=\mathcal R/(t-1)\mathcal R,
\)
the standard identification
\(
M\otimes_{\mathcal R}\mathcal R/(t-1)\mathcal R
\cong M/(t-1)M
\)
gives
\[
Q_{12,\mathcal R}/(t-1)Q_{12,\mathcal R}
\cong
Q_{12,\mathcal R}\otimes_{\mathcal R}\C_1.
\]
Put
\(
N_{\mathcal R}
:=
U_{\mathcal R}(\mathfrak n_-)
S_{13}(t)v_{12}.
\)
From the quotient presentation
\(
Q_{12,\mathcal R}
\cong
M_{\mathcal R}(c(t),H(t))/N_{\mathcal R},
\)
right exactness of tensor product gives
\[
Q_{12,\mathcal R}\otimes_{\mathcal R}\C_1
\cong
\frac{
M_{\mathcal R}(c(t),H(t))
\otimes_{\mathcal R}\C_1
}{
\operatorname{Im}
\bigl(
N_{\mathcal R}\otimes_{\mathcal R}\C_1
\to
M_{\mathcal R}(c(t),H(t))
\otimes_{\mathcal R}\C_1
\bigr)
}.
\]
Under the natural identification
\[
M_{\mathcal R}(c(t),H(t))
\otimes_{\mathcal R}\C_1
\cong M(1,36),
\]
the image in the denominator is precisely
\(
U(\mathfrak n_-)S_{13}(1)v_{1,36}.
\)
Indeed, specialization sends
\[
u(t)S_{13}(t)v_{12}
\longmapsto
u(1)S_{13}(1)v_{1,36}
\]
for \(u(t)\in U_{\mathcal R}(\mathfrak n_-)\), and every element of
\(U(\mathfrak n_-)S_{13}(1)v_{1,36}\) arises in this way.
Consequently,
\[
Q_{12,\mathcal R}\otimes_{\mathcal R}\C_1
\cong
M(1,36)\big/
U(\mathfrak n_-)S_{13}(1)v_{1,36}\cong L(1,36).
\]
For the vacuum quotient, the same specialization gives
\[
M_{\mathcal R}(c(t),0)\otimes_{\mathcal R}\C_1
\cong M(1,0),
\]
with \(\mathbf1\otimes1\) identified with the highest-weight vector
\(\mathbf1\) of \(M(1,0)\).  Therefore
\[
Q_{0,\mathcal R}/(t-1)Q_{0,\mathcal R}
\cong
M(1,0)\big/U(\mathfrak n_-)L(-1)\mathbf1.
\]
Proposition~\ref{prop:Milas-exact}, now with \(m=0\), states that
\(L(-1)\mathbf1\) generates the maximal proper submodule of \(M(1,0)\).
Consequently,
\(
M(1,0)\big/U(\mathfrak n_-)L(-1)\mathbf1
\cong L(1,0),
\)
and hence
\(
Q_{0,\mathcal R}/(t-1)Q_{0,\mathcal R}
\cong L(1,0).
\)
\end{proof}

\subsection{The symbol over the local family}

Set
$
\mathfrak n_{-,\mathcal R}
:=
\mathcal R\otimes_\Lambda\mathfrak n_{-,\Lambda}
=
\bigoplus_{n\ge1}\mathcal R L(-n),
$
so that \(U_{\mathcal R}(\mathfrak n_-)\) is canonically the universal enveloping algebra of \(\mathfrak n_{-,\mathcal R}\) over \(\mathcal R\). Define an \(\mathcal R\)-linear map
$
\ell_{\mathcal R}:\mathfrak n_{-,\mathcal R}\longrightarrow \mathcal R[T]
$
by
$
\ell_{\mathcal R}(L(-2))=T,
\ell_{\mathcal R}(L(-n))=0(n\neq2),
$
where \(\mathcal R[T]\) is graded by \(\deg T=2\) (with \(\mathcal R\) in degree \(0\)) and is regarded as an abelian Lie algebra over \(\mathcal R\). Since
$$
[L(-m),L(-n)]=(n-m)L(-(m+n)),
$$
we have \(\ell_{\mathcal R}([L(-m),L(-n)])=0\): indeed, the only possible case in which \(\ell_{\mathcal R}(L(-(m+n)))\neq0\) is \(m+n=2\), which forces \(m=n=1\), and then \(n-m=0\). Thus \(\ell_{\mathcal R}\) is a Lie algebra homomorphism. By the universal property of the enveloping algebra, it extends uniquely to an \(\mathcal R\)-algebra homomorphism
$
\sigma_{\mathcal R}:U_{\mathcal R}(\mathfrak n_-)\longrightarrow\mathcal R[T]
$
satisfying
$
\sigma_{\mathcal R}(L(-2))=T,
\sigma_{\mathcal R}(L(-n))=0(n\neq2).
$

\begin{lem}[Symbol descent]\label{lem:symbol-descent}
Let
$
q_{12,\mathcal R}:
M_{\mathcal R}(c(t),H(t))
\longrightarrow Q_{12,\mathcal R},
q_{0,\mathcal R}:
M_{\mathcal R}(c(t),0)
\longrightarrow Q_{0,\mathcal R}
$
be the quotient maps. Then there exist unique \(\mathcal R\)-linear maps
$$
\Sigma_{12,\mathcal R}:Q_{12,\mathcal R}\longrightarrow\mathcal R[T],
\qquad
\Sigma_{0,\mathcal R}:Q_{0,\mathcal R}\longrightarrow\mathcal R[T]
$$
such that, for every \(D\in U_{\mathcal R}(\mathfrak n_-)\),
$
\Sigma_{12,\mathcal R}
\bigl(q_{12,\mathcal R}(Dv_{12})\bigr)
=
\sigma_{\mathcal R}(D),
$
and
$
\Sigma_{0,\mathcal R}
\bigl(q_{0,\mathcal R}(D\mathbf1)\bigr)
=
\sigma_{\mathcal R}(D).
$

\end{lem}

\begin{proof}
By the PBW theorem, the maps
$$
U_{\mathcal R}(\mathfrak n_-)
\longrightarrow M_{\mathcal R}(c(t),H(t)),
\qquad
D\longmapsto Dv_{12},
$$
and
$$
U_{\mathcal R}(\mathfrak n_-)
\longrightarrow M_{\mathcal R}(c(t),0),
\qquad
D\longmapsto D\mathbf1,
$$
are isomorphisms of \(\mathcal R\)-modules. Hence the formulas
$$
\widetilde\Sigma_{12,\mathcal R}(Dv_{12})
:=
\sigma_{\mathcal R}(D),
\qquad
\widetilde\Sigma_{0,\mathcal R}(D\mathbf1)
:=
\sigma_{\mathcal R}(D)
$$
define well-defined \(\mathcal R\)-linear maps on the two Verma modules.
We first consider the self channel. By definition,
$$
Q_{12,\mathcal R}
=
M_{\mathcal R}(c(t),H(t))
\big/
U_{\mathcal R}(\mathfrak n_-)S_{13}(t)v_{12}.
$$
It is therefore enough to show that
$
\widetilde\Sigma_{12,\mathcal R}
\left(
U_{\mathcal R}(\mathfrak n_-)S_{13}(t)v_{12}
\right)=0.
$
Recall that
$$
S_{13}(t)
=
\sum_{\mathbf p\models13}
\beta_{\mathbf p}(t)
L(-p_1)\cdots L(-p_k).
$$
For a monomial
\(L(-p_1)\cdots L(-p_k)\)
to have nonzero \(\sigma_{\mathcal R}\)-image, every factor must be
\(L(-2)\), since
$$
\sigma_{\mathcal R}(L(-2))=T,
\qquad
\sigma_{\mathcal R}(L(-n))=0
\quad(n\neq2).
$$
But this is impossible for a composition
\(\mathbf p\models13\), because
\(p_1+\cdots+p_k=13\)
is odd. Consequently,
$
\sigma_{\mathcal R}(S_{13}(t))=0.
$
Thus, for every
\(A\in U_{\mathcal R}(\mathfrak n_-)\),
the multiplicativity of \(\sigma_{\mathcal R}\) gives
$$
\widetilde\Sigma_{12,\mathcal R}
\bigl(A S_{13}(t)v_{12}\bigr)
=
\sigma_{\mathcal R}\bigl(A S_{13}(t)\bigr)
=
\sigma_{\mathcal R}(A)
\sigma_{\mathcal R}(S_{13}(t))
=
0.
$$
Hence \(\widetilde\Sigma_{12,\mathcal R}\) vanishes on the submodule by which we quotient and therefore factors uniquely through an \(\mathcal R\)-linear map
$
\Sigma_{12,\mathcal R}:
Q_{12,\mathcal R}\longrightarrow\mathcal R[T].
$
For the vacuum channel,
$$
Q_{0,\mathcal R}
=
M_{\mathcal R}(c(t),0)
\big/
U_{\mathcal R}(\mathfrak n_-)L(-1)\mathbf1.
$$
Since
$
\sigma_{\mathcal R}(L(-1))=0,
$
for every \(A\in U_{\mathcal R}(\mathfrak n_-)\) we have
$$
\widetilde\Sigma_{0,\mathcal R}
\bigl(A L(-1)\mathbf1\bigr)
=
\sigma_{\mathcal R}\bigl(A L(-1)\bigr)
=
\sigma_{\mathcal R}(A)
\sigma_{\mathcal R}(L(-1))
=
0.
$$
Thus \(\widetilde\Sigma_{0,\mathcal R}\) also factors uniquely through
$
\Sigma_{0,\mathcal R}:
Q_{0,\mathcal R}\longrightarrow\mathcal R[T].
$
The stated formulas follow immediately from the definitions of the two factorizations.
\end{proof}

\begin{cor}[Compatibility with the special fiber]\label{cor:symbol-specialization}
Under the natural identifications of Proposition~\ref{prop:special-fibers},
the reductions of the local symbol maps modulo \((t-1)\) agree with the
symbol maps of Lemma~\ref{lem:sigma}.  Equivalently,
\[
\left.\Sigma_{0,\mathcal R}\right|_{t=1}=\Sigma_0,
\qquad
\left.\Sigma_{12,\mathcal R}\right|_{t=1}=\Sigma_{36},
\]
where the left-hand sides denote the maps induced after tensoring with
\(\mathcal R/(t-1)\mathcal R\cong\C\).
\end{cor}

\begin{proof}
The specialization of \(\sigma_{\mathcal R}\) at \(t=1\) is precisely the
\(\C\)-algebra homomorphism \(\sigma\) of Section~7.1, since both send
\(L(-2)\) to \(T\) and every \(L(-n)\), \(n\neq2\), to \(0\).
By Proposition~\ref{prop:special-fibers}, the special fibers of
\(Q_{0,\mathcal R}\) and \(Q_{12,\mathcal R}\) are naturally
\(L(1,0)\) and \(L(1,36)\), respectively, with the highest-weight vectors
and their PBW descendants identified in the evident way.  Hence the
specialized maps satisfy the defining formulas of Lemma~\ref{lem:sigma},
and therefore coincide with \(\Sigma_0\) and \(\Sigma_{36}\), respectively.
\end{proof}

\subsection{Positive-mode rigidity}\label{subsec:positive-mode-rigidity}

For \(\tau\in\{0,12\}\), set \(K_\tau(t)=0\) if \(\tau=0\) and \(K_\tau(t)=H(t)\) if \(\tau=12\), and write \(Q_{\tau,\mathcal R}\) accordingly. For \(N>0\), define
$$
A_N^{(\tau)}(t):
Q_{\tau,\mathcal R}[N]
\longrightarrow
\bigoplus_{m=1}^N Q_{\tau,\mathcal R}[N-m]
$$
by
$$
A_N^{(\tau)}(t)(w)
=
\bigl(L(1)w,\ldots,L(N)w\bigr).
$$
Let
\(
\mathbb C_1:=\mathcal R/(t-1)\mathcal R.
\)
Since $t$ has degree $0$, the submodule
$(t-1)Q_{\tau,\mathcal R}$ is graded. Set
$ Q_{\tau,1} := Q_{\tau,\mathcal R}\otimes_{\mathcal R}\mathbb C_1  $ and \(
Q_{\tau,1}[j]
:=
Q_{\tau,\mathcal R}[j]\otimes_{\mathcal R}\mathbb C_1 .
\)
Hence, for every $j\geq0$,
\[
Q_{\tau,1}[j]=Q_{\tau,\mathcal R}[j]\otimes_{\mathcal R}\mathbb C_1
\cong
Q_{\tau,\mathcal R}[j]/(t-1)Q_{\tau,\mathcal R}[j]
\cong
\bigl(Q_{\tau,\mathcal R}/(t-1)Q_{\tau,\mathcal R}\bigr)[j].
\]
The specialization of \(A_N^{(\tau)}(t)\) at \(t=1\) is the map
$$
A_N^{(\tau)}(1)
:=
A_N^{(\tau)}(t)\otimes_{\mathcal R}\mathbb C_1:
Q_{\tau,1}[N]
\longrightarrow
\bigoplus_{m=1}^N Q_{\tau,1}[N-m],
$$
given explicitly by
$$
A_N^{(\tau)}(1)(\overline w)
=
\bigl(
\overline{L(1)w},
\ldots,
\overline{L(N)w}
\bigr),
$$
where \(\overline w\) denotes the image of \(w\) in the corresponding special fiber. By Proposition~\ref{prop:special-fibers},
$$
Q_{0,\mathcal R}/(t-1)Q_{0,\mathcal R}\cong L(1,0),
\qquad
Q_{12,\mathcal R}/(t-1)Q_{12,\mathcal R}\cong L(1,36),
$$
so, under these identifications, \(A_N^{(\tau)}(1)\) is the positive-mode map on the level-\(N\) subspace of \(L(1,0)\) or \(L(1,36)\), respectively.

\begin{lem}[Positive-mode rigidity]\label{lem:positive-rigidity}
For every \(N>0\) and \(\tau\in\{0,12\}\), the special-fiber map
$$
A_N^{(\tau)}(1):
Q_{\tau,1}[N]
\longrightarrow
\bigoplus_{m=1}^N Q_{\tau,1}[N-m]
$$
is injective.
\end{lem}

\begin{proof}
The special fiber is $L(1,0)$ or $L(1,36)$ by Proposition~\ref{prop:special-fibers}. If $w\in\ker A_N^{(\tau)}(1)$, then $L(m)w=0$ for $1\le m\le N$; for $m>N$, $L(m)w=0$ automatically. Thus $w$ is singular. If $w\neq0$, the submodule it generates has weights strictly above the highest weight (since $N>0$), contradicting irreducibility. Hence $w=0$.
\end{proof}

Let
$
d_N^{(\tau)}
:=
\operatorname{rank}_{\mathcal R}Q_{\tau,\mathcal R}[N].
$
By Lemma~\ref{lem:freeness}, the graded piece
\(Q_{\tau,\mathcal R}[N]\) is a free \(\mathcal R\)-module of rank
\(d_N^{(\tau)}\), and the PBW bases there give bases for both the
domain and the graded pieces occurring in the codomain of
\(A_N^{(\tau)}(t)\). Thus, with respect to these bases,
\(A_N^{(\tau)}(t)\) is represented by a matrix over \(\mathcal R\)
having \(d_N^{(\tau)}\) columns.

After reduction modulo the maximal ideal
\((t-1)\mathcal R\), this matrix represents the special-fiber map
$$
A_N^{(\tau)}(1):
Q_{\tau,1}[N]
\longrightarrow
\bigoplus_{m=1}^N Q_{\tau,1}[N-m].
$$
With respect to the induced PBW bases on the special fibers, reduction modulo $(t-1)\mathcal R$ amounts to evaluating every matrix entry at $t=1$. Thus the matrix of $A_N^{(\tau)}(1)$ is obtained from that of $A_N^{(\tau)}(t)$ by setting $t=1$. 

By Lemma~\ref{lem:positive-rigidity}, this map is injective. Since
$
\dim_{\mathbb C}Q_{\tau,1}[N]=d_N^{(\tau)},
$
its matrix has column rank \(d_N^{(\tau)}\). Consequently, one can
choose \(d_N^{(\tau)}\) rows whose corresponding
\(d_N^{(\tau)}\times d_N^{(\tau)}\) minor has nonzero determinant
after specialization at \(t=1\).

Let
$$
\pi_N^{(\tau)}:
\bigoplus_{m=1}^N Q_{\tau,\mathcal R}[N-m]
\longrightarrow
\mathcal R^{d_N^{(\tau)}}
$$
be the coordinate projection onto these chosen rows, and set
$
B_N^{(\tau)}
:=
\pi_N^{(\tau)}\circ A_N^{(\tau)}(t).
$
Then \(B_N^{(\tau)}\) is represented by the selected square minor,
and hence
$
\det B_N^{(\tau)}(1)\neq0.
$

Since \(\mathcal R=\Lambda_{(t-1)}\) is local with maximal ideal
\((t-1)\mathcal R\), an element of \(\mathcal R\) is a unit precisely
when its image in
$
\mathcal R/(t-1)\mathcal R\cong\mathbb C
$
is nonzero. Therefore
$
\det B_N^{(\tau)}\in\mathcal R^\times.
$
It follows that
\begin{equation}\label{eq:B-isomorphism}
B_N^{(\tau)}:
Q_{\tau,\mathcal R}[N]
\overset{\sim}{\longrightarrow}
\mathcal R^{d_N^{(\tau)}}
\end{equation}
is an isomorphism of free \(\mathcal R\)-modules.

\subsection{Reconstruction of the highest-vector coefficients}

For each \(\tau\in\{0,12\}\), we shall construct recursively a sequence
\[
W_N^{(\tau)}\in Q_{\tau,\mathcal R}[N]\qquad(N\ge0),
\]
with
\(
W_0^{(0)}=\mathbf1, W_0^{(12)}=v_{12}.
\)

\begin{lem}[Positive-mode Ward identity]\label{lem:positive-Ward}
Let \(c,H,K\in\C\). Let \(M_H\) and \(M_K\) be highest-weight
modules for the Virasoro algebra of central charge \(c\), with highest
weights \(H\) and \(K\), respectively. Fix nonzero highest-weight vectors
\(
v_H\in M_H,
v_K\in M_K.
\)
Let
\[
\mathcal Y\in
I\binom{M_K}{M_H\;M_H}
\]
be a Virasoro intertwining operator. 
Suppose that
\[
\mathcal Y(v_H,z)v_H
=
z^{K-2H}\sum_{N\ge0}W_Nz^N,
\qquad
W_N\in M_K[N].
\]
Then, for every \(m>0\) and \(N\ge0\),
\begin{equation}\label{eq:positive-rec}
L(m)W_N
=
\bigl(N+K-H+m(H-1)\bigr)W_{N-m},
\end{equation}
where \(W_j=0\) for \(j<0\).
\end{lem}

\begin{proof}
Since \(v_H\) is a Virasoro highest-weight vector of conformal weight
\(H\), the intertwining-operator commutator formula gives
\[
[L(m),\mathcal Y(v_H,z)]
=
z^m
\left(
z\frac{d}{dz}+(m+1)H
\right)
\mathcal Y(v_H,z).
\]
Applying both sides to \(v_H\), using \(L(m)v_H=0\) for \(m>0\), and
substituting
\[
\mathcal Y(v_H,z)v_H
=
z^{K-2H}\sum_{N\ge0}W_Nz^N,
\]
we obtain
\[
\sum_{N\ge0}L(m)W_N\,z^{K-2H+N}
=
\sum_{N\ge0}
\bigl(K-H+N+mH\bigr)
W_Nz^{K-2H+N+m}.
\]
Extracting the coefficient of \(z^{K-2H+N}\) yields
\[
L(m)W_N
=
\bigl(N+K-H+m(H-1)\bigr)W_{N-m},
\]
with the convention \(W_j=0\) for \(j<0\).
\end{proof}

In the applications below, the preceding lemma is used only after
fixing a parameter value \(t_0\).  We then take
\[
c=c(t_0),\qquad
H=H(t_0)=\frac{42}{t_0}-6,
\qquad
K=K_\tau(t_0).
\]
The resulting Ward identity is then used
algebraically over \(\mathcal R\) to reconstruct the coefficients
\(W_N^{(\tau)}\).

We now construct the coefficients \(W_N^{(\tau)}\) inductively on
\(N\).  The initial term is fixed by the normalization
\(
W_0^{(0)}=\mathbf1,
W_0^{(12)}=v_{12}.
\)
Suppose that, for some \(N>0\), the vectors
\[
W_j^{(\tau)}\in Q_{\tau,\mathcal R}[j],
\qquad 0\le j<N,
\]
have already been constructed.  Motivated by the positive-mode Ward
identity \eqref{eq:positive-rec}, for each \(1\le m\le N\) define
\[
b_{N,m}^{(\tau)}
:=
\bigl(
N+K_\tau(t)-H(t)+m(H(t)-1)
\bigr)
W_{N-m}^{(\tau)}.
\]
Since \(K_\tau(t),H(t)\in\mathcal R\) and
\(W_{N-m}^{(\tau)}\in Q_{\tau,\mathcal R}[N-m]\), we have
\[
b_{N,m}^{(\tau)}
\in Q_{\tau,\mathcal R}[N-m].
\]
Hence
\[
b_N^{(\tau)}
:=
\bigl(
b_{N,1}^{(\tau)},\ldots,b_{N,N}^{(\tau)}
\bigr)
\in
\bigoplus_{m=1}^N Q_{\tau,\mathcal R}[N-m].
\]
If a vector \(w\in Q_{\tau,\mathcal R}[N]\) were to satisfy the full
positive-mode Ward equations at level \(N\), then it would satisfy
\[
A_N^{(\tau)}(t)w=b_N^{(\tau)}.
\]
At this stage we impose only the square subsystem selected in the
preceding subsection.  Applying the coordinate projection
\(\pi_N^{(\tau)}\) gives
\[
B_N^{(\tau)}w
=
\pi_N^{(\tau)}b_N^{(\tau)},
\qquad
B_N^{(\tau)}
=
\pi_N^{(\tau)}\circ A_N^{(\tau)}(t).
\]
By \eqref{eq:B-isomorphism},
\[
B_N^{(\tau)}:
Q_{\tau,\mathcal R}[N]
\overset{\sim}{\longrightarrow}
\mathcal R^{d_N^{(\tau)}}
\]
is an isomorphism.  Therefore this square system has a unique solution,
and we define
\begin{equation}\label{eq:WNdef}
W_N^{(\tau)}
:=
(B_N^{(\tau)})^{-1}
\pi_N^{(\tau)}b_N^{(\tau)}
\in Q_{\tau,\mathcal R}[N].
\end{equation}
This completes the inductive construction.

Because \(B_N^{(\tau)}\) is invertible over the local ring
\(\mathcal R\), rather than merely over its fraction field
\(\mathcal K\), the PBW coordinates of \(W_N^{(\tau)}\) belong to
\(\mathcal R\).  In particular, no pole at \(t=1\) is introduced by
the reconstruction.

Notice that the definition above a priori enforces only the selected
square subsystem
\[
\pi_N^{(\tau)}
A_N^{(\tau)}(t)W_N^{(\tau)}
=
\pi_N^{(\tau)}b_N^{(\tau)}.
\]
It remains to prove that the remaining positive-mode equations also
hold.  This will be established in the next subsection.

\subsection{Verification of all positive-mode equations}

For uniform notation, set
$\mathcal B_N^{(12)}:=\mathcal B_N,$
while $\mathcal B_N^{(0)}$ is as in Lemma~\ref{lem:freeness}. Thus $\mathcal B_N^{(\tau)}$ denotes the global PBW basis of $Q_{\tau,\mathcal R}[N]$ for $\tau\in\{0,12\}$.
Using the global
PBW bases, write any
$w\in Q_{\tau,\mathcal R}[N]$ uniquely as
\[
w=\sum_{e\in\mathcal B_N^{(\tau)}}f_e(t)e,
\qquad f_e(t)\in\mathcal R\subset\mathcal K.
\]
If \(t_0\in\C^\times\) is not a pole of any of these finitely many rational functions, then for each
\(e\in\mathcal B_N^{(\tau)}\) let
$
e(t_0):=e\otimes 1\in Q_{\tau,\Lambda}(t_0)[N]
$
denote its specialization under the base change \(t\mapsto t_0\). We define the
\emph{rational-coordinate evaluation} of \(w\) at \(t_0\) by
\begin{equation}\label{eq:rational-coordinate-evaluation}
w[t_0]
:=
\sum_{e\in\mathcal B_N^{(\tau)}}
f_e(t_0)e(t_0)
\in Q_{\tau,\Lambda}(t_0)[N].
\end{equation}
The same notation will be used for homomorphisms and matrices between finite
graded pieces: after expressing them in the corresponding global PBW bases,
we evaluate their finitely many rational matrix entries at \(t=t_0\), whenever
all such entries are defined. This operation should not be confused with base
change from the local ring \(\mathcal R\), since evaluation at \(t_0\ne1\) need
not define a homomorphism \(\mathcal R\to\C\). The actual fiber at \(t_0\) is
the \(\Lambda\)-fiber
$
Q_{\tau,\Lambda}(t_0)
=
Q_{\tau,\Lambda}\otimes_\Lambda \C_{t_0}.
$
Since the Virasoro action and the global PBW reduction are defined over
\(\Lambda\), rational-coordinate evaluation commutes with all operations
appearing in the finite Ward systems whenever the relevant rational
coordinates are defined.

\begin{prop}[Full positive-mode reconstruction]\label{prop:full-reconstruction}
For every \(\tau\in\{0,12\}\), \(N\ge0\), and \(m>0\), the recursively defined vectors
$
W_N^{(\tau)}\in Q_{\tau,\mathcal R}[N]
$
satisfy the full positive-mode Ward identity
\begin{equation}\label{eq:full-positive-rec}
L(m)W_N^{(\tau)}
=
\bigl(
N+K_\tau(t)-H(t)+m(H(t)-1)
\bigr)
W_{N-m}^{(\tau)},
\end{equation}
where
$
K_0(t)=0, K_{12}(t)=H(t),
$
and, by convention,
$
W_j^{(\tau)}=0 (j<0).
$
Thus the vectors \(W_N^{(\tau)}\), although constructed recursively from the selected square subsystem
$
B_N^{(\tau)}W_N^{(\tau)}
=
\pi_N^{(\tau)}b_N^{(\tau)},
$
in fact satisfy all positive-mode equations.
\end{prop}

\begin{proof}
Fix \(\tau\in\{0,12\}\) and \(N\ge0\). We shall prove that
$$
L(m)W_N^{(\tau)}
=
\bigl(
N+K_\tau(t)-H(t)+m(H(t)-1)
\bigr)W_{N-m}^{(\tau)}
$$
for every \(m>0\), with the convention \(W_j^{(\tau)}=0\) for \(j<0\).

If \(m>N\), then \(L(m)W_N^{(\tau)}=0\) for degree reasons, while
\(W_{N-m}^{(\tau)}=0\) by convention. Hence it remains to consider
\(1\le m\le N\).
We compare the recursively constructed vectors with genuine generic
intertwining-operator coefficients. For the finitely many levels
\(0\le j\le N\), exclude all parameter values at which one of the rational
PBW coordinates occurring in the vectors \(W_j^{(\tau)}\), the vectors
\(b_j^{(\tau)}\), or the matrices \(B_j^{(\tau)}\) is undefined. We also
exclude the zeros of the finitely many nonzero rational functions
$
\det B_j^{(\tau)}(t),
 1\le j\le N.
$
Only finitely many parameter values are excluded. Consequently, there remain
infinitely many generic values \(t_0\) of the type considered in
Section~\ref{subsec:generic-first-row}; in particular, we may take
\(4t_0\in(0,\infty)\setminus\mathbb Q\).

Fix such a \(t_0\). By the generic first-row fusion rule
\eqref{eq:generic-first-row-fusion}, the intertwining space
$$
I\binom{Q_\tau(4t_0)}
        {Q_{12}(4t_0)\;Q_{12}(4t_0)}
$$
is one-dimensional. Using the generic-fiber identifications established
above,
$Q_{\sigma,\Lambda}(t_0)\cong Q_\sigma(4t_0),  \sigma\in\{0,12\},$
choose the unique normalization of a nonzero intertwining operator
$$
\mathcal Y_{\tau,t_0}
\in
I\binom{Q_\tau(4t_0)}
        {Q_{12}(4t_0)\;Q_{12}(4t_0)}
$$
whose leading coefficient agrees with the specialization of our prescribed
level-zero vector \(W_0^{(\tau)}\). Thus
$$
\mathcal Y_{\tau,t_0}
\bigl(v_{12}(t_0),z\bigr)v_{12}(t_0)
=
z^{K_\tau(t_0)-2H(t_0)}
\sum_{j\ge0}\widetilde W_j^{(\tau)}(t_0)z^j,
$$
where
$$
\widetilde W_j^{(\tau)}(t_0)
\in Q_{\tau,\Lambda}(t_0)[j],
\qquad
\widetilde W_0^{(\tau)}(t_0)
=
W_0^{(\tau)}[t_0].
$$
Explicitly, the common level-zero coefficient is the highest-weight vector
of the target: it is \(\mathbf1\) for \(\tau=0\) and \(v_{12}(t_0)\) for
\(\tau=12\).

Lemma~\ref{lem:positive-Ward}, applied at the parameter value \(t_0\), gives
the full positive-mode Ward identities
\begin{equation}\label{eq:generic-positive-rec-proof}
L(m)\widetilde W_j^{(\tau)}(t_0)
=
\bigl(
j+K_\tau(t_0)-H(t_0)+m(H(t_0)-1)
\bigr)
\widetilde W_{j-m}^{(\tau)}(t_0)
\end{equation}
for every \(j\ge0\) and every \(m>0\), where
\(\widetilde W_r^{(\tau)}(t_0)=0\) for \(r<0\).

We now show, successively for \(0\le j\le N\), that
\begin{equation}\label{eq:generic-W-equality-proof}
W_j^{(\tau)}[t_0]
=
\widetilde W_j^{(\tau)}(t_0).
\end{equation}
For \(j=0\) this is exactly our normalization. Suppose that \(j>0\) and that
\eqref{eq:generic-W-equality-proof} has already been proved at every level
strictly below \(j\). By definition,
$$
b_{j,m}^{(\tau)}
=
\bigl(
j+K_\tau(t)-H(t)+m(H(t)-1)
\bigr)
W_{j-m}^{(\tau)},
\qquad 1\le m\le j.
$$
Evaluating at \(t=t_0\) and using the induction hypothesis gives
$$
b_{j,m}^{(\tau)}[t_0]
=
\bigl(
j+K_\tau(t_0)-H(t_0)+m(H(t_0)-1)
\bigr)
\widetilde W_{j-m}^{(\tau)}(t_0).
$$
By \eqref{eq:generic-positive-rec-proof}, the right-hand side is precisely
$
L(m)\widetilde W_j^{(\tau)}(t_0).
$
Hence
$
A_j^{(\tau)}[t_0]\,
\widetilde W_j^{(\tau)}(t_0)
=
b_j^{(\tau)}[t_0],
$
and therefore, after applying the selected coordinate projection,
$
B_j^{(\tau)}[t_0]\,
\widetilde W_j^{(\tau)}(t_0)
=
\pi_j^{(\tau)}b_j^{(\tau)}[t_0].
$
On the other hand, the recursive definition
\eqref{eq:WNdef} of \(W_j^{(\tau)}\), followed by rational-coordinate
evaluation at \(t_0\), gives
$
B_j^{(\tau)}[t_0]\,
W_j^{(\tau)}[t_0]
=
\pi_j^{(\tau)}b_j^{(\tau)}[t_0].
$
By our choice of \(t_0\),
$
\det B_j^{(\tau)}(t_0)\neq0,
$
so \(B_j^{(\tau)}[t_0]\) is invertible. The preceding two equations therefore
have the same unique solution, and hence
$
W_j^{(\tau)}[t_0]
=
\widetilde W_j^{(\tau)}(t_0).
$
This completes the induction on \(j\), and in particular proves
$
W_N^{(\tau)}[t_0]
=
\widetilde W_N^{(\tau)}(t_0).
$

It follows from \eqref{eq:generic-positive-rec-proof} that, for every
\(1\le m\le N\),
$$
L(m)W_N^{(\tau)}[t_0]
=
\bigl(
N+K_\tau(t_0)-H(t_0)+m(H(t_0)-1)
\bigr)
W_{N-m}^{(\tau)}[t_0].
$$
Thus, if we set
$$
\Delta_{N,m}^{(\tau)}
:=
L(m)W_N^{(\tau)}
-
\bigl(
N+K_\tau(t)-H(t)+m(H(t)-1)
\bigr)
W_{N-m}^{(\tau)}
\in Q_{\tau,\mathcal R}[N-m],
$$
then
$
\Delta_{N,m}^{(\tau)}[t_0]=0
$
for every \(t_0\) in the infinite generic set chosen above. 
Finally, expand \(\Delta_{N,m}^{(\tau)}\) in the global PBW basis:
$$
\Delta_{N,m}^{(\tau)}
=
\sum_{e\in\mathcal B_{N-m}^{(\tau)}}g_e(t)e,
\qquad
g_e(t)\in\mathcal R\subset\mathbb C(t).
$$
By Lemma~\ref{lem:freeness}, $\mathcal B_{N-m}^{(\tau)}$ is a $\Lambda$-basis of $Q_{\tau,\Lambda}[N-m]$. Hence, after base change $t\mapsto t_0$, the specialized vectors
$\{e(t_0):e\in\mathcal B_{N-m}^{(\tau)}\}$
form a basis of $Q_{\tau,\Lambda}(t_0)[N-m]$. Hence
\(\Delta_{N,m}^{(\tau)}[t_0]=0\) implies
$
g_e(t_0)=0$
for every 
$e\in\mathcal B_{N-m}^{(\tau)}.
$
Each \(g_e(t)\) is a rational function and vanishes at infinitely many
values of \(t_0\), so \(g_e(t)=0\) identically. Thus
$
\Delta_{N,m}^{(\tau)}=0.
$
This proves the full positive-mode Ward identity for every
\(1\le m\le N\); together with the case \(m>N\) considered at the beginning,
the result follows for all \(m>0\).
\end{proof}

\subsection{Identification with the \texorpdfstring{$c=1$}{c=1} blocks}

We now identify the specialization at the closed point $t=1$ of the
locally reconstructed coefficients with the coefficients of the actual
$c=1$ Virasoro intertwining operators. Recall from the positive-mode
rigidity subsection that
\[
Q_{\tau,1}
=
Q_{\tau,\mathcal R}\otimes_{\mathcal R}\mathbb C_1,
\qquad
\mathbb C_1=\mathcal R/(t-1)\mathcal R,
\]
and that
\[
Q_{\tau,1}[N]
=
Q_{\tau,\mathcal R}[N]\otimes_{\mathcal R}\mathbb C_1.
\]
By Proposition~\ref{prop:special-fibers}, there are natural Virasoro-module
identifications
\begin{equation}\label{eq:special-fiber-identifications-for-blocks}
Q_{0,1}\cong L(1,0),
\qquad
Q_{12,1}\cong L(1,36),
\end{equation}
under which the image of $v_{12}$ in $Q_{12,1}$ is the highest-weight
vector $v_{1,36}$ introduced above.

For $w\in Q_{\tau,\mathcal R}[N]$, we write
\begin{equation}\label{eq:specialization-at-one-notation}
w(1):=w\otimes_{\mathcal R}1\in Q_{\tau,1}[N]
\end{equation}
for its specialization at the closed point $t=1$. In particular,
\[
W_N^{(\tau)}(1)
:=
W_N^{(\tau)}\otimes_{\mathcal R}1
\in Q_{\tau,1}[N].
\]
This notation denotes genuine specialization of the local
$\mathcal R$-family and should be distinguished from the
rational-coordinate evaluation $w[t_0]$ used above for generic
$t_0\ne1$.

\begin{prop}[Special-fiber identification]\label{prop:special-identification}
For each $\tau\in\{0,12\}$, there is a unique normalized nonzero
intertwining operator
\[
\mathcal Y_\tau
\in
I\binom{L(1,K_\tau(1))}
        {L(1,36)\;L(1,36)}
\]
whose highest-vector expansion is normalized by
\(
\widetilde W_0^{(0)}=\mathbf1,
\widetilde W_0^{(12)}=v_{1,36},
\)
where
\begin{equation}\label{eq:c1-block-expansion}
\mathcal Y_\tau(v_{1,36},z)v_{1,36}
=
z^{K_\tau(1)-2H(1)}
\sum_{N\ge0}\widetilde W_N^{(\tau)}z^N,
\qquad
\widetilde W_N^{(\tau)}\in L(1,K_\tau(1))[N].
\end{equation}
Under the identifications
\eqref{eq:special-fiber-identifications-for-blocks}, one has
\begin{equation}\label{eq:special-coefficient-identification}
W_N^{(\tau)}(1)=\widetilde W_N^{(\tau)}
\qquad(N\ge0).
\end{equation}
Thus $W_N^{(0)}(1)$ is the level-$N$ coefficient of the normalized
vacuum-channel intertwining operator, while $W_N^{(12)}(1)$ is the
level-$N$ coefficient of the normalized self-channel intertwining
operator.
\end{prop}

\begin{proof}
Fix $\tau\in\{0,12\}$. We first justify the existence and uniqueness of
the normalized intertwining operator in the statement.
Since
\(
L(1,36)=L\!\left(1,\frac{12^2}{4}\right),
\)
Proposition~\ref{prop:Milas-fusion}, with $m=n=12$, gives
\(
\dim_{\mathbb C}
I\binom{L(1,0)}{L(1,36)\;L(1,36)}=1
\)
and
\(
\dim_{\mathbb C}
I\binom{L(1,36)}{L(1,36)\;L(1,36)}=1.
\)
Indeed, the corresponding values $r=0$ and $r=12$ both belong to
\(
\{0,2,\ldots,24\}.
\)
Thus the relevant intertwining space is one-dimensional in either
channel.

It remains to check that a nonzero intertwining operator in either
channel has a nonzero level-zero highest-vector coefficient. Apply
Lemma~\ref{lem:lowest-coeff} with
\[
s=t=6,
\qquad
r=
\begin{cases}
0,&\tau=0,\\
6,&\tau=12.
\end{cases}
\]
Since $s^2=t^2=36$ and $r^2=K_\tau(1)$, every nonzero intertwining
operator in the required space has an expansion
\[
\mathcal Y(v_{1,36},z)v_{1,36}
=
z^{K_\tau(1)-72}
\left(c\,v_{K_\tau(1)}+\sum_{N\ge1}w_Nz^N\right),
\qquad c\ne0,
\]
where $v_{K_\tau(1)}$ denotes the highest-weight vector of the target.
Therefore the scalar freedom in the one-dimensional intertwining space
can be fixed uniquely by requiring the level-zero coefficient to be
$\mathbf1$ in the vacuum channel and $v_{1,36}$ in the self channel.
This gives the uniquely normalized operator $\mathcal Y_\tau$ and the
expansion \eqref{eq:c1-block-expansion}.

We next compare the coefficients
$\widetilde W_N^{(\tau)}$ with the specialization of the locally
reconstructed coefficients. Since $H(1)=36$, Lemma~\ref{lem:positive-Ward}
applied to $\mathcal Y_\tau$ gives, for every $N\ge0$ and every $m>0$,
\begin{equation}\label{eq:c1-block-positive-rec}
L(m)\widetilde W_N^{(\tau)}
=
\bigl(
N+K_\tau(1)-H(1)+m(H(1)-1)
\bigr)
\widetilde W_{N-m}^{(\tau)},
\end{equation}
where, by convention,
$\widetilde W_j^{(\tau)}=0$ for $j<0$.

On the other hand, Proposition~\ref{prop:full-reconstruction} gives over
$\mathcal R$
\[
L(m)W_N^{(\tau)}
=
\bigl(
N+K_\tau(t)-H(t)+m(H(t)-1)
\bigr)
W_{N-m}^{(\tau)}
\]
for every $N\ge0$ and $m>0$. Since all coefficients in this identity
belong to $\mathcal R$, tensoring with
$\mathbb C_1=\mathcal R/(t-1)\mathcal R$ is legitimate. The Virasoro
action commutes with this base change, and hence specialization at
$t=1$ yields
\begin{equation}\label{eq:specialized-local-positive-rec}
L(m)W_N^{(\tau)}(1)
=
\bigl(
N+K_\tau(1)-H(1)+m(H(1)-1)
\bigr)
W_{N-m}^{(\tau)}(1),
\end{equation}
again with $W_j^{(\tau)}(1)=0$ for $j<0$. Thus the two sequences
\[
\{W_N^{(\tau)}(1)\}_{N\ge0}
\qquad\text{and}\qquad
\{\widetilde W_N^{(\tau)}\}_{N\ge0}
\]
satisfy the same full system of positive-mode Ward identities.

We now prove \eqref{eq:special-coefficient-identification} by induction
on $N$. At level $N=0$, the locally reconstructed sequence is normalized
by
\(
W_0^{(0)}=\mathbf1,
W_0^{(12)}=v_{12}.
\)
After specialization and the identifications
\eqref{eq:special-fiber-identifications-for-blocks}, this gives
\(
W_0^{(0)}(1)=\mathbf1=\widetilde W_0^{(0)},
W_0^{(12)}(1)=v_{1,36}=\widetilde W_0^{(12)}.
\)
Hence the assertion holds at level zero.
Now let $N>0$ and suppose inductively that
\begin{equation}\label{eq:special-identification-induction}
W_j^{(\tau)}(1)=\widetilde W_j^{(\tau)}
\qquad(0\le j<N).
\end{equation}
Set
\(
\delta_N
:=
W_N^{(\tau)}(1)-\widetilde W_N^{(\tau)}
\in Q_{\tau,1}[N],
\)
where we use \eqref{eq:special-fiber-identifications-for-blocks} to
regard both terms as vectors in the same special-fiber module.

For $1\le m\le N$, subtracting
\eqref{eq:c1-block-positive-rec} from
\eqref{eq:specialized-local-positive-rec} gives
\[
L(m)\delta_N
=
\bigl(
N+K_\tau(1)-H(1)+m(H(1)-1)
\bigr)
\left(
W_{N-m}^{(\tau)}(1)
-
\widetilde W_{N-m}^{(\tau)}
\right).
\]
Because $0\le N-m<N$, the induction hypothesis
\eqref{eq:special-identification-induction} makes the right-hand side
zero. Therefore
\(
L(m)\delta_N=0
(1\le m\le N).
\)
For $m>N$, one also has $L(m)\delta_N=0$ for degree reasons, since
$\delta_N$ lies at level $N$.

In particular, by the definition of the special-fiber positive-mode
map,
\[
A_N^{(\tau)}(1)(\delta_N)
=
\bigl(L(1)\delta_N,\ldots,L(N)\delta_N\bigr)
=
0.
\]
Lemma~\ref{lem:positive-rigidity} asserts that
\[
A_N^{(\tau)}(1):
Q_{\tau,1}[N]
\longrightarrow
\bigoplus_{m=1}^NQ_{\tau,1}[N-m]
\]
is injective. Hence $\delta_N=0$, and therefore
\(
W_N^{(\tau)}(1)=\widetilde W_N^{(\tau)}.
\)
This completes the induction on $N$ and proves
\eqref{eq:special-coefficient-identification} for all $N\ge0$.
\end{proof}

\subsection{The dense-parameter principle and the null equation}

\begin{lem}[Dense-parameter principle]\label{lem:dense-parameter}
Let $F_\Lambda$ be a finite free $\Lambda$-module with a fixed
$\Lambda$-basis
\(
e_1,\ldots,e_r.
\)
For $t_0\in\C^\times$, let $\C_{t_0}$ denote $\C$ regarded as a
$\Lambda$-algebra via the specialization $t\mapsto t_0$, and set
\[
F_\Lambda(t_0)
:=
F_\Lambda\otimes_\Lambda \C_{t_0},
\qquad
e_i(t_0)
:=
e_i\otimes 1
\in F_\Lambda(t_0).
\]
Then
\(
e_1(t_0),\ldots,e_r(t_0)
\)
is a $\C$-basis of $F_\Lambda(t_0)$.

Put
\(
F_{\mathcal R}
:=
\mathcal R\otimes_\Lambda F_\Lambda.
\)
For
\[
w=\sum_{i=1}^r f_i(t)e_i\in F_{\mathcal R},
\qquad
f_i(t)\in\mathcal R\subset\mathcal K,
\]
and for any $t_0\in\C^\times$ at which all the rational functions
$f_i(t)$ are defined, define the rational-coordinate evaluation
\(
w[t_0]
:=
\sum_{i=1}^r f_i(t_0)e_i(t_0)
\in F_\Lambda(t_0).
\)
If
\(
w[t_0]=0
\)
for infinitely many such values $t_0$, then $w=0$.

More generally, the same conclusion applies coefficientwise to a
formal series whenever each fixed formal coefficient lies in a
localized finite free $\Lambda$-module and, with respect to a fixed
$\Lambda$-basis of that module, is given by finitely many rational
coordinate functions.  Namely, if the rational-coordinate evaluation
of such a fixed coefficient vanishes for infinitely many values of
$t_0$ at which all of its coordinates are defined, then that
coefficient is identically zero.
\end{lem}

\begin{proof}
Since
\(
F_\Lambda\cong\bigoplus_{i=1}^r\Lambda e_i,
\)
base change along $\Lambda\to\C_{t_0}$ gives
\[
F_\Lambda(t_0)
=
F_\Lambda\otimes_\Lambda\C_{t_0}
\cong
\bigoplus_{i=1}^r\C\,e_i(t_0).
\]
Thus
\(
e_1(t_0),\ldots,e_r(t_0)
\)
form a $\C$-basis of $F_\Lambda(t_0)$.

Suppose now that $w[t_0]=0$ for infinitely many values of $t_0$ at
which all $f_i(t)$ are defined.  For each such $t_0$, the linear
independence of the vectors $e_i(t_0)$ gives
\(
f_i(t_0)=0
(1\le i\le r).
\)
Hence each rational function $f_i(t)\in\mathcal K=\C(t)$ has
infinitely many zeros.  A nonzero rational function has only finitely
many zeros, so
\(
f_i(t)=0
(1\le i\le r).
\)
Therefore $w=0$.

For the coefficientwise assertion, fix one formal coefficient.  By
assumption, it belongs to some localized finite free
$\Lambda$-module and has only finitely many rational coordinates with
respect to a fixed $\Lambda$-basis.  Applying the preceding argument
to those coordinates shows that vanishing of its
rational-coordinate evaluation for infinitely many admissible
specializations forces every coordinate to vanish identically.
Hence the fixed formal coefficient itself is zero.
\end{proof}

Before defining the local descendant operation, we record the standard
Ward identity that motivates it.  This identity concerns an actual
intertwining operator and will later be specialized coefficientwise to
the local family.

\begin{prelem}[Descendant Ward identity]
Let $W_1,W_2,W_3$ be modules for a Virasoro vertex operator algebra,
and let
\(
\mathcal Y\in I\binom{W_3}{W_1\;W_2}
\)
be an intertwining operator.  Suppose that $v_H\in W_2$ is a
highest-weight vector of weight $H$.
For $u\in W_1$, set
\(
F_u(z):=\mathcal Y(u,z)v_H.
\)
Then, for every integer $m\ge1$,
\begin{equation}\label{eq:pre-symbol-Ward}
  \begin{aligned}
F_{L(-m)u}(z)
={}&\sum_{k\ge0}\binom{1-m}{k}(-z)^kL(-m-k)F_u(z)\\
&+(-1)^m z^{-m}
\left(
(m-1)H F_u(z)+zL(-1)F_u(z)
-z\frac{d}{dz}F_u(z)
\right).
\end{aligned}
\end{equation}

\end{prelem}

\begin{proof}
Taking the Virasoro mode $n=-m$ in the standard intertwining-operator
identity \cite[Corollary~4.7(1)]{KoshidaKytola} gives
\begin{align*}
  \mathcal Y(L(-m)u,z)
&=
\sum_{k\ge0}\binom{1-m}{k}(-z)^k
L(-m-k)\mathcal Y(u,z)\\
&\quad-
\sum_{k\ge0}\binom{1-m}{k}(-z)^{1-m-k}
\mathcal Y(u,z)L(k-1).
\end{align*}
After applying this identity to $v_H$, only $k=0,1$ survive in the
second sum.  Using $L(0)v_H=Hv_H$ and
\[
\mathcal Y(u,z)L(-1)v_H
=
L(-1)F_u(z)-\frac{d}{dz}F_u(z),
\]
we obtain \eqref{eq:pre-symbol-Ward}.
\end{proof}

Motivated by \eqref{eq:pre-symbol-Ward}, we now formulate explicitly
the coefficientwise local descendant operation needed below.  This
avoids presupposing the existence of a complete intertwining operator
over \(\mathcal R\). Fix \(\tau\in\{0,12\}\) and set
$
\delta_\tau(t):=K_\tau(t)-2H(t).
$
For \(\ell\ge0\), let
$$
\mathscr F_{\tau,\ell}
:=
\left\{
\sum_{N\ge0}a_N z^{\delta_\tau(t)+N-\ell}
\;\middle|\;
a_N\in Q_{\tau,\mathcal R}[N]
\right\}.
$$
The monomials \(z^{\delta_\tau(t)+r}\), \(r\in\mathbb Z\), are regarded as formal basis symbols, with
$$
z^s z^{\delta_\tau(t)+r}
=
z^{\delta_\tau(t)+r+s},
\qquad
z\frac{d}{dz}z^{\delta_\tau(t)+r}
=
\bigl(\delta_\tau(t)+r\bigr)
z^{\delta_\tau(t)+r}
$$
for \(r,s\in\mathbb Z\).
For \(m\ge1\), define
$$
\mathfrak D_m:\mathscr F_{\tau,\ell}
\longrightarrow
\mathscr F_{\tau,\ell+m}
$$
coefficientwise by
\begin{equation}\label{eq:local-descendant-operator}
\begin{aligned}
\mathfrak D_m F(z):={}&
\sum_{k\ge0}\binom{1-m}{k}(-z)^kL(-m-k)F(z)\\
&+(-1)^m z^{-m}
\left((m-1)H(t)F(z)+zL(-1)F(z)
-z\frac{d}{dz}F(z)\right).
\end{aligned}
\end{equation}
This operation is well defined coefficientwise and indeed maps
$\mathscr F_{\tau,\ell}$ into $\mathscr F_{\tau,\ell+m}$. To see this,
write
\[
F(z)=\sum_{N\ge0}a_Nz^{\delta_\tau(t)+N-\ell},
\qquad a_N\in Q_{\tau,\mathcal R}[N].
\]
Since
\(
L(-r)Q_{\tau,\mathcal R}[N]
\subseteq Q_{\tau,\mathcal R}[N+r],
\)
the term
\(
z^kL(-m-k)a_N\,z^{\delta_\tau(t)+N-\ell}
\)
has coefficient in $Q_{\tau,\mathcal R}[N+m+k]$ and exponent
\(
\delta_\tau(t)+N+k-\ell
=\delta_\tau(t)+(N+m+k)-(\ell+m).
\)
Thus it has precisely the grading required for
$\mathscr F_{\tau,\ell+m}$. For a fixed final degree $M$, such a term
can occur only when $N+m+k=M$, so only finitely many pairs $(N,k)$
contribute. The remaining terms in \eqref{eq:local-descendant-operator}
have the same grading property, using
\(
L(-1)Q_{\tau,\mathcal R}[N]
\subseteq Q_{\tau,\mathcal R}[N+1]
\)
and the formal differentiation rule above. Hence every coefficient of
$\mathfrak D_mF(z)$ is a finite sum in the corresponding finite free
homogeneous component of $Q_{\tau,\mathcal R}$.

For each $\tau\in\{0,12\}$, define the locally reconstructed series
\begin{equation}\label{eq:local-Phi-series}
\Phi_\tau(z)
:=
z^{\delta_\tau(t)}\sum_{N\ge0}W_N^{(\tau)}z^N
=
z^{K_\tau(t)-2H(t)}\sum_{N\ge0}W_N^{(\tau)}z^N.
\end{equation}
Since $W_N^{(\tau)}\in Q_{\tau,\mathcal R}[N]$, we have
\(
\Phi_\tau(z)\in\mathscr F_{\tau,0}.
\) For $\ell\ge0$ and
\[
F(z)
=
\sum_{N\ge0}
a_N z^{\delta_\tau(t)+N-\ell}
\in
\mathscr F_{\tau,\ell},
\qquad
a_N\in Q_{\tau,\mathcal R}[N],
\]
let $t_0\in\C^\times$.  We say that $F(z)$ is \emph{regular at $t_0$}
if, for every $N$, all rational PBW coordinates of $a_N$ are regular
at $t=t_0$.  In that case define its coefficientwise
rational-coordinate evaluation by
\begin{equation}\label{eq:series-rational-coordinate-evaluation}
F(z)[t_0]
:=
\sum_{N\ge0}
a_N[t_0]\,
z^{\delta_\tau(t_0)+N-\ell}
\in
z^{\delta_\tau(t_0)-\ell}Q_{\tau,\Lambda}(t_0)[[z]].
\end{equation}
When only a fixed coefficient is under consideration, the same notation
is used coefficientwise provided that this coefficient is regular at
$t_0$ in the rational PBW coordinates.

\begin{prop}[Local descendant and null-vector identities]
\label{prop:local-descendant-null}
Let $\mathbf p=(p_1,\ldots,p_k)$ be an ordered composition, and put
\(
|\mathbf p|:=p_1+\cdots+p_k.
\)
Define
\begin{equation}\label{eq:local-descendant-series}
F_{\mathbf p,\tau}(z)
:=
\mathfrak D_{p_1}\cdots\mathfrak D_{p_k}\Phi_\tau(z),
\end{equation}
with the rightmost operator acting first. Then
\(
F_{\mathbf p,\tau}(z)\in\mathscr F_{\tau,|\mathbf p|}.
\)
Write
\[
F_{\mathbf p,\tau}(z)
=
\sum_{N\ge0}
F_{\mathbf p,\tau;N}
z^{\delta_\tau(t)+N-|\mathbf p|}
\]
with \(F_{\mathbf p,\tau;N}\in Q_{\tau,\mathcal R}[N]\). 
For every admissible generic $t_0$, the series
$F_{\mathbf p,\tau}(z)$ is regular at $t_0$, and
\begin{equation}\label{eq:generic-descendant-identification}
F_{\mathbf p,\tau}(z)[t_0]
=
\mathcal Y_{\tau,t_0}
\bigl(
L(-p_1)\cdots L(-p_k)v_{12}(t_0),z
\bigr)v_{12}(t_0).
\end{equation}
Consequently, using the level-$13$ BSA operator
\eqref{eq:BSA-beta}, the following local BSA null identity holds
coefficientwise over $\mathcal R$:
\begin{equation}\label{eq:local-BSA-null}
\sum_{\mathbf p\models13}
\beta_{\mathbf p}(t)F_{\mathbf p,\tau}(z)=0.
\end{equation}
\end{prop}

\begin{proof}
Fix \(\tau\in\{0,12\}\) and an ordered composition
$
\mathbf p=(p_1,\ldots,p_k).
$
We divide the proof into three parts.

\medskip
\noindent
\textbf{Step 1: grading and coefficientwise finiteness.}
By \eqref{eq:local-Phi-series},
$$
\Phi_\tau(z)
=
z^{\delta_\tau(t)}
\sum_{N\ge0}W_N^{(\tau)}z^N
\in
\mathscr F_{\tau,0}.
$$
As verified immediately before the statement of the proposition, for every
\(m\ge1\) and every \(\ell\ge0\), the operator
$
\mathfrak D_m:
\mathscr F_{\tau,\ell}
\longrightarrow
\mathscr F_{\tau,\ell+m}
$
is well defined coefficientwise.  Since the rightmost operator in
\eqref{eq:local-descendant-series} acts first, repeated application gives
$
\Phi_\tau
\in\mathscr F_{\tau,0},
$
$
\mathfrak D_{p_k}\Phi_\tau
\in\mathscr F_{\tau,p_k},
$
$
\mathfrak D_{p_{k-1}}\mathfrak D_{p_k}\Phi_\tau
\in
\mathscr F_{\tau,p_{k-1}+p_k},
$
and, continuing in this way,
$$
F_{\mathbf p,\tau}(z)
=
\mathfrak D_{p_1}\cdots\mathfrak D_{p_k}\Phi_\tau(z)
\in
\mathscr F_{\tau,|\mathbf p|}.
$$
We also record explicitly the finiteness needed below.  At each application
of \(\mathfrak D_m\), the coefficient of any fixed final degree receives
contributions from only finitely many pairs of an initial degree and an
index \(k\) in the first sum of
\eqref{eq:local-descendant-operator}.  Consequently, after the finitely
many operations
\(\mathfrak D_{p_k},\ldots,\mathfrak D_{p_1}\), every fixed coefficient
\(F_{\mathbf p,\tau;N}\) is a finite
\(\mathcal R\)-linear combination of finitely many Virasoro descendants
of finitely many vectors \(W_j^{(\tau)}\).
All terms lie in the single homogeneous component
\(Q_{\tau,\mathcal R}[N]\), which is finite free over \(\mathcal R\) by
Lemma~\ref{lem:freeness}.  Hence, relative to the fixed global PBW basis
of this homogeneous component,
\(F_{\mathbf p,\tau;N}\) has only finitely many coordinates, and each of
these coordinates belongs to
\(\mathcal R\subset\mathbb C(t)\).

\medskip
\noindent
\textbf{Step 2: identification with the generic descendant coefficients.}
Fix an admissible generic value \(t_0\).  We first compare the
highest-vector coefficients
\(W_N^{(\tau)}\) with those of the normalized generic intertwining
operator.

By the generic first-row fusion rule
\eqref{eq:generic-first-row-fusion}, there is, up to scalar, a unique
nonzero intertwining operator
$$
\mathcal Y_{\tau,t_0}
\in
I\binom{Q_\tau(4t_0)}
        {Q_{12}(4t_0)\;Q_{12}(4t_0)}.
$$
Normalize it so that its level-zero coefficient agrees with
\(W_0^{(\tau)}\).  Write
\begin{equation}\label{eq:generic-Phi-proof-810}
\mathcal Y_{\tau,t_0}
\bigl(v_{12}(t_0),z\bigr)v_{12}(t_0)
=
z^{\delta_\tau(t_0)}
\sum_{N\ge0}
\widetilde W_N^{(\tau)}(t_0)z^N.
\end{equation}
Thus
$$ 
\widetilde W_0^{(0)}(t_0)=\mathbf1,
\qquad
\widetilde W_0^{(12)}(t_0)=v_{12}(t_0).
$$
Proposition~\ref{prop:full-reconstruction} already identifies
$W_N^{(\tau)}[t_0]$ with the normalized generic intertwining-operator
coefficient on a Zariski-dense set of generic parameters at which the
auxiliary minors $B_j^{(\tau)}$ are nonzero. We now strengthen this
statement to every admissible generic parameter $t_0$: for every
$N\ge0$, the rational PBW coordinates of $W_N^{(\tau)}$ are regular
at $t=t_0$, and
\begin{equation}\label{eq:generic-W-all-proof-810}
W_N^{(\tau)}[t_0]
=
\widetilde W_N^{(\tau)}(t_0).
\end{equation}
We prove this simultaneously by induction on \(N\).

For \(N=0\), this is exactly the chosen normalization.
Assume now that \(N>0\) and that
\eqref{eq:generic-W-all-proof-810} has already been established at all
levels \(j<N\).  In particular, the PBW coordinates of the vectors
\(W_j^{(\tau)}\), \(j<N\), are regular at \(t_0\).
Hence the recursively defined right-hand side
$$
b_N^{(\tau)}
=
\bigl(
b_{N,1}^{(\tau)},\ldots,b_{N,N}^{(\tau)}
\bigr),
$$
where
$$
b_{N,m}^{(\tau)}
=
\bigl(
N+K_\tau(t)-H(t)+m(H(t)-1)
\bigr)
W_{N-m}^{(\tau)},
$$
also has rational PBW coordinates regular at \(t_0\).

Recall the full positive-mode map \(A_N^{(\tau)}(t)\) introduced in the subsection on positive-mode rigidity.  At the admissible generic fiber \(t=t_0\), under the identification
$
Q_{\tau,\Lambda}(t_0)\cong Q_\tau(4t_0),
$
it becomes
$$
A_N^{(\tau)}(t_0):
Q_{\tau,\Lambda}(t_0)[N]
\longrightarrow
\bigoplus_{m=1}^N Q_{\tau,\Lambda}(t_0)[N-m],
\qquad
w\longmapsto
\bigl(L(1)w,\ldots,L(N)w\bigr).
$$
Since \(Q_\tau(4t_0)\) is irreducible, the same singular-vector argument as in Lemma~\ref{lem:positive-rigidity} shows that \(A_N^{(\tau)}(t_0)\) is injective.  Indeed, a vector in its kernel is annihilated by \(L(m)\) for \(1\le m\le N\), and by \(L(m)\) for \(m>N\) for degree reasons; hence it is a positive-level singular vector and must vanish.

Let
$
d_N^{(\tau)}
=
\dim_{\mathbb C}Q_{\tau,\Lambda}(t_0)[N].
$
Since \(A_N^{(\tau)}(t_0)\) is injective, its matrix has column rank
\(d_N^{(\tau)}\).  We may therefore choose \(d_N^{(\tau)}\) rows whose
corresponding square minor, say \(C_N^{(\tau)}(t)\), satisfies
$
\det C_N^{(\tau)}(t_0)\ne0.
$

Let
\(\rho_N^{(\tau)}\)
denote the corresponding row projection.  Because the entries of the
global positive-mode matrix are rational functions of \(t\), the inverse
\((C_N^{(\tau)}(t))^{-1}\) is defined in a neighborhood of \(t_0\) in
the sense of rational functions regular at \(t_0\).  Hence
\begin{equation}\label{eq:local-solution-810}
W_{N,t_0}^{\mathrm{loc}}
:=
\bigl(C_N^{(\tau)}(t)\bigr)^{-1}
\rho_N^{(\tau)}b_N^{(\tau)}
\end{equation}
has PBW coordinates regular at \(t_0\).

On the other hand, Proposition~\ref{prop:full-reconstruction} gives,
over \(\mathcal R\),
$
A_N^{(\tau)}(t)W_N^{(\tau)}
=
b_N^{(\tau)}.
$
After applying \(\rho_N^{(\tau)}\), we obtain over the rational function
field \(\mathbb C(t)\),
$
C_N^{(\tau)}(t)W_N^{(\tau)}
=
\rho_N^{(\tau)}b_N^{(\tau)}.
$
Since \(\det C_N^{(\tau)}(t)\) is not the zero rational function,
comparison with \eqref{eq:local-solution-810} gives
$
W_N^{(\tau)}
=
W_{N,t_0}^{\mathrm{loc}}
$
as vectors over \(\mathbb C(t)\).
Thus the rational PBW coordinates of \(W_N^{(\tau)}\) are regular at
\(t_0\).

It remains to identify their values.  Lemma~\ref{lem:positive-Ward},
applied to the generic intertwining operator
\(\mathcal Y_{\tau,t_0}\), gives
$$
L(m)\widetilde W_N^{(\tau)}(t_0)
=
\bigl(
N+K_\tau(t_0)-H(t_0)+m(H(t_0)-1)
\bigr)
\widetilde W_{N-m}^{(\tau)}(t_0)
$$
for \(1\le m\le N\).
By the induction hypothesis,
$
\widetilde W_{N-m}^{(\tau)}(t_0)
=
W_{N-m}^{(\tau)}[t_0].
$
Therefore
\(\widetilde W_N^{(\tau)}(t_0)\)
satisfies the specialized square subsystem
$
C_N^{(\tau)}(t_0)
\widetilde W_N^{(\tau)}(t_0)
=
\rho_N^{(\tau)}b_N^{(\tau)}[t_0].
$

The matrix \(C_N^{(\tau)}(t_0)\) is invertible, so this subsystem has a
unique solution.  By \eqref{eq:local-solution-810}, that solution is
\(W_N^{(\tau)}[t_0]\).  Hence
$
W_N^{(\tau)}[t_0]
=
\widetilde W_N^{(\tau)}(t_0).
$
This completes the induction and proves
\eqref{eq:generic-W-all-proof-810} at every admissible generic \(t_0\).
Notice that this argument does not require the particular auxiliary
minor \(B_N^{(\tau)}\) used in the construction of
\(W_N^{(\tau)}\) to remain nonzero at \(t_0\).

Since all coefficients \(W_N^{(\tau)}\) are now known to be regular at
\(t=t_0\), and since \(H(t)\), \(K_\tau(t)\), and
\(\delta_\tau(t)\) are regular on \(\C^\times\), the coefficientwise
finiteness established in Step~1 shows that \(\Phi_\tau(z)\) and every
iterated descendant series obtained from it by the operators
\(\mathfrak D_m\) are regular at \(t_0\).  Hence all coefficientwise
evaluations used below are well defined.  From
\eqref{eq:generic-W-all-proof-810} we obtain
\begin{equation}\label{eq:Phi-generic-identification-810}
\Phi_\tau(z)[t_0]
=
\mathcal Y_{\tau,t_0}
\bigl(v_{12}(t_0),z\bigr)v_{12}(t_0).
\end{equation}

We now pass from the highest vector to arbitrary Virasoro descendants.
For an actual intertwining operator, put
$
F_u(z):=\mathcal Y_{\tau,t_0}(u,z)v_{12}(t_0).
$
Let \(\mathfrak D_{m,t_0}\) denote the operator obtained from
\eqref{eq:local-descendant-operator} by replacing \(H(t)\) with
\(H(t_0)\) and letting the Virasoro operators act on the fiber
\(Q_{\tau,\Lambda}(t_0)\).  Specializing the descendant Ward
identity \eqref{eq:pre-symbol-Ward} at \(t=t_0\) says exactly that, for
every \(m\ge1\),
\[
F_{L(-m)u}(z)=\mathfrak D_{m,t_0}F_u(z).
\]
For every fixed homogeneous coefficient, only finitely many terms in
the first sum of \eqref{eq:local-descendant-operator} contribute.
Therefore rational-coordinate evaluation commutes coefficientwise with
the Virasoro operators and with \(z\,d/dz\), and for the descendant
series considered here
\begin{equation}\label{eq:D-evaluation-compatibility-810}
\bigl(\mathfrak D_mF\bigr)(z)[t_0]
=
\mathfrak D_{m,t_0}\bigl(F(z)[t_0]\bigr).
\end{equation}
Starting from \eqref{eq:Phi-generic-identification-810}, repeated use of
\eqref{eq:D-evaluation-compatibility-810} and the descendant Ward
identity gives
\begin{equation}\label{eq:whole-descendant-identification-810}
F_{\mathbf p,\tau}(z)[t_0]
=
\mathcal Y_{\tau,t_0}
\bigl(
L(-p_1)\cdots L(-p_k)v_{12}(t_0),z
\bigr)v_{12}(t_0).
\end{equation}
This proves \eqref{eq:generic-descendant-identification}.

\medskip
\noindent
\textbf{Step 3: descent of the BSA null relation to \(\mathcal R\).}
Define
$
G_\tau(z)
:=
\sum_{\mathbf p\models13}
\beta_{\mathbf p}(t)F_{\mathbf p,\tau}(z).
$
Since every composition occurring here has total size \(13\),
Step~1 gives
$
F_{\mathbf p,\tau}(z)
\in
\mathscr F_{\tau,13}.
$
Hence
$
G_\tau(z)
=
\sum_{N\ge0}
G_{\tau;N}
z^{\delta_\tau(t)+N-13},
$
where
\begin{equation}\label{eq:GtauN-810}
G_{\tau;N}
=
\sum_{\mathbf p\models13}
\beta_{\mathbf p}(t)
F_{\mathbf p,\tau;N}
\in
Q_{\tau,\mathcal R}[N].
\end{equation}
The sum is finite, since there are only finitely many ordered
compositions of \(13\).  Moreover,
\(\beta_{\mathbf p}(t)\in\Lambda\) is regular on \(\C^\times\).
By Step~2, every \(F_{\mathbf p,\tau}(z)\) is regular at every
admissible generic \(t_0\).  Hence \(G_\tau(z)\) is also regular at
such \(t_0\), and coefficientwise linearity gives
\[
G_\tau(z)[t_0]
=
\sum_{\mathbf p\models13}
\beta_{\mathbf p}(t_0)F_{\mathbf p,\tau}(z)[t_0].
\]
Using \eqref{eq:generic-descendant-identification} and the definition
\eqref{eq:BSA-beta} of \(S_{13}(t)\), we therefore obtain
\begin{align*}
G_\tau(z)[t_0]
&=
\sum_{\mathbf p\models13}
\beta_{\mathbf p}(t_0)
\mathcal Y_{\tau,t_0}
\bigl(
L(-p_1)\cdots L(-p_k)v_{12}(t_0),z
\bigr)v_{12}(t_0)\\
&=
\mathcal Y_{\tau,t_0}
\bigl(
S_{13}(t_0)v_{12}(t_0),z
\bigr)v_{12}(t_0).
\end{align*}
For admissible generic \(t_0\), the first source of
\(\mathcal Y_{\tau,t_0}\) is the quotient
\(Q_{12}(4t_0)\), and
\(S_{13}(t_0)v_{12}(t_0)\)
is precisely the level-\(13\) singular vector killed in this quotient;
see Lemma~\ref{lem:BSA-local-singular} and the generic-fiber
identification above.  Hence
$
\mathcal Y_{\tau,t_0}
\bigl(
S_{13}(t_0)v_{12}(t_0),z
\bigr)v_{12}(t_0)
=
0.
$
Thus
\(
G_\tau(z)[t_0]=0
\)
coefficientwise for every admissible generic \(t_0\).  Fix \(N\ge0\).
By the definition of coefficientwise evaluation,
\(
G_{\tau;N}[t_0]=0
\)
for every admissible generic \(t_0\), and hence for infinitely many
values of \(t_0\).

By Lemma~\ref{lem:freeness},
\(Q_{\tau,\Lambda}[N]\)
is finite free over \(\Lambda\), and
$
Q_{\tau,\mathcal R}[N]
=
\mathcal R\otimes_\Lambda Q_{\tau,\Lambda}[N].
$
Relative to its fixed global PBW basis, the vector \(G_{\tau;N}\) has
finitely many rational coordinate functions.  Since its
rational-coordinate evaluation vanishes at infinitely many admissible
generic values of \(t_0\), Lemma~\ref{lem:dense-parameter} implies
$
G_{\tau;N}=0
$ in $Q_{\tau,\mathcal R}[N].
$
Because \(N\ge0\) was arbitrary, every coefficient of \(G_\tau(z)\)
vanishes.  Therefore
$
\sum_{\mathbf p\models13}
\beta_{\mathbf p}(t)F_{\mathbf p,\tau}(z)
=
0
$
coefficientwise over \(\mathcal R\), which is
\eqref{eq:local-BSA-null}.  This completes the proof.
\end{proof}

\begin{thm}[Local regularization]\label{thm:local-regularization}
For $\tau\in\{0,12\}$, let
$\{W_N^{(\tau)}\}_{N\ge0}$ be the sequences constructed in
\eqref{eq:WNdef}. Then these are the unique sequences with the prescribed
initial terms satisfying the full positive-mode Ward identities
\begin{equation}\label{eq:local-regularization-Ward}
L(m)W_N^{(\tau)}
=
\bigl(
N+K_\tau(t)-H(t)+m(H(t)-1)
\bigr)
W_{N-m}^{(\tau)}.
\end{equation}
Moreover:

{\rm(1)} each $W_N^{(\tau)}$ is regular at $t=1$, and its specialization
agrees with the corresponding normalized $c=1$ intertwining-operator
coefficient;

{\rm(2)} the descendant series defined in
\eqref{eq:local-descendant-series} satisfy the level-$13$ BSA relation
$$\sum_{\mathbf p\models13} \beta_{\mathbf p}(t)F_{\mathbf p,\tau}(z)=0;$$

{\rm(3)} the $C_2$-symbols
$\Sigma_{\tau,\mathcal R}(W_N^{(\tau)})$ are regular at $t=1$.

\end{thm}

\begin{proof}
By \eqref{eq:WNdef} and the isomorphism
\eqref{eq:B-isomorphism}, the coefficients
\(W_N^{(\tau)}\in Q_{\tau,\mathcal R}[N]\) are uniquely determined
recursively by the selected square Ward subsystem.  Proposition~
\ref{prop:full-reconstruction} shows that the resulting sequences satisfy
the full positive-mode Ward identities
\eqref{eq:local-regularization-Ward}.

It remains only to check uniqueness for the full Ward system.
Suppose that
\(\{\widehat W_N^{(\tau)}\}_{N\ge0}\) is another sequence with the same
initial term and satisfying \eqref{eq:local-regularization-Ward}.
We argue by induction on \(N\).  The case \(N=0\) is immediate.
Assume that
$
\widehat W_j^{(\tau)}=W_j^{(\tau)}
 (0\le j<N).
$
Then, for \(1\le m\le N\), the Ward identities and the induction
hypothesis give
$
L(m)\widehat W_N^{(\tau)}=b_{N,m}^{(\tau)}.
$
Hence
$
A_N^{(\tau)}(t)\widehat W_N^{(\tau)}=b_N^{(\tau)},
$
and therefore, after applying the selected coordinate projection,
$$
B_N^{(\tau)}\widehat W_N^{(\tau)}
=
\pi_N^{(\tau)}b_N^{(\tau)}
=
B_N^{(\tau)}W_N^{(\tau)}.
$$
Since \(B_N^{(\tau)}\) is an isomorphism over \(\mathcal R\), we obtain
\(\widehat W_N^{(\tau)}=W_N^{(\tau)}\).  This proves uniqueness.

Assertion~{\rm(1)} follows from
\(W_N^{(\tau)}\in Q_{\tau,\mathcal R}[N]\), together with
Proposition~\ref{prop:special-identification}, which gives
$
W_N^{(\tau)}(1)=\widetilde W_N^{(\tau)}.
$
Assertion~{\rm(2)} is Proposition~
\ref{prop:local-descendant-null}.  Finally, assertion~{\rm(3)} follows
immediately from Lemma~\ref{lem:symbol-descent}, since
$
W_N^{(\tau)}\in Q_{\tau,\mathcal R}
$ and $
\Sigma_{\tau,\mathcal R}(Q_{\tau,\mathcal R})
\subseteq\mathcal R[T].
$
\end{proof}

\section{The scalar Benoit--Saint-Aubin recurrence}
\label{sec:BSA-recurrence}

We now apply the local $C_2$-symbol to the highest-vector coefficients constructed in Section~\ref{sec:local-regularization}, converting the level-$13$ Benoit--Saint-Aubin null relation into a finite-width scalar recurrence. For $\tau\in\{0,12\}$, recall $K_\tau(t)=0$ if $\tau=0$ and $H(t)$ if $\tau=12$, with $H(t)=42/t-6$. In a fixed channel, abbreviate $H:=H(t)$, $K:=K_\tau(t)$, $\delta:=K-2H$. Recall the locally reconstructed series $\Phi_\tau(z)$ from \eqref{eq:local-Phi-series}. We extend $\Sigma_{\tau,\mathcal R}$ coefficientwise to every local formal series
\[
F(z)=\sum_{N\ge0}a_N z^{\delta_\tau(t)+N-\ell}\in\mathscr F_{\tau,\ell}
\]
by
\begin{equation}\label{eq:Sigma-coefficientwise-extension}
\Sigma_{\tau,\mathcal R}(F(z))
:=
\sum_{N\ge0}\Sigma_{\tau,\mathcal R}(a_N)
 z^{\delta_\tau(t)+N-\ell}.
\end{equation}
The powers $z^{\delta_\tau(t)+r}$ and the formal operator $z\,d/dz$ are understood with the convention fixed in Section~\ref{sec:local-regularization}. Since $\Sigma_{\tau,\mathcal R}$ is defined on every homogeneous coefficient in $Q_{\tau,\mathcal R}$, this extension is well defined. For $\Phi_\tau(z)$, odd-level coefficients have zero symbol and level-$2n$ coefficients have symbols proportional to $T^n$, so there exist unique $c_n^{(\tau)}(t)\in\mathcal R$ ($n\ge0$) with
\begin{equation}\label{eq:Phi-symbol}
\Phi_\tau(z,T;t):=\Sigma_{\tau,\mathcal R}(\Phi_\tau(z))=z^{\delta_\tau(t)}\sum_{n\ge0}c_n^{(\tau)}(t)(Tz^2)^n,
\end{equation}
where $\delta_\tau(t)=K_\tau(t)-2H(t)$ and $c_0^{(\tau)}(t)=1$. Introducing $w:=Tz^2$, we have $\Phi_\tau(z,T;t)=z^{\delta_\tau(t)}\sum_{n\ge0}c_n^{(\tau)}(t)w^n$. By Proposition~\ref{prop:special-identification} and Section~\ref{sec:symbol-determinant},
\begin{equation}\label{eq:uv-specialization}
u_n=c_n^{(0)}(1),\qquad v_n=c_n^{(12)}(1).
\end{equation}

\subsection{The descendant operator on the \texorpdfstring{$C_2$}{C2}-symbol}

We now pass entirely to the local formal-series setting.  The genuine
intertwining operators used in Section~\ref{sec:local-regularization}
were needed there to establish the local descendant and BSA null
identities.  From this point on, only the explicitly defined local
operators $\mathfrak D_m$ and the coefficientwise extension
\eqref{eq:Sigma-coefficientwise-extension} are used.

\begin{lem}[Descendant operator on the $C_2$-symbol]
\label{lem:symbol-descendant-Ward}
Fix $\tau\in\{0,12\}$.  For every integer $m\ge1$, define the
$\mathcal R$-linear formal differential operator
\begin{equation}\label{eq:Dhat}
\widehat D_m
:=
\delta_{m,2}T
+(-1)^m z^{-m}
\left((m-1)H-z\frac{d}{dz}\right)
\end{equation}
on the $\mathcal R$-module
\[
\bigoplus_{\ell\ge0}z^{K-2H-\ell}\mathcal R[[w]],
\qquad w=Tz^2.
\]
For every $\ell\ge0$, it restricts to
\begin{equation}\label{eq:Dhat-level-shift}
\widehat D_m\!\left(
 z^{K-2H-\ell}\mathcal R[[w]]
\right)
\subseteq
 z^{K-2H-\ell-m}\mathcal R[[w]].
\end{equation}
Moreover, for every $F(z)\in\mathscr F_{\tau,\ell}$,
$\Sigma_{\tau,\mathcal R}(F(z))$ belongs to
$z^{K-2H-\ell}\mathcal R[[w]]$, and
\begin{equation}\label{eq:local-Dhat-compatibility}
\Sigma_{\tau,\mathcal R}\bigl(\mathfrak D_mF(z)\bigr)
=
\widehat D_m\,\Sigma_{\tau,\mathcal R}(F(z)).
\end{equation}
In particular, both sides are well defined coefficientwise.
\end{lem}

\begin{proof}
First fix $\ell\ge0$ and put $A:=K-2H-\ell$.  Since
$w=Tz^2$, the formal operator $z\,d/dz$ preserves
$z^A\mathcal R[[w]]$.  Also,
\[
T\,z^A\mathcal R[[w]]
=
z^{A-2}w\mathcal R[[w]]
\subseteq
z^{A-2}\mathcal R[[w]].
\]
It follows directly from \eqref{eq:Dhat} that
\(
\widehat D_m\bigl(z^A\mathcal R[[w]]\bigr)
\subseteq
z^{A-m}\mathcal R[[w]],
\)
which proves \eqref{eq:Dhat-level-shift} and also shows that
$\widehat D_m$ is well defined on the displayed direct sum.

Now write
\[
F(z)=\sum_{N\ge0}a_Nz^{\delta_\tau(t)+N-\ell},
\qquad a_N\in Q_{\tau,\mathcal R}[N].
\]
By Lemma~\ref{lem:symbol-descent} and the grading of
$\sigma_{\mathcal R}$,
\[
\Sigma_{\tau,\mathcal R}(a_N)\in
\begin{cases}
\mathcal R\,T^{N/2},&N\text{ even},\\
0,&N\text{ odd}.
\end{cases}
\]
Since $\delta_\tau(t)=K-2H$ and $w=Tz^2$, it follows that
\[
\Sigma_{\tau,\mathcal R}(F(z))
\in
z^{K-2H-\ell}\mathcal R[[Tz^2]]
=
z^{K-2H-\ell}\mathcal R[[w]],
\]
so the right-hand side of
\eqref{eq:local-Dhat-compatibility} is defined by
\eqref{eq:Dhat} and the restriction property
\eqref{eq:Dhat-level-shift}.

By the coefficientwise finiteness proved after
\eqref{eq:local-descendant-operator}, every coefficient of
$\mathfrak D_mF(z)$ is a finite sum of vectors in a homogeneous
component of $Q_{\tau,\mathcal R}$.  Hence
$\Sigma_{\tau,\mathcal R}(\mathfrak D_mF(z))$ is defined by
\eqref{eq:Sigma-coefficientwise-extension}.

We next record the effect of the target symbol on a negative Virasoro
mode.  If $q\in Q_{\tau,\mathcal R}$, Lemma~\ref{lem:symbol-descent}
and the multiplicativity of $\sigma_{\mathcal R}$ give
\begin{equation}\label{eq:Sigma-negative-mode}
\Sigma_{\tau,\mathcal R}(L(-r)q)
=
\begin{cases}
T\,\Sigma_{\tau,\mathcal R}(q),& r=2,\\
0,& r\neq2.
\end{cases}
\end{equation}
Indeed, this follows first for a PBW representative of $q$ and is
independent of the representative because
$\Sigma_{\tau,\mathcal R}$ already descends to
$Q_{\tau,\mathcal R}$.

Apply $\Sigma_{\tau,\mathcal R}$ coefficientwise to
\eqref{eq:local-descendant-operator}.  In its first sum,
\eqref{eq:Sigma-negative-mode} kills every term unless $m+k=2$.
The case $(m,k)=(2,0)$ contributes
$T\Sigma_{\tau,\mathcal R}(F(z))$, whereas the only other possibility,
$(m,k)=(1,1)$, has coefficient $\binom{0}{1}=0$.  Thus the first sum
contributes precisely
\(
\delta_{m,2}T\,\Sigma_{\tau,\mathcal R}(F(z)).
\)
Moreover, \eqref{eq:Sigma-negative-mode} with $r=1$ gives
$\Sigma_{\tau,\mathcal R}(L(-1)q)=0$.  Since the symbol acts only on
the coefficients, it commutes with multiplication by elements of
$\mathcal R$, with integral powers of $z$, and with the formal operator
$z\,d/dz$.  Therefore the remaining terms give
\[
(-1)^m z^{-m}
\left((m-1)H-z\frac{d}{dz}\right)
\Sigma_{\tau,\mathcal R}(F(z)).
\]
Combining the two contributions proves
\eqref{eq:local-Dhat-compatibility}.
\end{proof}

\begin{rmk}\label{rem:D1}
For $m=1$, $\widehat D_1=d/dz$, recovering on the symbol side the
formal $L(-1)$-derivative operation.  For $m=2$,
\[
\widehat D_2=T+z^{-2}\left(H-z\frac{d}{dz}\right).
\]
Moreover,
$[\widehat D_1,\widehat D_2]=\widehat D_3$, in agreement with
$[L(-1),L(-2)]=L(-3)$.  No intertwining operator over $\mathcal R$ is
asserted or required here; the statement concerns only the local
formal-series operators constructed in Section~\ref{sec:local-regularization}.
\end{rmk}

\subsection{Removal of the \texorpdfstring{$z$}{z}-powers}

We now factor out the common $z$-power in the action of
$\widehat D_m$ and restrict the induced operator to the polynomial
subspace $\mathcal R[w]\subset\mathcal R[[w]]$.  Recall that, after
descendant operators of total level $\ell$ have been applied, the
corresponding coefficientwise symbol is a linear combination of terms
of the form
\(
z^{K-2H-\ell}w^n,
 n\ge0.
\)
Thus the common $z$-power may be factored out, leaving an induced
differential operator on the polynomial variable $w$.

\begin{lem}\label{lem:Dmell}
For $m\ge1$ and $\ell\ge0$, define an $\mathcal R$-linear operator
\[
\mathscr D_{m,\ell}^{(K)}:\mathcal R[w]\longrightarrow\mathcal R[w]
\]
by
\begin{equation}\label{eq:Dmell}
\mathscr D_{m,\ell}^{(K)}
:=
\delta_{m,2}\,w
+
(-1)^m
\left(
(m+1)H-K+\ell
-
2w\frac{d}{dw}
\right).
\end{equation}
Then, for every $f(w)\in\mathcal R[w]$,
\begin{equation}\label{eq:Dhat-factor}
\widehat D_m
\left(
z^{K-2H-\ell}f(w)
\right)
=
z^{K-2H-\ell-m}
\mathscr D_{m,\ell}^{(K)}f(w).
\end{equation}
In particular, for every $n\ge0$,
\begin{equation}\label{eq:Dmell-monomial}
\mathscr D_{m,\ell}^{(K)}(w^n)
=
\delta_{m,2}w^{n+1}
+
(-1)^m
\bigl(
(m+1)H-K+\ell-2n
\bigr)w^n.
\end{equation}
\end{lem}

\begin{proof}
Set
\(
A:=K-2H-\ell.
\)
Since $w=Tz^2$, with $T$ independent of $z$, we have
\(
z\frac{d}{dz}w=2w.
\)
Hence, for every $f(w)\in\mathcal R[w]$,
\[
z\frac{d}{dz}
\left(
z^A f(w)
\right)
=
z^A
\left(
Af(w)+2w\frac{d}{dw}f(w)
\right).
\]
Using the definition \eqref{eq:Dhat}, the second term of $\widehat D_m$
therefore gives
\begin{align*}
&(-1)^m z^{-m}
\left(
(m-1)H-z\frac{d}{dz}
\right)
\left(
z^A f(w)
\right)
\\
&\qquad
=
(-1)^m z^{A-m}
\left(
(m-1)H-A
-
2w\frac{d}{dw}
\right)f(w).
\end{align*}
Since
\[
(m-1)H-A
=
(m-1)H-(K-2H-\ell)
=
(m+1)H-K+\ell,
\]
this becomes
\[
(-1)^m z^{A-m}
\left(
(m+1)H-K+\ell
-
2w\frac{d}{dw}
\right)f(w).
\]
It remains to treat the first term of \eqref{eq:Dhat}.  This term is
present only when $m=2$, and in that case
\[
Tz^A f(w)
=
z^{A-2}(Tz^2)f(w)
=
z^{A-2}wf(w).
\]
Combining the two contributions yields
\[
\widehat D_m
\left(
z^A f(w)
\right)
=
z^{A-m}
\left[
\delta_{m,2}w
+
(-1)^m
\left(
(m+1)H-K+\ell
-
2w\frac{d}{dw}
\right)
\right]f(w),
\]
which is precisely \eqref{eq:Dhat-factor} and \eqref{eq:Dmell}.
Finally, taking $f(w)=w^n$ and using
\(
w\frac{d}{dw}w^n=nw^n
\)
gives \eqref{eq:Dmell-monomial}.
\end{proof}

\subsection{Ordered compositions and the width-six operator}

For $\mathbf p=(p_1,\ldots,p_k)\models13$, define suffix levels $\ell_j:=p_{j+1}+\cdots+p_k$ ($\ell_k=0$). For $n\ge0$, set
\begin{equation}\label{eq:Ppn}
P_{\mathbf p,n}^{(K)}(w):=\mathscr D_{p_1,\ell_1}^{(K)}\mathscr D_{p_2,\ell_2}^{(K)}\cdots\mathscr D_{p_k,0}^{(K)}(w^n),
\end{equation}
with rightmost acting first. For $F(w)\in\mathcal R[w]$, write $[w^j]F(w)$ for the coefficient of $w^j$. For $n,q\ge0$, define
\begin{equation}\label{eq:Cq}
C_q^{(K)}(n;t)
:=
\sum_{\mathbf p\models13}
\beta_{\mathbf p}(t)
[w^{n+q}]P_{\mathbf p,n}^{(K)}(w).
\end{equation}
Since
$
P_{\mathbf p,n}^{(K_\tau)}(w)\in\mathcal R[w],
$
we have
$
[w^{n+q}]P_{\mathbf p,n}^{(K_\tau)}(w)\in\mathcal R.
$
Together with $\beta_{\mathbf p}(t)\in\mathcal R$, this implies
$
C_q^{(K_\tau)}(n;t)\in\mathcal R.
$

\begin{lem}[Width bound]\label{lem:width-bound}
For every $n\ge0$, $C_q^{(K)}(n;t)=0$ for $q>6$.
\end{lem}

\begin{proof}
For $d\ge0$, let
$$
\mathcal R[w]_{\le d}
:=
\operatorname{span}_{\mathcal R}\{1,w,\ldots,w^d\}.
$$
By \eqref{eq:Dmell-monomial},
$$
\mathscr D_{m,\ell}^{(K)}(w^r)
=
\delta_{m,2}w^{r+1}
+
(-1)^m
\bigl((m+1)H-K+\ell-2r\bigr)w^r.
$$
Hence
$$
\mathscr D_{m,\ell}^{(K)}
\bigl(\mathcal R[w]_{\le d}\bigr)
\subseteq
\mathcal R[w]_{\le d+\delta_{m,2}}.
$$
In particular, applying an operator with $m\ne2$ does not increase the
$w$-degree, whereas an operator with $m=2$ increases the possible
$w$-degree by at most one.
Now fix $\mathbf p=(p_1,\ldots,p_k)\models13$, and let
$
r_2(\mathbf p):=\#\{j: p_j=2\}.
$
Starting from $w^n$ and applying the operators in
\eqref{eq:Ppn} from right to left, the preceding inclusion gives
inductively
$
P_{\mathbf p,n}^{(K)}(w)
\in
\mathcal R[w]_{\le n+r_2(\mathbf p)}.
$
Since $p_1+\cdots+p_k=13$ and all $p_j\ge1$, we have
$
2r_2(\mathbf p)\le13,
$
and therefore $r_2(\mathbf p)\le6$. Thus
$
P_{\mathbf p,n}^{(K)}(w)\in\mathcal R[w]_{\le n+6},
$
so that
$
[w^{n+q}]P_{\mathbf p,n}^{(K)}(w)=0
(q>6).
$
This holds for every composition $\mathbf p\models13$. Hence, by
\eqref{eq:Cq},
$
C_q^{(K)}(n;t)
=
\sum_{\mathbf p\models13}
\beta_{\mathbf p}(t)
[w^{n+q}]P_{\mathbf p,n}^{(K)}(w)
=0
(q>6).
$
\end{proof}

\subsection{The scalar null-vector recurrence}

\begin{prop}[Width-six recurrence]\label{prop:width6}
The coefficients of \eqref{eq:Phi-symbol} satisfy, for every $N\ge1$,
\begin{equation}\label{eq:width6}
\sum_{q=0}^{\min(6,N)}C_q^{(K)}(N-q;t)c_{N-q}^{(\tau)}(t)=0.
\end{equation}
Equivalently,
\begin{equation}\label{eq:solved-recurrence}
C_0^{(K)}(N;t)c_N^{(\tau)}(t)=-\sum_{q=1}^{\min(6,N)}C_q^{(K)}(N-q;t)c_{N-q}^{(\tau)}(t).
\end{equation}
\end{prop}

\begin{proof}
The local null identity \eqref{eq:local-BSA-null} of
Proposition~\ref{prop:local-descendant-null}, together with
\eqref{eq:local-Dhat-compatibility}, gives after applying the target
symbol
\begin{equation}\label{eq:null-Dhat}
0=\sum_{\mathbf p\models13}\beta_{\mathbf p}(t)\widehat D_{p_1}\cdots\widehat D_{p_k}\Phi_\tau(z,T;t).
\end{equation}
Substituting $\Phi_\tau=z^{K-2H}\sum_{n\ge0}c_n^{(\tau)}(t)w^n$ and applying Lemma~\ref{lem:Dmell} repeatedly yields $$\widehat D_{p_1}\cdots\widehat D_{p_k}(z^{K-2H}w^n)=z^{K-2H-13}P_{\mathbf p,n}^{(K)}(w).$$ Removing the common $z$-factor gives
\begin{equation}\label{eq:null-polynomial}
0=\sum_{n\ge0}c_n^{(\tau)}(t)\sum_{\mathbf p\models13}\beta_{\mathbf p}(t)P_{\mathbf p,n}^{(K)}(w).
\end{equation}
Extracting $[w^N]$ and using Lemma~\ref{lem:width-bound} (so $0\le N-n\le6$), set $q=N-n$:
\[
\sum_{q=0}^{\min(6,N)}c_{N-q}^{(\tau)}(t)\sum_{\mathbf p\models13}\beta_{\mathbf p}(t)[w^N]P_{\mathbf p,N-q}^{(K)}(w)=\sum_{q=0}^{\min(6,N)}C_q^{(K)}(N-q;t)c_{N-q}^{(\tau)}(t)=0,
\]
proving \eqref{eq:width6}. Isolating $q=0$ gives \eqref{eq:solved-recurrence}.
\end{proof}

\begin{rmk}[The level-zero equation]\label{rem:indicial-equation}
The $w^0$ coefficient gives $C_0^{(K)}(0;t)c_0^{(\tau)}(t)=0$; since $c_0^{(\tau)}(t)=1$, $C_0^{(K)}(0;t)=0$. This indicial condition holds for both channels; recursion begins at $N=1$.
\end{rmk}

\subsection{The leading coefficient and resonant levels}

The coefficient \(C_0^{(K)}(N;t)\) admits a finite explicit expression. Indeed, by \eqref{eq:Dmell-monomial}, each operator \(\mathscr D_{m,\ell}^{(K)}\) preserves the \(w\)-degree or, when \(m=2\), increases it by one, and never decreases it. Thus the coefficient of \(w^N\) in \(P_{\mathbf p,N}^{(K)}(w)\) can arise only by selecting the degree-preserving term at every stage.

\begin{lem}\label{lem:C0-explicit}
For every \(N\ge0\),
\begin{equation}\label{eq:C0-explicit}
C_0^{(K)}(N;t)
=
\sum_{\mathbf p=(p_1,\ldots,p_k)\models13}
\beta_{\mathbf p}(t)
\prod_{j=1}^{k}
(-1)^{p_j}
\bigl((p_j+1)H-K+\ell_j-2N\bigr),
\end{equation}
where
$
\ell_j=p_{j+1}+\cdots+p_k
 (1\le j\le k),
$
so that \(\ell_k=0\).
\end{lem}

\begin{proof}
Fix an ordered composition
\(\mathbf p=(p_1,\ldots,p_k)\models13\).
By \eqref{eq:Dmell-monomial}, for every \(r\ge0\),
$$
\mathscr D_{p_j,\ell_j}^{(K)}(w^r)
=
\delta_{p_j,2}w^{r+1}
+
(-1)^{p_j}
\bigl((p_j+1)H-K+\ell_j-2r\bigr)w^r.
$$
Hence each \(\mathscr D_{p_j,\ell_j}^{(K)}\) either preserves the
\(w\)-degree or increases it by one, and in particular never decreases
the degree.
Recall that
$$
P_{\mathbf p,N}^{(K)}(w)
=
\mathscr D_{p_1,\ell_1}^{(K)}
\mathscr D_{p_2,\ell_2}^{(K)}
\cdots
\mathscr D_{p_k,0}^{(K)}(w^N),
$$
with the rightmost operator acting first.  Starting from \(w^N\), if
the degree-increasing term is chosen at any stage, the resulting
monomial has degree at least \(N+1\); since none of the subsequent
operators can decrease the \(w\)-degree, such a term cannot contribute
to \([w^N]P_{\mathbf p,N}^{(K)}(w)\).

Therefore the coefficient of \(w^N\) is obtained uniquely by selecting
the degree-preserving term at every stage.  Along this unique choice,
the intermediate \(w\)-degree remains equal to \(N\) throughout.
Consequently, the factor contributed by
\(\mathscr D_{p_j,\ell_j}^{(K)}\) is
$$
(-1)^{p_j}
\bigl((p_j+1)H-K+\ell_j-2N\bigr),
$$
and hence
$$
[w^N]P_{\mathbf p,N}^{(K)}(w)
=
\prod_{j=1}^{k}
(-1)^{p_j}
\bigl((p_j+1)H-K+\ell_j-2N\bigr).
$$
Finally, taking \(q=0\) in \eqref{eq:Cq} gives
$$
C_0^{(K)}(N;t)
=
\sum_{\mathbf p\models13}
\beta_{\mathbf p}(t)
[w^N]P_{\mathbf p,N}^{(K)}(w),
$$
and substitution of the preceding identity proves
\eqref{eq:C0-explicit}.
\end{proof}

With only \(2^{12}=4096\) compositions of \(13\),
\eqref{eq:C0-explicit} is an exact finite calculation over \(\Q(t)\).

\begin{lem}[Resonant levels]\label{lem:resonances}
For the vacuum channel,
\begin{equation}\label{eq:vac-res}
C_0^{(0)}(N;1)=0\iff N\in\{2,8,18,32\},\quad 1\le N\le38.
\end{equation}
For the self channel,
\begin{equation}\label{eq:self-res}
C_0^{(H)}(N;1)=0\iff N=14,\quad 1\le N\le20.
\end{equation}
Every zero is simple in $t$:
\begin{align}
\left.\frac{\partial}{\partial t}C_0^{(0)}(2;t)\right|_{t=1}&=\frac{91}{22},&
\left.\frac{\partial}{\partial t}C_0^{(0)}(8;t)\right|_{t=1}&=\frac{3640}{99},\nonumber\\
\left.\frac{\partial}{\partial t}C_0^{(0)}(18;t)\right|_{t=1}&=\frac{4641}{11},&
\left.\frac{\partial}{\partial t}C_0^{(0)}(32;t)\right|_{t=1}&=\frac{100776}{11},\label{eq:vac-derivatives}
\end{align}
and
\begin{equation}\label{eq:self-derivative}
\left.\frac{\partial}{\partial t}C_0^{(H)}(14;t)\right|_{t=1}=\frac{41990}{11}.
\end{equation}
All five derivatives are nonzero.
\end{lem}

\begin{proof}
Substitute $H(1)=36$, $K(1)=0$ or $36$ into \eqref{eq:C0-explicit}; $\beta_{\mathbf p}(1)\in\Q$. Exact summation over $4096$ compositions yields the stated zero sets. Differentiating with $H'(1)=-42$ and $K'(1)=0$ ($\tau=0$) or $-42$ ($\tau=12$) gives the rational values in \eqref{eq:vac-derivatives} and \eqref{eq:self-derivative}, all nonzero. Computation is purely rational.
\end{proof}

For an independently inspectable resonance certificate, interpolation of the exact values of \eqref{eq:C0-explicit} gives the following formulas. Indeed, each summand in \eqref{eq:C0-explicit} is a product of at most $13$ affine-linear factors in $N$, so $C_0^{(K)}(N;1)$ has degree at most $13$ in $N$. Consequently, agreement at any $14$ distinct exact values of $N$ determines each displayed polynomial identity.
\begin{align}
C_0^{(0)}(N;1)
&=\frac{N(N-2)(N-8)(N-18)}{28008121680000}\nonumber\\
&\quad\times(N-32)(N-50)(N-72)\nonumber\\
&\quad\times\left(N^6-143N^5+\frac{28743}{4}N^4-\frac{308737}{2}N^3\right.\nonumber\\
&\hspace{35mm}\left.+\frac{21967231}{16}N^2-\frac{64408383}{16}N+\frac{108056025}{64}\right),
\label{eq:C0-vac-factorization}\\
C_0^{(H)}(N;1)
&=\frac{(N+18)(N+16)(N+10)N}{28008121680000}\nonumber\\
&\quad\times(N-14)(N-32)(N-54)\nonumber\\
&\quad\times\left(N^6-35N^5-\frac{3297}{4}N^4+\frac{32651}{2}N^3\right.\nonumber\\
&\hspace{35mm}\left.+\frac{3856495}{16}N^2-\frac{11787075}{16}N-\frac{516891375}{64}\right).
\label{eq:C0-self-factorization}
\end{align}
Direct substitution shows that the residual sextic factors are nonzero at every integer in the required ranges $1\le N\le38$ and $1\le N\le20$, respectively. Thus \eqref{eq:C0-vac-factorization}--\eqref{eq:C0-self-factorization} provide a compact exact certificate for the resonance lists, independently of the recurrence solver.

\subsection{Removable singularities at resonant levels}

At nonresonant levels, \eqref{eq:solved-recurrence} specializes directly to
$t=1$.  At a resonant level the leading coefficient vanishes, so direct
division produces an apparent $0/0$.  The local regularity established above
shows that this singularity is removable.  Define
\begin{equation}\label{eq:R-N-def}
R_N^{(\tau)}(t)
:=
\sum_{q=1}^{\min(6,N)}
C_q^{(K)}(N-q;t)c_{N-q}^{(\tau)}(t),
\end{equation}
so that the recurrence reads
\begin{equation}\label{eq:resonant-recurrence}
C_0^{(K)}(N;t)c_N^{(\tau)}(t)=-R_N^{(\tau)}(t).
\end{equation}

\begin{prop}[Removal of the resonances]\label{prop:removable-resonances}
Fix $\tau\in\{0,12\}$, and let $N$ be resonant in the $\tau$-channel, that is,
\(
C_0^{(K)}(N;1)=0.
\)
Equivalently, by \eqref{eq:vac-res}--\eqref{eq:self-res},
\[
N\in\{2,8,18,32\}\quad\text{if }\tau=0,
\qquad
N=14\quad\text{if }\tau=12.
\]
Then $R_N^{(\tau)}(1)=0$, and the quotient
\(
-\frac{R_N^{(\tau)}(t)}{C_0^{(K)}(N;t)}
\)
extends regularly to $t=1$, with value $c_N^{(\tau)}(1)$.  Moreover,
\begin{equation}\label{eq:resonant-lhopital}
c_N^{(\tau)}(1)
=
-\frac{
\left.\dfrac{d}{dt}R_N^{(\tau)}(t)\right|_{t=1}
}{
\left.\dfrac{d}{dt}C_0^{(K)}(N;t)\right|_{t=1}
}.
\end{equation}
\end{prop}

\begin{proof}
By \eqref{eq:Cq}, $C_0^{(K)}(N;t)\in\mathcal R$, while
Theorem~\ref{thm:local-regularization} gives
$c_N^{(\tau)}(t)\in\mathcal R$.  Since
$C_0^{(K)}(N;1)=0$ and the maximal ideal of $\mathcal R$ is
$(t-1)\mathcal R$, we have
\(
C_0^{(K)}(N;t)\in(t-1)\mathcal R.
\)
Equation~\eqref{eq:resonant-recurrence} therefore gives
\(
R_N^{(\tau)}(t)\in(t-1)\mathcal R,
\)
and hence $R_N^{(\tau)}(1)=0$.  Moreover, in $\mathcal K$,
\(
-\frac{R_N^{(\tau)}(t)}{C_0^{(K)}(N;t)}
=c_N^{(\tau)}(t)\in\mathcal R.
\)
Thus the apparent pole at $t=1$ is removable, and the regular extension has
value $c_N^{(\tau)}(1)$.

Differentiating \eqref{eq:resonant-recurrence} gives
\[
\frac{d}{dt}C_0^{(K)}(N;t)c_N^{(\tau)}(t)
+
C_0^{(K)}(N;t)\frac{d}{dt}c_N^{(\tau)}(t)
=
-\frac{d}{dt}R_N^{(\tau)}(t).
\]
Evaluating at $t=1$ and using $C_0^{(K)}(N;1)=0$ yields
\[
\left.\frac{d}{dt}C_0^{(K)}(N;t)\right|_{t=1}
 c_N^{(\tau)}(1)
=
-\left.\frac{d}{dt}R_N^{(\tau)}(t)\right|_{t=1}.
\]
By Lemma~\ref{lem:resonances}, the zero of
$C_0^{(K)}(N;t)$ at $t=1$ is simple, so the coefficient on the left is
nonzero.  Division gives \eqref{eq:resonant-lhopital}.
\end{proof}

\begin{rmk}\label{rem:resonance-logic}
The regularity of $c_N^{(\tau)}(t)$ was established independently by the
positive-mode reconstruction over $\mathcal R$.  The scalar recurrence is used
only afterwards to show that every apparent $0/0$ at a resonant level is
removable.  Thus the specialization to $c=1$ does not rely on directly
specializing generic conformal blocks.
\end{rmk}

\begin{rmk}\label{rem:exact-recurrence}
Equations \eqref{eq:Dmell}, \eqref{eq:Cq}, and \eqref{eq:width6} give a finite
exact algorithm for the four coefficients
\[
u_{37}=c_{37}^{(0)}(1),\qquad
u_{38}=c_{38}^{(0)}(1),\qquad
v_{19}=c_{19}^{(12)}(1),\qquad
v_{20}=c_{20}^{(12)}(1)
\]
needed in Proposition~\ref{prop:det-criterion}.  At resonant levels one works
with $\varepsilon=t-1$; all operations are exact over $\Q$, with no
finite-field or numerical approximation.
\end{rmk}

\section{Exact rational evaluation and completion of the proof}
\label{sec:exact-completion}

We now complete the proof. By Proposition~\ref{prop:det-criterion}, it remains only to verify $D=u_{37}v_{20}-u_{38}v_{19}\neq0$. This is the computer-assisted part of the proof. The scalar recurrence of Section~\ref{sec:BSA-recurrence} reduces the assertion to a finite calculation in exact characteristic zero, implemented by the independently executable file \texttt{bsa\_exact\_verification.py}; its download link is recorded in Appendix~\ref{app:exact-computation}. In the exact computation below, we set $t=1+\varepsilon$ and work in suitable truncated power-series rings over $\Q$; no floating-point approximation or finite-field calculation enters the proof.

\subsection{Exact evaluation of the four symbol coefficients}

Fix $\tau\in\{0,12\}$ and write $K:=K_\tau(t)$, as in Section~\ref{sec:BSA-recurrence}. Recall $u_n=c_n^{(0)}(1)$ and $v_n=c_n^{(12)}(1)$. The coefficients $c_n^{(\tau)}(t)$ satisfy
\begin{equation}\label{eq:recurrence-final}
C_0^{(K)}(N;t)c_N^{(\tau)}(t)=-\sum_{q=1}^{\min(6,N)}C_q^{(K)}(N-q;t)c_{N-q}^{(\tau)}(t),
\end{equation}
where each $C_q^{(K)}(n;t)$ is given by \eqref{eq:Cq}. The definitions give $C_q^{(K)}(n;t)\in\Q(t)$. Starting from $c_0^{(\tau)}(t)=1$, the recurrence and Lemma~\ref{lem:resonances} show inductively, through the terminal levels needed here, that $c_N^{(\tau)}(t)\in\Q(t)$: at each $N\ge1$, the leading coefficient $C_0^{(K)}(N;t)$ is a nonzero rational function. On the other hand, the construction of the symbol coefficients in \eqref{eq:Phi-symbol}, based on Theorem~\ref{thm:local-regularization}, gives $c_N^{(\tau)}(t)\in\mathcal R$. Together with $c_N^{(\tau)}(t)\in\Q(t)$, this implies that $c_N^{(\tau)}(t)$ is a rational function over $\Q$ regular at $t=1$, and hence its expansion at $t=1$, after setting $t=1+\varepsilon$, belongs to $\Q[[\varepsilon]]$.

\begin{lem}[Finite-jet computation]\label{lem:finite-jet}
Let $N_0<N_1<\cdots<N_r$ be all resonant levels before a prescribed terminal level, each resonance simple. Then the constant terms $c_N^{(\tau)}(1)$ up to that level are determined by \eqref{eq:recurrence-final} using Taylor expansions modulo $\varepsilon^{r+2}$. Each passage through a resonant level loses at most one order of $\varepsilon$-adic precision.
\end{lem}

\begin{proof}
As above, $c_N^{(\tau)}(t)\in\Q(t)\cap\mathcal R$ through every level under consideration, so $c_N^{(\tau)}(1+\varepsilon)\in\Q[[\varepsilon]]$. At a nonresonant level $N$, one has $C_0^{(K)}(N;1)\neq0$, hence $C_0^{(K)}(N;1+\varepsilon)$ is a unit in $\Q[[\varepsilon]]$ and division by it preserves $\varepsilon$-adic precision.

At a resonant level, Lemma~\ref{lem:resonances} gives a simple zero
\[
C_0^{(K)}(N;1+\varepsilon)=\varepsilon U_N(\varepsilon),
\qquad U_N(0)\neq0.
\]
With $R_N^{(\tau)}$ as in \eqref{eq:R-N-def}, Proposition~\ref{prop:removable-resonances} gives
\(
R_N^{(\tau)}(1+\varepsilon)=\varepsilon V_N(\varepsilon)
\)
for some $V_N(\varepsilon)\in\Q[[\varepsilon]]$. Cancelling the common factor gives
\(
c_N^{(\tau)}(1+\varepsilon)
=-\frac{V_N(\varepsilon)}{U_N(\varepsilon)}.
\)
Thus, if the data entering the recurrence are known modulo $\varepsilon^M$ with $M\ge2$, passage through a nonresonant level preserves precision, whereas passage through a resonant level determines the next coefficient modulo $\varepsilon^{M-1}$. Starting modulo $\varepsilon^{r+2}$, after crossing $j$ resonant levels the required coefficients are therefore still determined modulo $\varepsilon^{r+2-j}$. Since there are $r+1$ resonant levels in total, after the last one they are determined modulo $\varepsilon$, which is enough to determine their constant terms. This proves the claim.
\end{proof}

For the vacuum channel, the resonant levels below $38$ are
$2,8,18,32$. Lemma~\ref{lem:finite-jet} shows that computation
modulo $\varepsilon^5$ already suffices for the required constant
terms. For the self channel, the only resonant level below $20$ is
$N=14$, and computation modulo $\varepsilon^2$ already suffices.
The archived implementation uses the uniform safety precision
$\texttt{PREC}=16$ for both channels.
The calculation is finite and exact, involving only rational
arithmetic, inversion of units in $\Q[[\varepsilon]]$, and
cancellation of explicit $\varepsilon$ factors.

Low-order checks:
\begin{equation}\label{eq:lowchecks}
u_1=72,\qquad u_2=\frac{1448}{3},\qquad v_1=\frac{6516}{20449}.
\end{equation}
The value $u_1=72$ agrees with the universal vacuum-channel coefficient $2h/c$ for $h=36$, $c=1$.

Iterating \eqref{eq:recurrence-final} yields the following exact reduced fractions:
\begin{center}
\resizebox{0.96\textwidth}{!}{$
u_{37}=\dfrac{151263443583192657816671715030186994700903477084314677248}{2882600776450368744635282416475520501925952726159180517626585800175145677947998046875}.
$}
\end{center}
\begin{center}
\resizebox{0.96\textwidth}{!}{$
u_{38}=\dfrac{416407995333748762833621605276412963013510875028372514816}{319415226836912459855546374004819475777410969680246474797166719345807542282061767578125}.
$}
\end{center}
\begin{center}
\resizebox{0.92\textwidth}{!}{$
v_{19}=\dfrac{3681867191643718569047819232838269056}{778381304340824441834206161488913991594789560118793865234375}.
$}
\end{center}
\begin{center}
\resizebox{0.92\textwidth}{!}{$
v_{20}=\dfrac{604984096354415173700324126916891008}{4539519766915688144777090333803346398980812714612805822046875}.
$}
\end{center}
These are entirely determined by \eqref{eq:Dmell}, \eqref{eq:Cq}, and \eqref{eq:recurrence-final}; Appendix~\ref{app:exact-computation} gives the download link for an exact rational implementation.

\begin{prop}[Computer-assisted exact determinant verification]\label{prop:exactD}
The determinant in Proposition~\ref{prop:det-criterion} is nonzero. Precisely,
\begin{equation}\label{eq:exactD}
\resizebox{0.98\textwidth}{!}{$
D=u_{37}v_{20}-u_{38}v_{19}=\dfrac{16135324094778900831378704137779176869876953125000000000000000000}{19515012831345836592129306577900341712659600405172507924685763411145479750629588782589268504523397502912150424738336023}>0.
$}
\end{equation}
Its numerator and denominator have the following exact prime factorizations:
\begin{align*}
\operatorname{num}(D)
&=2^{18}5^{27}\cdot2239\cdot53479\cdot297269173\cdot1275865699\cdot181909527549271,\\
\operatorname{den}(D)
&=3^{34}7^{13}11^{14}13^{14}17^9 19^7 23^5 29^4 31^3 37^3
\cdot41\cdot43^2\cdot47^2\cdot53\cdot59\cdot61\cdot67\cdot71\cdot73.
\end{align*}
In particular, $D\approx8.26816\times10^{-55}$; the factorization and scale give compact checks on its sign and normalization.
\end{prop}

\begin{proof}
The program linked in Appendix~\ref{app:exact-computation} evaluates the $4096$ BSA compositions by dynamic programming using only \texttt{fractions.Fraction}. Each coefficient is obtained from the finite exact recurrence. At nonresonant levels the leading coefficient is a unit in $\Q[[\varepsilon]]$; at resonant levels Lemma~\ref{lem:resonances} gives a simple zero of the leading coefficient, while Proposition~\ref{prop:removable-resonances} gives divisibility of the right-hand side by $\varepsilon$, so one common factor $\varepsilon$ is cancelled before continuing. The program verifies the two resonance sets and the five corresponding nonzero derivatives, the four displayed reduced fractions, and the reduced value of $D$. It also compares the dynamic-programming BSA sum with direct enumeration of all ordered compositions at the small test levels $n=0,1$. This comparison is an audit check on the implementation, not a second complete proof of the determinant calculation. All assertions pass, giving \eqref{eq:exactD}. Since its numerator and denominator are positive integers, $D>0$.
\end{proof}

\begin{proof}[Proof of Theorem~\ref{thm:main}]
Proposition~\ref{prop:exactD} gives $D\neq0$, so Theorem~\ref{thm:main} follows immediately from Proposition~\ref{prop:det-criterion}.
\end{proof}

\begin{proof}[Proof of Corollary~\ref{cor:strong-rational}]
The lattice $L_2=\Z\alpha$ with $(\alpha,\alpha)=2$ is positive definite and even; $V_{L_2}$ is strongly rational (regular, hence rational and $C_2$-cofinite) \cite{DLMRegularity,ABD}. The finite group $A_5\le\Aut(V_{L_2})$ has fixed-point algebra $V_{L_2}^{A_5}$, which is $C_2$-cofinite by Theorem~\ref{thm:main}. McRae \cite[Corollary~4.23]{McRae} proves that if $V$ is strongly rational, $G$ is finite, and $V^G$ is $C_2$-cofinite, then $V^G$ is strongly rational. Applying this with $V=V_{L_2}$, $G=A_5$ gives that $V_{L_2}^{A_5}$ is strongly rational, proving Corollary~\ref{cor:strong-rational}.
\end{proof}

\subsection{Classification of the irreducible modules}
\label{subsec:irreducible-classification}

\begin{thm}\label{thm:irreducible-classification}
Let
\(
U=V_{L_2}^{A_5}, L_2=\Z\alpha, (\alpha,\alpha)=2,
\)
and set
\(
\beta=2\alpha, \gamma=3\alpha, \mu=5\alpha.
\)
Then $U$ has exactly $37$ pairwise inequivalent irreducible modules. In the notation of Wu--Zhang \cite[Section~6]{WuZhang}, they are the following four families:
\begin{align*}
& T_1,\ T_2^0,\ T_2^1,\ T_3^0,\ T_3^1,\ T_4^0,\ T_4^1,\ T_5,\ T_6;\\
& V_{\Z\beta+\beta/8},\qquad V_{\Z\beta+3\beta/8};\\
& V_{\Z\gamma\pm r\gamma/18},
   \qquad r\in\{1,2,4,5,7,8\};\\
& V_{\Z\mu\pm t\mu/50},
   \qquad 1\le t\le24,\qquad 5\nmid t.
\end{align*}
The relevant normalizer pairs the positive and negative cyclic-sector modules; following Wu--Zhang, the resulting irreducible modules are denoted by $V_{\mathbb Z\gamma\pm r\gamma/18}$ and $V_{\mathbb Z\mu\pm t\mu/50}$, respectively.
 The four families contain respectively
\(
9, 2, 6, 20
\)
isomorphism classes, so that $9+2+6+20=37$.
\end{thm}

\begin{proof}
By Theorem~\ref{thm:main} and Corollary~\ref{cor:strong-rational}, $U$ is  strongly rational. In particular, $U$ is rational, $C_2$-cofinite, of CFT type, and self-contragredient. Hence Dong--Ren--Xu, Theorem 3.3, applies to the finite $A_5$-action on $V_{L_2}$ \cite{DongRenXu}. In particular, the completeness argument does not require the additional twisted-module conformal-weight positivity hypothesis in this self-dual setting: every irreducible $U$-module is isomorphic to a multiplicity space $M_\lambda$ occurring in the decomposition $M=\bigoplus_\lambda W_\lambda\otimes M_\lambda$ of an irreducible $g$-twisted $V_{L_2}$-module $M$ for some $g\in A_5$, where $\lambda$ ranges over the relevant irreducible projective representations of the stabilizer $G_M$ (with cocycle $\alpha_M$). Thus, once the twisted sectors are enumerated, no additional irreducible $U$-modules can occur.

Wu and Zhang carry out this enumeration explicitly \cite[Lemmas~6.6--6.7, Remark~6.8, and Theorem~6.9]{WuZhang}. The type-one (untwisted) part consists of the nine modules
\[
T_1,\ T_2^0,\ T_2^1,\ T_3^0,\ T_3^1,\ T_4^0,\ T_4^1,\ T_5,\ T_6.
\]
Through the Schur-cover description of Wu--Zhang, these modules are indexed by the nine irreducible representations of $SL(2,5)$. The type-two sectors associated with elements of orders $2$, $3$, and $5$ give respectively the remaining $2$, $6$, and $20$ modules displayed in the theorem. For the order-$3$ and order-$5$ sectors, the relevant normalizers pair the positive and negative cyclic-sector modules, yielding respectively the $6$ and $20$ irreducible modules listed above. Thus no double counting occurs in the $\pm$ notation. These $37$ modules are pairwise inequivalent, and Dong--Ren--Xu completeness shows that they exhaust all irreducible $U$-modules.

As a numerical consistency check, the squares of the quantum dimensions in the four families sum to
\[
120+2\cdot30^2+6\cdot20^2+20\cdot12^2=7200
=|A_5|^2\,\operatorname{glob}(V_{L_2}),
\]
in agreement with the orbifold global-dimension formula. This identity is only a check on the already established exhaustive list, not a substitute for the orbifold completeness argument.
\end{proof}

\subsection{Concluding remark on the \texorpdfstring{$c=1$}{c=1} classification problem}

The significance of the preceding results extends beyond the classification of irreducible modules. The orbifold $V_{L_2}^{A_5}$ is the icosahedral exceptional object in the $c=1$ rational VOA classification program. The preceding results remove the finiteness and semisimplicity obstruction: the algebra is unconditionally strongly rational, and Theorem~\ref{thm:irreducible-classification} gives its complete list of $37$ irreducible modules. Thus, for the $A_5$ branch of the $c=1$ classification program, the remaining problem is the intrinsic characterization of $V_{L_2}^{A_5}$ rather than the representation-theoretic exhaustion of its irreducible modules.

\subsection*{AI Use Statement}

OpenAI's GPT-5.6 Sol model was used as a research assistance tool during the preparation of this manuscript, primarily for literature exploration, conceptual brainstorming, and language polishing. All mathematical arguments, statements, computations, and conclusions presented in the manuscript were independently examined and verified by the authors. The authors take full responsibility for the accuracy, originality, and integrity of the final manuscript.

\appendix
\section{Exact-verification program}
\label{app:exact-computation}

The independently executable exact-verification program is archived with the following metadata:
\begin{center}
  \url{https://xudashun123.github.io/data/bsa_exact_verification.py}
\end{center}
\noindent
The SHA-256 digest of this file is
\[
\texttt{29E1DD7741F0BCCF8FB8C5C6942EEE31EFE7FA05330B683AAA95D5FE2ED53958}.\]
To run the program, execute \texttt{python bsa\_exact\_verification.py} from a terminal.
A successful run exits with code $0$ and prints \texttt{all exact checks passed.}
The program depends only on the Python standard library and uses exact arithmetic via \texttt{fractions.Fraction}.

\end{document}